\documentclass[11pt]{article}

\usepackage[dvips]{graphics}
\usepackage{amsmath,amsthm,amsopn,amssymb,epsfig,psfrag,color}
\usepackage{multicol}
\usepackage{amsfonts}
\usepackage[english]{babel}
\usepackage{multirow}
\usepackage{epstopdf}
\usepackage[title,titletoc,toc]{appendix}
\usepackage[right, pagewise]{lineno}
\usepackage{graphicx}
\usepackage{subfigure}
\usepackage[numbers, sort&compress]{natbib}

 \newcommand{\D}{\displaystyle}

\newtheorem{Theorem}{Theorem}[section]
\newtheorem{Lemma}[Theorem]{Lemma}
\newtheorem{Proposition}[Theorem]{Proposition}
\newtheorem{Corollary}[Theorem]{Corollary}

\newtheorem{Definition}[Theorem]{Definition}

\numberwithin{equation}{section} \numberwithin{figure}{section}

\def\rr{\mathbb{R}}

\def\D{\displaystyle}
\begin{document}

\title{Population dynamics in river networks}
\author{Yu Jin\footnote{
Department of Mathematics,
University of Nebraska-Lincoln,
Lincoln, NE, 68588, USA. Email: yjin6@unl.edu}, \,\,
Rui Peng\footnote{Department of Mathematics, Jiangsu Normal University, Xuzhou, 221116, Jiangsu Province, China.
Email: pengrui\_seu@163.com},\,\,
Junping Shi\footnote{Department of Mathematics, College of William and Mary,
Williamsburg, VA 23187-8795, USA. Email: jxshix@wm.edu
}
}
\date{}
\maketitle

\begin{abstract}
Natural rivers connect to each other to form networks. The geometric structure of a river network can significantly influence spatial dynamics of populations in the system. We consider a process-oriented model to describe population dynamics in river networks of trees, establish the fundamental theories of the corresponding parabolic problems and elliptic problems, derive the persistence threshold by using the principal eigenvalue of the  eigenvalue problem, and define the net reproductive rate to describe population persistence or extinction. By virtue of numerical simulations, we investigate the effects of hydrological, physical, and biological factors, especially the structure of the river network, on population persistence.
\end{abstract}

\section{Introduction}

River and stream ecosystems is a key component of the global environmental ecosystems, and it  has gained increasing attention of ecologists and environmental scientists. The organisms living in river systems are subjected to the biased flow in the downstream direction. Instream flow needs (IFNs) and ``drift paradox" are two important problems for stream ecologists and river managers. The former asks how much stream flow can be changed while still maintaining an intact
stream ecology \cite{Anderson2006,McKenzieetal2012}, and the latter asks how stream-dwelling organisms can persist
without being washed out when continuously subjected to a unidirectional water flow \cite{Muller1954,Muller1982,PLNL,Speirs2001}. The problems are challenging due to the complex and dynamic nature of interactions between the stream environment and the biological community. The study of population models in rivers or streams reveals the dependence of spatial population dynamics on environmental and biological factors in rivers or streams,  hence it has become an important explanation of the Instream flow needs and ``drift paradox"; see e.g.,
 \cite{AndersonNisbet2006,AndersonNisbet2005,DeAngelis1995,Hilkerlewis2010,lutscherlewis2006,
Lutscher2007,Lutscher2009,Lutscherseo,McKenzieetal2012,PLNL,Speirs2001,
VLutscher2010,Vasilyeva20122,Jinlewisjmb,Jin2011,lutscher2008,Lutscher2010,Lutcherlewis2005,SeoLutscher2011,JJL}.

In mathematical models, rivers and streams have been traditionally treated as  a finite or infinite length one-dimensional interval on the real line (see e.g., \cite{Speirs2001,McKenzieetal2012,Jinlewisjmb,Jin2011,Lutcherlewis2005}).
When homogeneous river intervals (see e.g., \cite{Speirs2001,Hilkerlewis2010}) are recognized as oversimplification of real river systems, more realistic rivers have been approximated by alternating good and bad patches or pool-and-riffle channels
(see e.g, \cite{lutscherlewis2006,JinLewisBMB}), drift and benthic (or storage)  zones (see e.g., \cite{Huangetal2016,PLNL,Jinwangfb}), or meandering rivers consisting of a main channel and point bars (see e.g., \cite{JinLutscher,JinLewisBMB}). These later generalizations are still in one-dimensional space or considered as one channel.

Nevertheless, natural river systems are often in a spatial network structure such as dendritic trees. The network topology (i.e., the topological structure of a network), together with other physical and hydrological features in a river network,
can greatly influence the spatial distribution of the flow profile (including the flow velocity, water depth, etc.) in branches of the network.  Moreover, the population dispersal vectors may be constrained by the network configuration and the flow profile, and species life history traits may depend on varying habitat conditions in the network. As a result, population distribution and persistence in river systems can be significantly affected by the network topology or structure, see e.g.,
\cite{Cuddington2002,Fagan2002,Fausch2002,Goldberg2010,Grant2007,Grant2010,Peterson2013}.
Then there arise interesting questions such as whether a population can persist in the desert streams of the southwestern United State while the streams are experiencing substantial natural drying trends \cite{Fagan2002}, or whether dendritic geometry enhances dynamic stability of ecological systems \cite{Grant2007} etc. Furthermore, other related dynamics in the network, such as the dynamics of water-born infectious diseases like cholera may also be greatly affected by the river network geometry (see e.g., \cite{Bertuzzo2010}).

Branches in a river network have been modeled as point nodes in a network of habitats  in individual based models \cite{Fagan2002,Grant2010} and matrix population models for stage-structured populations \cite{Goldberg2010,Mari2014}.
However this oversimplifies the spatial heterogeneity of river networks. In a real river ecosystem, organisms mainly live in the branches of the network and the connections between branches (e.g., the network nodes) are mainly for population transitions from one branch to another. To take into account this realistic situation, in recent works \cite{Ramirez2012,Sarhad2015,Sarhad2014},  integro-differential equations and reaction-diffusion-advection equations were used to model population dynamics in  river networks where the network branches, instead of the network nodes, are the main habitat for organisms. Here the river networks are modeled under the framework of metric graphs (or metric networks). A metric graph is a graph $G=(V,E)$ with a set $V$ of vertices and a set $E$ of edges, such that each edge $e\in E$ is associated with either a closed bounded interval. Mathematic notion of metric graphs was first introduced in the context of wave propagation on thin graph-like domains \cite{BK2013,k2005}, and they are also called quantum graphs.

The theories of parabolic and elliptic equations as well as the corresponding eigenvalue problems on metric networks are important in establishing population dynamics of biological species in river networks. The existence and uniqueness of solutions of linear parabolic equations and nonlinear parabolic equations have been established in \cite{vonbelow1988} and \cite{vonbelow1994}, respectively. A maximum principle for semilinear parabolic network equations was obtained in \cite{vonbelow1991}, and the eigenvalue problems associated with parabolic equations in networks were studied in \cite{vonbeloweigen,vonbelow2012}. Stability of steady states of parabolic equations were studied in \cite{vonbelow2015ss,Yanagida}. More studies of diffusion equations in networks can be found in e.g., \cite{WArendt,vonbelow1996par,Lumer1979,Lumer1980}. In these work, the model parameters are allowed to be time and/or space dependent; the so-called Kirchhoff laws or an excitatoric Kirchhoff condition (or dynamical node condition) are assumed at the interior connecting points.

The goal of this paper is to establish a mathematically rigorous foundation of reaction-diffusion-advection equations defined on a metric tree network, which models population dynamics of a biological species on a river network. The population model consists of reaction-diffusion-advection equations describing population dynamics on network branches and zero population flux at interior connecting points in the network, allowing variations of diffusion rates, advection rates, and growth rates throughout the network.  We will rigorously derive the theories for the time-dependent parabolic equations, the corresponding elliptic equations for the steady states, and the associated eigenvalue problems for linearized equations, to establish population persistence conditions in terms of the principal eigenvalue and/or the net reproductive rate of the system. We will also prove the existence and uniqueness of a stable positive steady state when population persists under the logistic type growth rate. The theory of infinite-dimensional dynamical systems and existing theories of parabolic and elliptic equations as well as eigenvalue problems on a real line and on  metric networks will be applied. We will also study how different factors influence population dynamics, especially persistence and the distribution of the stable positive steady state (if exists) in the whole network.

The population persistence in a spatial population model has been described by uniform persistence, the (in)stability of the
trivial extinction solution, and  the critical domain size (minimal length of the habitat such that a species can persist); see e.g.,  \cite{Jin2011,Huangetal2016,LLL2015JBD,LLL2016SIAM,LL2014JMB,Lutcherlewis2005,McKenzieetal2012,PLNL}. For a single species population in one-dimensional rivers, the persistence theory was established in a homogeneous environment in \cite{Speirs2001,VLutscher2010,Lutcherlewis2005,lutscher2008}, in temporally periodically varying environments in \cite{Jin2011},  in temporally randomly varying environments in \cite{JJL}, and in spatially heterogeneous environment in \cite{Lutscher2010}. For a benthic-drift population consisting of individuals drifting in water and individuals staying on the benthos, the critical domain size was studied in a spatially homogeneous river in \cite{PLNL} and in a river with alternating good and bad channels in \cite{lutscherlewis2006}. In particular, persistence metrics (fundamental niche, source/sink metric, and the net reproductive rate) have been established  for a single stage population in \cite{McKenzieetal2012} and for a benthic-drift population in \cite{Huangetal2016}, respectively.  Population persistence for a single species in river networks has also been studied in \cite{Ramirez2012,Sarhad2015,Sarhad2014}.  Integro-differential equations were used to describe population dynamics in river networks in  \cite{Ramirez2012}, where the diffusion coefficients, advection rates and growth rates were assumed to be the same in all branches, and the population persistence was determined by the stability of the extinction state.  Reaction-diffusion-advection equations were used in \cite{Sarhad2015,Sarhad2014}, where again constant diffusion coefficients, advection rates, and growth rates were assumed throughout the network but zero-flux interior junction conditions were not assumed. The principal eigenvalue of the corresponding eigenvalue problem was used to determine the stability of the extinction state and also the population persistence. Most of the analyses and results about persistence conditions were restricted to radial trees, in which all branches on the same level are essentially the same habitats, and  hence population dynamics in such networks is essentially equivalent to that in a one-dimensional river.

This paper is organized as follows. In section \ref{modelsection}, we introduce the notion of river network of a general tree and the initial boundary value problem for population dynamics on the network.  In section \ref{eigensection}, we establish the existence of a principal eigenvalue of the corresponding eigenvalue problem and we show that it determines whether the population persists (the extinction solution is unstable) or becomes  extinct (the extinction solution is globally asymptotically stable). Moreover, we obtain the existence of a globally asymptotically stable positive steady state when the population persists. In section \ref{netR0section}, we define the next generation operator and the net reproductive rate $R_0$ of the population living in the river network, and we prove that $R_0=1$ can be used as a persistence threshold for the population. We also provide a method to calculate $R_0$. In section \ref{simulationsection}, by virtue of numerical simulations, we study the influences of hydrological, physical, and biological factors on the net reproductive rate as well as
the positive steady state. In Appendix \ref{appendixtheories}, we provide the derivation of the theories for the parabolic and elliptic problems on networks, including the maximal principle, the comparison principle, and the existence, uniqueness and estimations of the solutions.

\section{Model}\label{modelsection}

\subsection{The river network - a metric tree}

In this work, we assume that the river network is a finite metric tree, \textit{i.e.}, a connected finite metric graph
with no cycles, or equivalently, a finite metric graph on which any two vertices can be connected by a unique simple path.

We first introduce the mathematical definition of a river network  (a finite tree) and notations on it
(see e.g.,  \cite{vonbelow1994}). Let $G$ be a $C^\kappa$-network for $\kappa\geq 2$ with the set of vertices
 $$
 E=\{e_i: 1\leq i\leq N\},
 $$
the set of edges
 $$
 K=\{k_j:j\in I_{N-1}\},\ \ n\geq2,
 $$
and arc length parameterization $\pi_j\in C^\kappa([0,l_j],\mathbb{R}^2)$ on edge $k_j$,  where $N$ and $N-1$ are the numbers of vertices
and edges, respectively, and
$$
I_{N-1}=\{1,2,\cdots, N-1\}.
$$
The edge $k_j$ is isomorphic to the interval $[0,l_j]$ with length $l_j$ and spatial variable $x_j$ on it, where $x_j=0$ and $x_j=l_j$ represent the upstream end and the downstream end of $k_j$, respectively.  The topological graph $\Lambda=(E,K)$ embedded in $G$ is
assumed to be simple and connected. Thus, $\Lambda$ admits the following properties: each $k_j$ has its endpoints in $E$, any two vertices in $E$
can be connected by a unique simple path with arcs in $K$, and any two distinct edges $k_j$ and $ k_h $ intersect at no more than one point in $E$. See Figure \ref{riverexamplefig} for an example of a river network.

Endowed with the above graph topology and metric defined on each edge, $G$ is a connected and compact subset of $\rr^2$.
The orientation of $G$ is given by the incidence
matrix $(d_{ij})_{N\times (N-1)} $ with
 \begin{equation}
 \label{directiond}
 d_{ij}=\left\{
 \begin{array}{ll}
 1\, &\mbox{ if } \pi_j(l_j)=e_i \  \mbox{(i.e., $e_i$  is the downstream end of the edge $k_j$)},\\
 -1\, &\mbox{ if } \pi_j(0)=e_i \  \mbox{(i.e., $e_i$  is the upstream end of the edge $k_j$)},\\
 0\, &\mbox{ otherwise}  \  \mbox{(i.e., $e_i$  is not a vertex on the edge $k_j$)}.
 \end{array}
 \right.
 \end{equation}
We distinguish the set $E$ of vertices as follows:
 $$
 \begin{array}{l}
 E_r=\{e_i\in E: \gamma_i>1\} \mbox{ (ramification (or interior junction) vertices)}, \\
 E_b=\{e_i\in E: \gamma_i=1\}  \mbox{ (boundary vertices)},\\
 E_u=\{e_i\in E: \gamma_i=1, e_i \mbox{ is an upstream boundary vertex}\} ,\\
 E_d=\{e_i\in E: \gamma_i=1, e_i \mbox{ is a downstream boundary vertex}\},
 \end{array}
 $$
where $\gamma_i=\gamma(e_i)$ is the valency of $e_i$ that represents the number of edges that connect to $e_i$,  and $E_b=E_u\cup E_d$.

\begin{figure}[t!]
\centering
\includegraphics[height=3in, width=5.5in]{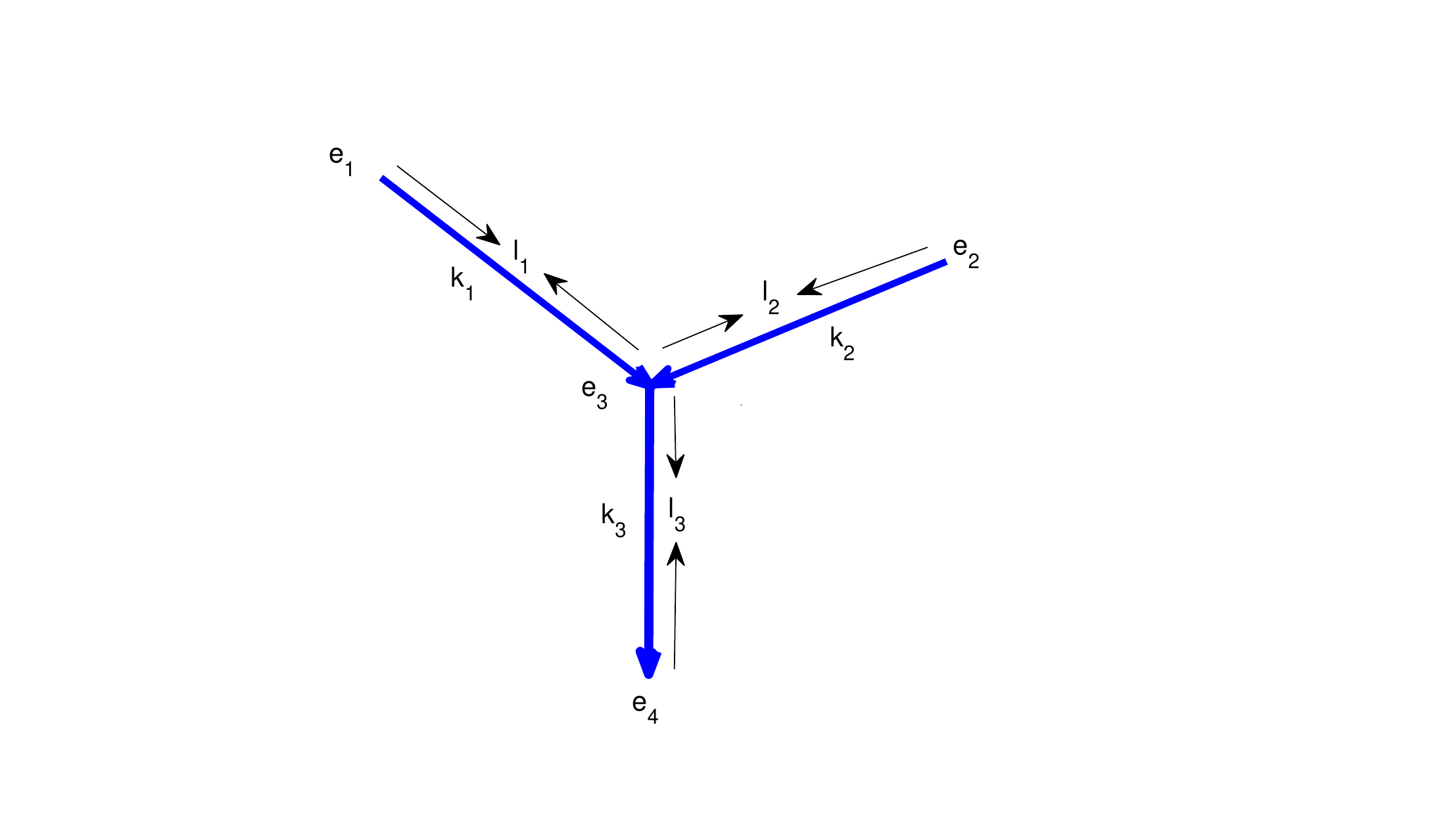}
\vspace{-1cm}

\caption{A river network. Each blue arrow represents a river branch with the specific water flow direction.
 } \label{riverexamplefig}
\end{figure}

 Let $t$ be the time variable and for $ T>0$, denote
 $$
 \begin{array}{l}
 \Omega=G\times [0,T],\quad
 \Omega_j=[0,l_j]\times [0,T],
 \end{array}
 $$
 $$
 \begin{array}{l}
 \Omega_p=(G\setminus E_b)\times (0,T],\quad
 \omega_p=(G\times\{0\})\cup (E_b\times(0,T]).
 \end{array}
 $$
For a function $u:\Omega\rightarrow \mathbb{R}$, we define $u_j=u\circ (\pi_j,id):\Omega_j\rightarrow \mathbb{R}$. Differentiation is carried out on each edge $k_j$ with respect to the arc length
parameter $x_j$.
A function is differentiable on $G$ means that it is differentiable at all points  $x\in G\setminus E$.  We use the following notations for functions and derivatives at a vertex
 $$
 u_j(e_i,t)=u_j(\pi_j^{-1}(e_i),t),\ \ u_{x_j}(e_i)=\frac{\partial }{\partial x_j}u_j(\pi_j^{-1}(e_i),t),\ \ u_{x_jx_j}(e_i)
 =\frac{\partial^2 }{\partial x_j^2}u_j(\pi_j^{-1}(e_i),t).
 $$
Any function $u$ on $\Omega$ satisfies $u_j(e_i,t)=u_h(e_i,t) \mbox{ if } k_j\cap k_h=\{e_i\}.$

We now introduce function spaces on $G$. Let
$$
C(G)=\{g: g_j\in C([0,l_j],\rr), \, j\in I_{N-1}\}
$$ with the norm:
 \[{\|g\|}_{C(G)} =\max\limits_{j\in I_{N-1}} \max\limits_{x\in[0,l_j]}|g_j(x)|.\]
The Banach space $C^m(G)$ consists of all functions that are $m$ times continuously differentiable over $G$  with norm given by
\[{\|g\|}_{C^m(G)}=\sum_{\beta=1}^m\|g^{(\beta)}\|_{C(G)}+{\|g\|}_{C(G)}, \]
where $g^{(\beta)}$ is the $\beta$-th derivative of $g$.
Similarly $L^p(G)$ is the Banach space of all real-valued functions defined on
$G$ that are measurable and $p$-summable with respect to
$G$ with $p\geq 1$. The norm in this space is defined by
\[\|g\|_{L^p(G)}=\sum_{j=1}^{N-1}\Big(\int _0^{l_j}|g_j|^p\Big)^{1/p}. \]

For $\alpha\in [0,1)$, define
 $$
 C^{2+\alpha,1+\frac{\alpha}{2}}(\Omega)=\{u\in C(\Omega):\ \ u_j\in
 C^{2+\alpha,1+\frac{\alpha}{2}}(\Omega_j),\ \ \forall
 j\in I_{N-1}\}
 $$
where $C^{2+\alpha,1+\frac{\alpha}{2}}(\Omega_j)$ with the usual norm
$\|\cdot\|_{C^{2+\alpha,1+\frac{\alpha}{2}}(\Omega_j)}$ denotes the Banach space of
functions $u$ on $\Omega_j$ having continuous derivatives
$\frac{\partial^{r+s} u}{\partial t^r \partial x_j^s}$ for $2r+s\leq
2$ and finite  H\"{o}lder constraints of the indicated exponents in the
case of $\alpha>0$. Then $C^{2+\alpha,1+\frac{\alpha}{2}}(\Omega)$ is a
Banach space endowed with the norm
$$\|u\|_{C^{2+\alpha,1+\frac{\alpha}{2}}(\Omega)}=\sum_{j=1}^{N-1}\|u_j\|_{C^{2+\alpha,1+\frac{\alpha}{2}}(\Omega_j)}.$$
Similarly we can define $C^{2+\alpha}(G)$, $W^{2}_p(G)$ and $W^{2,1}_p(\Omega)$ for any fixed $\alpha\in [0,1)$ and $p\geq 1$.

\subsection{The population model in the river network}

Since Speirs and Gurney's work \cite{Speirs2001}, the dynamics of a
population living in a one-dimensional river has been described by
the following reaction-diffusion-advection equation:
 \begin{equation}\label{modelsingleriver}
 \frac{\partial u}{\partial t}=D\frac{\partial^2 u}{\partial x^2}-v\frac{\partial u}{\partial x}+f(x,u)u,
\end{equation}
where $u(x,t)$ is the population density at location $x$ and time $t$, $D$
is the diffusion coefficient, $v$ is the flow velocity, and $f$ is the per
capita growth rate.

We adapt  model (\ref{modelsingleriver}) to a population living in a river network $G$. The dynamics of the population can be described by
 \begin{equation}\label{model}
 \frac{\partial u_j}{\partial t}=D_j\frac{\partial^2 u_j}{\partial x_j^2}-v_j\frac{\partial u_j}{\partial x_j}+f_j(x_j,u_j)u_j,
 \,\, x_j\in (0,l_j),\,\,j\in I_{N-1}, \,\,t>0,
 \end{equation}
where $u_j$ is the population density on the edge $k_j$, $D_j$ is the diffusion coefficient on $k_j$, $v_j$
is the flow velocity on $k_j$, and $f_j$ is the per capita growth rate on $k_j$.
The initial population distribution in $G$ is $u^0$, that is,
 \begin{equation}\label{modelic}
 u_j(x_j,0)=u_j^0(x_j),\, x_j\in [0,l_j],\,\, j\in I_{N-1}.
 \end{equation}

There are three types of vertices
in the river network $G$: upstream boundary ends,
downstream boundary ends, and interior junction vertices. Correspondingly, boundary or interface conditions are imposed at each vertex of $E$.
 \begin{itemize}
 \item At an upstream boundary point $e_i\in E_u$ that only connects
 to the edge $k_j$, the boundary condition can be assumed as
 \begin{equation}\label{bcupstreamzf}
 \alpha_{j,1}u_j(e_i,t)-\beta_{j,1}\frac{\partial u_j}{\partial x_j}(e_i,t)=0 \,\,\mbox{ with } \alpha_{j,1}\geq 0,\,\,
 \beta_{j,1}\geq 0,\,\, \alpha_{j,1}+\beta_{j,1}> 0,
 \end{equation}
for instance,
 \begin{equation}\label{bczfex}
\mbox{the zero-flux boundary condition: } \quad \left(D_j\frac{\partial u_j}{\partial x_j}-v_j u_j\right)(e_i,t)=0. \,
 \end{equation}

 \item At a downstream boundary point $e_i\in E_d$ that only connects
to the edge $k_j$, the boundary condition can be assumed as

 \begin{equation}\label{bcdownstreamzf}
 \alpha_{j,2}u_j(e_i,t)+\beta_{j,2}\frac{\partial u_j}{\partial x_j}(e_i,t)=0 \,\,
 \mbox{ with } \alpha_{j,2}\geq 0,\,\, \beta_{j,2}\geq 0,\,\, \alpha_{j,2}+\beta_{j,2}> 0,
 \end{equation}
for instance,
 \begin{equation}\label{bcff}
 \mbox{ the free flow (or Neumann) condition: }\quad \frac{\partial u_j}{\partial x_j}(e_i,t)=0 \,\mbox{ or }
 \end{equation}
 \begin{equation}\label{bchostile}
 \mbox{ the hostile (or Dirichlet) condition:  }\quad u_j(e_i,t)=0.
 \end{equation}

  \item At an interior junction point $e_i\in E_r$, the population density is
   continuous and the total population flux in and out is zero. Hence,
   the connection conditions are the continuity conditions and Kirchhoff laws:
 \begin{subequations}\label{interconds}
\begin{equation}\label{1-interconds}
  u_{i_1}(e_i,t)=u_{i_2}(e_i,t)=\cdots=u_{i_m}(e_i,t),
\end{equation}
\begin{equation}\label{2-interconds}
   \sum_{j=i_1}^{i_m} d_{ij}A_j D_j\frac{\partial u_j}{\partial x_j}(e_i,t)=0,
\end{equation}
\end{subequations}
where $e_i\in E_r$ connects to edges $k_{i_1}$, $k_{i_2}$, $\cdots$, and $k_{i_m}$, $d_{ij}$ is defined in (\ref{directiond}),
 $A_j$ is the wetted cross-sectional area of the edge $k_j$,
 and (\ref{2-interconds}) is the result of substituting the continuity condition (\ref{1-interconds}) and the conservation of the flow at $e_i$
 \begin{equation}\label{conservationflow}
 \sum_{j=i_1}^{i_m} d_{ij}A_jv_j=0,
 \end{equation}
 into the zero-flux condition at $e_i$
 \begin{equation}\label{intercond2}
 \sum_{j=i_1}^{i_m} d_{ij}A_j\left(D_j\frac{\partial u_j}{\partial x_j}-v_j u_j \right)(e_i,t)=0.
 \end{equation}

 \end{itemize}

According to different ecological conditions at boundary vertices on $G$, we further use the following notations throughout the paper:
 $$
 \begin{array}{l}
 E_0=\{e_i\in E_b:
 \mbox{ vertices with hostile (or Dirichlet) condition}\},\\
 E_b\setminus E_0: \mbox{ vertices not assigned with hostile condition}.
 \end{array}
 $$

We finally define an initial boundary value problem for a population  in a river network:
 $$
 \noindent
 \hspace{-1.5cm}{\bf (IBVP) \hspace{3cm}(\ref{model}),\, (\ref{modelic}),\, (\ref{bcupstreamzf}),\, (\ref{bcdownstreamzf})\,  \mbox{
 and } (\ref{interconds})}.
 $$

Furthermore, we impose the following assumptions in different parts of the paper.
\begin{enumerate}
 \item [{\bf [H1]}] For each $j\in I_{N-1}$, $D_j>0$, $v_j\geq 0$, $A_j>0$.

  \item   [{\bf [H2]}] For each $j\in I_{N-1}$, $f_j: [0,l_j]\times[0,\infty)\rightarrow \mathbb{R}$ is continuous and there exists a constant $M_j>0$ such that $f_j(x,u)\leq 0$ for any $x\in [0,l_j]$ and $u\ge M_j$,
and $f_j(\cdot,u_j)u_j$ is Lipschitz continuous in $u_j$ with Lipschitz constant $L_j>0$.
 \item [{\bf [H3]}]  For each $j\in I_{N-1}$, $f_j(\cdot,u_j)$
is monotonically decreasing in $u_j$.
 \end{enumerate}

By adapting theories for parabolic and elliptic equations on intervals and/or networks \cite{DMugnolo,WArendt,Fijavz,Protterbook1967,vonbelow1988, vonbelow1991,vonbelow1994,LSU,So,Pao,RDbook}, we develop the fundamental theories of parabolic and elliptic problems on networks corresponding to {\bf (IBVP)}; see Appendix \ref{appendixtheories}.
In particular, for  linear parabolic problems, we establish the strong maximum principle (in Lemma \ref{maxiprinciple}), Hopf boundary lemma for networks (in Lemma \ref{signderivative}), comparison principle (in Lemma \ref{compprinciple}), and the existence, uniqueness, $L^p$ and Schauder estimates of solutions (in Theorem \ref{linearexistence}, via writing the differential operator into a self-adjoint operator on the network); for the nonlinear problem {\bf (IBVP)}, we develop the theory of the existence, uniqueness and positivity of solutions (in Theorem \ref{globalexistence}, by using the upper and lower solutions) and prove the monotonicity and strict subhomogeneity of the solution map (in Lemmas \ref{Qtmonotone} and \ref{Qtstrsubhomo}, respectively); for the corresponding elliptic problems, we also obtain the strong maximum principle (in Lemma \ref{maxiprinciple-elliptic}), Hopf boundary lemma (in Lemma \ref{signderivative-elliptic}), comparison principle (in Lemma \ref{positivesolell}), and the existence, uniqueness, $L^p$ and Schauder estimates of solutions  (in Theorems \ref{modelsslinsolex} and \ref{ellipticeqnsolexis}). These mathematical preparations enable us to establish the extinction/persistence criteria for system {\bf (IBVP)}.

\section{The eigenvalue problem and population persistence}\label{eigensection}

In this section, we consider the eigenvalue problem corresponding to the linearized system of {\bf (IBVP)} at the trivial solution, obtain the existence of the principal eigenvalue,  and then use the principal eigenvalue as a threshold for population persistence and extinction. We also obtain the existence, uniqueness and stability of a positive steady state when the population persists. Assumptions {\bf [H1]}-{\bf [H3]} are all imposed throughout the rest of the paper.

\subsection{The eigenvalue problem and its principal eigenvalue}

We first introduce some Banach spaces which will be used frequently later. Denote
 \begin{equation}
 \label{functionspace1}
 X=\{\varphi\in C^1(G):\ \,\varphi \mbox{ satisfies (\ref{bcupstreamzf}) and (\ref{bcdownstreamzf})}\},
 \end{equation}
and let
 \begin{equation}
 \label{functionspace1cone}
 X_+=\{\varphi\in X:\ \ \varphi\geq 0,\ \varphi\not\equiv 0\}
 \end{equation}
be the positive cone in $X$. The interior of $X_+$ is
 \begin{equation}
 \label{functionspace1int}
 X^o=\{\varphi\in X:\ \ \varphi>0 \mbox{ on } G\setminus E_0, \mbox{ and }  d_{ij}\varphi_{x_j}(e_i)<0 \mbox{ if }  e_i\in E_0 \}.
 \end{equation}
Then $X_+$ is a solid cone of $X$ with nonempty interior $X^o$. We also write
$
 \varphi_1\gg\varphi_2\ \  \mbox{if}\ \varphi_1-\varphi_2\in X^o.
$

The linearization of {\bf (IBVP)} at the trivial solution $u=0$ is
\begin{equation}\label{modellin}
 \begin{cases}
 \D
\frac{\partial u_j}{\partial t}=D_j\frac{\partial^2 u_j}{\partial x_j^2}-v_j\frac{\partial u_j}{\partial x_j}+f_j(x_j,0)u_j,
 & x_j\in (0,l_j),\,j\in I_{N-1}, \,t>0,\\
 \displaystyle
  (\ref{modelic}),\,  (\ref{bcupstreamzf}),\, (\ref{bcdownstreamzf}),\, \mbox{
and } (\ref{interconds}).&
 \end{cases}
\end{equation}
Substituting  $u_j(x_j,t)=e^{\lambda t}\psi_j(x_j)$ into
(\ref{modellin}), we obtain the corresponding eigenvalue
problem
 \begin{equation}\label{modellineigen10}
 \begin{cases}
 \D
 \lambda\psi_j(x_j)=D_j\frac{\partial^2 \psi_j}{\partial x_j^2}-v_j\frac{\partial \psi_j}{\partial x_j}+f_j(x_j,0)\psi_j, & x_j\in (0,l_j),\,j\in I_{N-1}, \\
 \displaystyle
  \alpha_{j,1}\psi_j(e_i)-\beta_{j,1}\frac{\partial \psi_j}{\partial x_j}(e_i)= 0, & \forall e_i\in E_u,\\
 \displaystyle
 \alpha_{j,2}\psi_j(e_i)+\beta_{j,2}\frac{\partial \psi_j}{\partial x_j}(e_i)= 0,& \forall e_i\in E_d,  \\
 \displaystyle
 \psi_{i_1}(e_i)=\cdots=\psi_{i_m}(e_i),\,\,\,
 \sum\limits_{j=i_1}^{i_m} d_{ij}A_jD_j\frac{\partial \psi_j}{\partial x_j} (e_i)=0, &
 \forall e_i\in E_r.
 \end{cases}
 \end{equation}

For simplicity, denote $\mathcal{L}$ to be the operator such that $\mathcal{L}|_{k_j}=\mathcal{L}_j$, where
\begin{equation}\label{Ljdefinition}
\mathcal{L}_j=D_j\frac{\partial^2 }{\partial x_j^2}-v_j\frac{\partial}{\partial x_j}+f_j(\cdot,0).
\end{equation}

The following result indicates that (\ref{modellineigen10}) admits a
simple eigenvalue associated with a positive
eigenfunction. The proof is given in Appendix \ref{proofpropeigenvalueexth}.
\begin{Proposition}\label{eigenvalueexth}
The eigenvalue problem (\ref{modellineigen10}) admits a simple
eigenvalue $\lambda^\ast$ associated with a positive
eigenfunction $\psi^\ast\in X^o$. None of the other eigenvalues of (\ref{modellineigen10})
corresponds to a positive eigenfunction; and if $\lambda\neq
\lambda^\ast$ is an eigenvalue of (\ref{modellineigen10}), then
$\mathrm{Re}(\lambda)\leq \lambda^\ast.$
\end{Proposition}

Let $\lambda^\ast$ be the eigenvalue of the eigenvalue
problem (\ref{modellineigen10}) with a corresponding positive eigenfunction
$\psi^\ast\in X^o$. We call $\lambda^\ast$ the {\it principal eigenvalue } of (\ref{modellineigen10}).

We say that $\mathcal{L}$ has the {\it strong maximum principle property} if $u\in C^2(G)$ satisfying
 \begin{equation}\label{modelsslintt-aa}
 \begin{cases}
 \D
  -\mathcal{L}_{j} u_j(x_j)\geq0, & x_j\in (0,l_j),\,j\in I_{N-1},\\
 \displaystyle
 \alpha_{j,1}u_j(e_i)-\beta_{j,1}\frac{\partial u_j}{\partial x_j}(e_i)\geq0,& \forall e_i\in E_u,\\
 \displaystyle
\alpha_{j,2}u_j(e_i)+\beta_{j,2}\frac{\partial u_j}{\partial x_j}(e_i)\geq0, & \forall e_i\in E_d,  \\
 \displaystyle
  u_{i_1}(e_i)=\cdots=u_{i_m}(e_i),\,\,
 \sum_{j=i_1}^{i_m} d_{ij}A_jD_j\frac{\partial u_j}{\partial x_j}(e_i)\geq0, &\forall e_i\in E_r
 \end{cases}
 \end{equation}
implies that $u>0$ in $G\backslash E_0$ unless $u\equiv0$. We also say $u\in C^2(G)$ is an upper solution of $\mathcal{L}$ if
\eqref{modelsslintt-aa} holds, and such $u$ is called as a strict upper solution of $\mathcal{L}$ if it is an upper solution but
is not a solution.  Then the analysis of \cite[Theorem 2.4]{Dubook} can be easily adapted to conclude the following result.

\begin{Proposition}\label{equivalent-property} The following statements are equivalent.
 \begin{enumerate}
 \item[{\rm(i)}] $\mathcal{L}$ has the strong maximum principle property;

 \item[{\rm(ii)}] $\mathcal{L}$ has a strict upper solution which is positive in $G\backslash E_0$;

 \item[{\rm(iii)}] $\lambda^\ast<0$.

 \end{enumerate}
\end{Proposition}

\subsection{Persistence and extinction}

We now use the sign of the principal eigenvalue of (\ref{modellineigen10}) to
determine the population persistence or extinction as well as the existence of a positive steady state for
{\bf (IBVP)}, which satisfies the following elliptic equations:
 \begin{equation}\label{modelss}
 \begin{cases}
 \D
 -D_j\frac{\partial^2 u_j}{\partial x_j^2}+v_j\frac{\partial u_j}{\partial x_j}=f_j(x_j,u_j)u_j, &x_j\in (0,l_j),\,j\in I_{N-1}, \\
 \displaystyle
 \alpha_{j,1}u_j(e_i)-\beta_{j,1}\frac{\partial u_j}{\partial x_j}(e_i)= 0, & \forall e_i\in E_u,\\
 \displaystyle
 \alpha_{j,2}u_j(e_i)+\beta_{j,2}\frac{\partial u_j}{\partial x_j}(e_i)= 0, &  \forall e_i\in E_d,\\
 \displaystyle
   u_{i_1}(e_i)=\cdots=u_{i_m}(e_i),\ \
\sum_{j=i_1}^{i_m} d_{ij}A_jD_j\frac{\partial u_j}{\partial
x_j}(e_i)= 0,& \forall e_i\in E_r.
  \end{cases}
\end{equation}

The following result shows that the principal eigenvalue $\lambda^\ast$ is the key threshold of extinction/persistence for {\bf (IBVP)}. The proof is given in Appendix \ref{proofThpersistenceresult}.

\begin{Theorem}\label{persistenceresult}
Let $\lambda^\ast$ be the principal eigenvalue of the eigenvalue
problem (\ref{modellineigen10}) with corresponding eigenfunction
$\psi^\ast\in X^o$. Then
 \begin{enumerate}
 \item[{\rm(i)}] If $\lambda^\ast\leq0$, then $u\equiv 0$ is globally attractive for {\bf (IBVP)}
 for all initial values in $X_+$.

 \item[{\rm(ii)}] If $\lambda^\ast> 0$, then {\bf (IBVP)} admits a unique positive steady state $u^\ast\in X^o$ which is globally
attractive for all initial values in $X_+\setminus \{0\}$.
\end{enumerate}
\end{Theorem}

The above theorem indicates that $\lambda^\ast=0$ is a
threshold to determine population persistence on a river network.
The idea has been used in \cite{Ramirez2012, Sarhad2015,Sarhad2014}
but this is the first time that it is rigorously proved for population
models on river networks.

\section{The net reproductive rate $\mathcal{R}_0$}\label{netR0section}

The net reproductive rate has been defined and proved to be a threshold quantity for population persistence in a single river channel \cite{McKenzieetal2012,Huangetal2016}.
In this section, we will define the next generation operator and the net reproductive rate $\mathcal{R}_0$ for
{\bf (IBVP)} and then use $\mathcal{R}_0$ to determine the population persistence and extinction.
Moreover, we will provide a numerical method to calculate $\mathcal{R}_0$.

\subsection{Definition of the net reproductive rate $\mathcal{R}_0$}

Assume the growth rate of the population on edge $k_j$ satisfies
$f_j(x_j,u_j)u_j=\tilde{f}_j(x_j,u_j)u_j-m_j(x_j)u_j$, where $\tilde{f}_j$
is the recruitment rate and $m_j(x_j)$ is the mortality rate. Let
$r_j(x_j)=\tilde{f}_j(x_j,0)$ and assume $r, m\in C(G)$. Then
 $$
 \frac{\partial(f_j(\cdot,u_j)u_j)}{\partial u_j}(x_j,0)=f_j(x_j,0)=r_j(x_j)-m_j(x_j).
 $$

For $\phi^0\in X$, assume that $\phi$ satisfies
 \begin{equation}\label{modellinphi}
 \begin{cases}
 \D
 \frac{\partial \phi_j}{\partial t}=D_j\frac{\partial^2 \phi_j}{\partial x_j^2}
 -v_j\frac{\partial \phi_j}{\partial x_j}-m_j(x_j)\phi_j, & x_j\in (0,l_j),\,j\in I_{N-1}, \,t>0,\\
 \displaystyle
  \phi_j(x_j,0)=\phi_j^0(x_j),& x_j\in [0,l_j],\\
 \displaystyle
 \phi \mbox{ satisfies } (\ref{bcupstreamzf}),\, (\ref{bcdownstreamzf}),\, \mbox{
 and } (\ref{interconds}).&
 \end{cases}
 \end{equation}
Define $\Gamma:X\rightarrow X$  by
 $$
 [\Gamma(\phi^0)]_j(x_j)=\int_0^\infty
 r_j(x_j)\phi_j(x_j,t)dt,\,\,x_j\in[0,l_j],\,\,j\in I_{N-1},
 $$
where $\phi$ is the solution of (\ref{modellinphi}) with initial
condition $\phi^0$. That is, $\Gamma$ is a linear operator mapping an initial distribution
of the population to its offspring distribution. Hence, we call
$\Gamma$ the {\it next generation operator}. Let
 $$
 \mathcal{R}_0=r(\Gamma),
 $$
where $r(\Gamma)$ is the spectral radius of the linear operator $\Gamma$ on $X$.
Then $\mathcal{R}_0$ represents the average number of offsprings that an individual
produces during its lifetime and we call $\mathcal{R}_0$
the {\it net reproductive rate}.

Let $B:\mathcal{D}\rightarrow \mathcal{D}$ with
 \begin{equation}
 \label{Adomaindefinition}
 \mathcal{D}=\{\varphi\in C^2(G\setminus E_b)\cap C^1(G):\ \ \varphi \mbox{ satisfies (\ref{bcupstreamzf}),\,
 (\ref{bcdownstreamzf}),\, (\ref{interconds})}\},
 \end{equation}
be defined by
 $$
 \begin{array}{l}
 B_j=\D D_j\frac{\partial^2 }{\partial x_j^2}-v_j\frac{\partial }{\partial x_j}-m_j(x_j), \,\,j\in I_{N-1}. \\
 \end{array}
 $$
Let $\Gamma_1: X\rightarrow X$ be such that
 $$
 [\Gamma_1(\phi^0)]_j(x_j)=\int_0^\infty \phi_j(x_j,t)dt,\,\,x_j\in[0,l_j],\,\,j\in I_{N-1},
 $$
where $\phi$ is the solution of (\ref{modellinphi}) with initial condition $\phi^0$.
Similarly as in Proposition 2.10 of \cite{McKenzieetal2012}, we can prove that $-\Gamma_1$ is the inverse operator of $B$,
\textit{i.e.}, $B^{-1}=-\Gamma_1$. Hence, $\Gamma (\phi)=-QB^{-1}(\phi)$, where the operator $Q$ is defined as
 $$
 [Q(\phi)]_j(x_j)=r_j(x_j)\phi_j(x_j),\,\,\forall x_j\in[0,l_j],\,\,
 j\in I_{N-1}, \forall \phi\in X.
 $$
Then $\mathcal{L}=B+Q$. By Proposition \ref{eigenvalueexth} and the above analysis,
noting that $(\lambda I-B)^{-1}$ is defined such that
 $$
 [(\lambda I-B)^{-1}(\phi^0)]_j(x_j)=\int_0^\infty e^{-\lambda t}\phi_j(x_j,t)dt,\,\,x_j\in[0,l_j],\,\,j\in I_{N-1},
 $$
where $\phi$ is the solution of (\ref{modellinphi}) with initial condition $\phi^0$,
we know that both $\mathcal{L}$ and $B$ are resolvent-positive operators in $X$. It follows from Propositions \ref{eigenvalueexth} and \ref{equivalent-property}
that the spectral bound $s(B)$ of $B$ is the principal eigenvalue of $B$ and $s(B)<0$. We then obtain the following result
by using \cite[Theorem 3.5]{Thieme2010}.

\begin{Lemma}\label{signr0lambda}
$\mathcal{R}_0-1$ and $\lambda^\ast$ have the same sign, where $\lambda^\ast$ is the principal eigenvalue of the eigenvalue
problem (\ref{modellineigen10}).
\end{Lemma}

Theorem \ref{persistenceresult} and Lemma \ref{signr0lambda} imply the following result.

\begin{Corollary}\label{r0persistence}
If $\mathcal{R}_0\leq 1$, then $u\equiv 0$ is globally attractive for {\bf (IBVP)}; if $\mathcal{R}_0>1$,
then {\bf (IBVP)} admits a unique positive steady state $u^\ast\in X^o$, which is globally
attractive for all initial values in $X_+\setminus \{0\}$.
\end{Corollary}

Therefore, $\mathcal{R}_0=1$ is the threshold for population persistence and extinction.
The population will be extinct if $\mathcal{R}_0\leq 1$ and it is persistent if $\mathcal{R}_0>1$.

\subsection{Calculation of $\mathcal{R}_0$}

 Let
 $$
 \overline{B}_j=D_j\frac{\partial^2 }{\partial x_j^2}-v_j\frac{\partial}{\partial x_j}, \,\, j\in I_{N-1}.
 $$
Integrating (\ref{modellinphi}) with respect to $t$ from $0$ to $\infty$ yields
 $$
 \begin{array}{l}
\D \int_0^\infty\frac{\partial \phi_j}{\partial t}dt=\int_0^\infty\left[D_j\frac{\partial^2 \phi_j}{\partial
 x_j^2}-v_j\frac{\partial \phi_j}{\partial x_j}-m_j(x_j)\phi_j\right]dt,\,j\in I_{N-1},  \,t>0.
 \end{array}
 $$
Note that $\phi_j(\cdot,t)\rightarrow 0$ as $t\rightarrow\infty$. The above equation implies that
 \begin{equation}\label{gamma1equation0}
 \begin{cases}
 \D
 -\phi_j^0(x_j)=\overline{B}_j[(\Gamma_1(\phi^0))_j]-m_j(x_j)(\Gamma_1(\phi^0))_j], & j\in I_{N-1},  \,t>0,\\
 \displaystyle
\Gamma_1(\phi^0) \mbox{ satisfies }  (\ref{bcupstreamzf}),\,
(\ref{bcdownstreamzf}),\, \mbox{
 and } (\ref{interconds}). &
 \end{cases}
 \end{equation}
Therefore, $\Gamma_1(\phi^0)$ is the solution of
 \begin{equation}\label{gamma1equation}
  \begin{cases}
 \D
 -\phi_j^0(x_j)=\overline{B}_ju_j(x_j)-m_j(x_j)u_j(x_j), & x_j\in (0,l_j), \,j\in I_{N-1},\,t>0,\\
 \displaystyle
  u \mbox{ satisfies }  (\ref{bcupstreamzf}),\, (\ref{bcdownstreamzf}),\, \mbox{
 and } (\ref{interconds}).&
 \end{cases}
  \end{equation}
We define
 $$
 T_1: X\rightarrow X,\,\, u=T_1\phi^0,
 $$
where $u$ is the solution of (\ref{gamma1equation}). Then $T_1$ is a compact and strongly positive operator on
$X$. By the definition of $T_1$ and equations (\ref{gamma1equation0}) and (\ref{gamma1equation}), we know that
$\Gamma_1=T_1$ on $X$. Hence, $\Gamma=Q\Gamma_1$ is also a compact and strongly positive operator on $X$.
It follows from  \cite[Theorem 1.2]{Dubook} that $\mathcal{R}_0=r(\Gamma)>0$ is a simple eigenvalue of $\Gamma$ with an
eigenfunction $\phi^\ast\in X^o$, \textit{i.e.},
$$
\Gamma \phi^\ast=\mathcal{R}_0 \phi^\ast,
$$
and there is no other eigenvalues of $\Gamma$ associated
with positive eigenfunctions.

By following the idea in the proof of \cite[Theorem 3.2]{Wang2012}, we can obtain $\mathcal{R}_0$ via the principal eigenvalue of another eigenvalue problem.
\begin{Theorem}\label{R0calwzidea}
If the eigenvalue problem
 \begin{equation}\label{modellineigen22}
  \begin{cases}
 \D
 \mu r_j(x_j)\psi_j(x_j)=-D_j\frac{\partial^2 \psi_j}{\partial x_j^2}+v_j\frac{\partial \psi_j}{\partial x_j}+m_j(x_j)\psi_j,
 &x_j\in (0,l_j),\,j\in I_{N-1},  \\
 \displaystyle
 \alpha_{j,1}\psi_j(e_i)-\beta_{j,1}\frac{\partial \psi_j}{\partial x_j}(e_i)= 0, &\forall e_i\in E_u, \\
 \displaystyle
 \alpha_{j,2}\psi_j(e_i)+\beta_{j,2}\frac{\partial \psi_j}{\partial x_j}(e_i)= 0,&  \forall e_i\in E_d, \\
 \displaystyle
  \psi_{i_1}(e_i)=\cdots=\psi_{i_m}(e_i),\ \
 \sum\limits_{j=i_1}^{i_m} d_{ij}A_jD_j\frac{\partial \psi_j}{\partial x_j} (e_i)=0,&\forall e_i\in E_r,
 \end{cases}
 \end{equation}
admits a unique positive eigenvalue $\mu^0$ with a positive eigenfunction, then $\mathcal{R}_0=1/\mu^0$.
\end{Theorem}

To numerically calculate $\mathcal{R}_0$, we use the finite difference method
to discretize (\ref{modellineigen22}) and approximate
(\ref{modellineigen22}) by
$$
\mu \hat{r}\psi=\hat{A}\psi,
$$
where $\hat{r}$ is a diagonal matrix containing values of $r(x_j)$ on the main diagonal and $\hat{A}$ is the discretization of the operator on
the right-hand side of (\ref{modellineigen22}). The matrix $\hat{A}^{-1}\hat{r}$ is a non-negative and irreducible matrix. The  Perron-Frobenius Theorem implies that $\hat{A}^{-1}\hat{r}$ admits a
principal eigenvalue $\varsigma^\ast$, which is the unique simple eigenvalue of $\hat{A}^{-1}\hat{r}$ associated
with a
positive eigenvector $\psi^*$, that is,
\begin{equation}\label{r0approxmationmat}
\hat{A}^{-1}\hat{r}\psi^\ast=\varsigma^\ast\psi^\ast.
\end{equation}
Then we approximate
$1/\mu^0$ by  $\varsigma^\ast$, i.e.,
\begin{equation}\label{r0approxmationmat2}
\mathcal{R}_0 =\frac{1}{\mu^0}\approx\varsigma^\ast
\end{equation}
 by using Theorem \ref{R0calwzidea} and above approximating scheme.

It follows from (\ref{r0approxmationmat}) that $\hat{r}\psi^\ast$ is the eigenvector of $\hat{r}\hat{A}^{-1}$ corresponding to $\varsigma^\ast$, i.e., $\hat{r}\hat{A}^{-1}(\hat{r}\psi^\ast)=\varsigma^\ast(\hat{r}\psi^\ast)$. Note that
the next generation operator $\Gamma $ can be approximated by
$\hat{r}\hat{A}^{-1}$ and that $\mathcal{R}_0=1/\mu^0$.
 The eigenvalue problem $\Gamma \phi^\ast=\mathcal{R}_0 \phi^\ast$ can be approximated by
 $$
  \hat{r}\hat{A}^{-1}\phi^\ast\approx\frac{1}{\mu^0}\phi^\ast\approx\varsigma^\ast\phi^\ast.
  $$
 Then we obtain
  $$
\phi^\ast\approx\hat{r}\psi^\ast.
$$
That is, $\hat{r}\psi^\ast$
can be used to approximate the
eigenfunction of $\Gamma$ associated with the eigenvalue $\mathcal{R}_0$.
We call $
\phi^\ast$ (or $\hat{r}\psi^\ast$ as the approximation) the {\it next generation
distribution} of the population.

\section{The influences of factors on population persistence}\label{simulationsection}

The results in Sections \ref{eigensection} and \ref{netR0section} show that both  the principal eigenvalue $\lambda^\ast$ of the eigenvalue problem (\ref{modellineigen10}) and the net reproductive rate $\mathcal{R}_0$ can be used to determine the population persistence. For the biological significance of $\mathcal{R}_0$, now we apply the theory in Section \ref{netR0section} to investigate the influences of biotic and abiotic factors on population dynamics (in particular, persistence or extinction) of {\bf (IBVP)} via numerical studies
of $\mathcal{R}_0$ and the stable positive steady state (if exists).

Real river networks are complex and the quantitative influence of a factor on the population persistence highly depends on the structure and scales of the network. While the choice of a general river network is random, we consider a few simple but typical river networks of trees with one, three, four, five, and seven branches, representing different types of network topologies, merging from the upstream or splitting into the downstream; see Figure \ref{rivershapesfig}. In particular, river networks (3-a), (3-b), (7-a) and (7-b) are radial trees, which are
rooted trees with all tree features, including edge lengths, sectional
areas, and boundary conditions depending only on the distance to the root (see Section of \cite{Sarhad2014}).

 Three sets of boundary conditions for {\bf (IBVP)} are considered:
\begin{enumerate}
  \item [(ZF-FF)] Zero-flux condition (\ref{bczfex}) at the upstream and free flow (i.e., Neumann) condition (\ref{bcff}) at the downstream;
  \item [(ZF-H)] Zero-flux condition (\ref{bczfex}) at the upstream and hostile condition (\ref{bchostile}) at the downstream;
  \item [(H-H)] Hostile condition  (\ref{bchostile}) at both the upstream and the downstream.
\end{enumerate}
\noindent We adopt the baseline parameters in \cite{Speirs2001} and vary their values to see the influences of different factors. The units of parameters are given in Table \ref{parameterunit}. Note that for simplicity, we choose a constant growth rate $r_j$ on each edge $k_j$, which may result in discontinuity of the growth rate $r(x)$ at the interior vertices of the network. As this assumption can be considered as an approximation of a continuous growth rate in the network, it does not change the essence of our results.

\begin{figure}[t!]
\centering
\includegraphics[height=3.5in, width=5.5in]{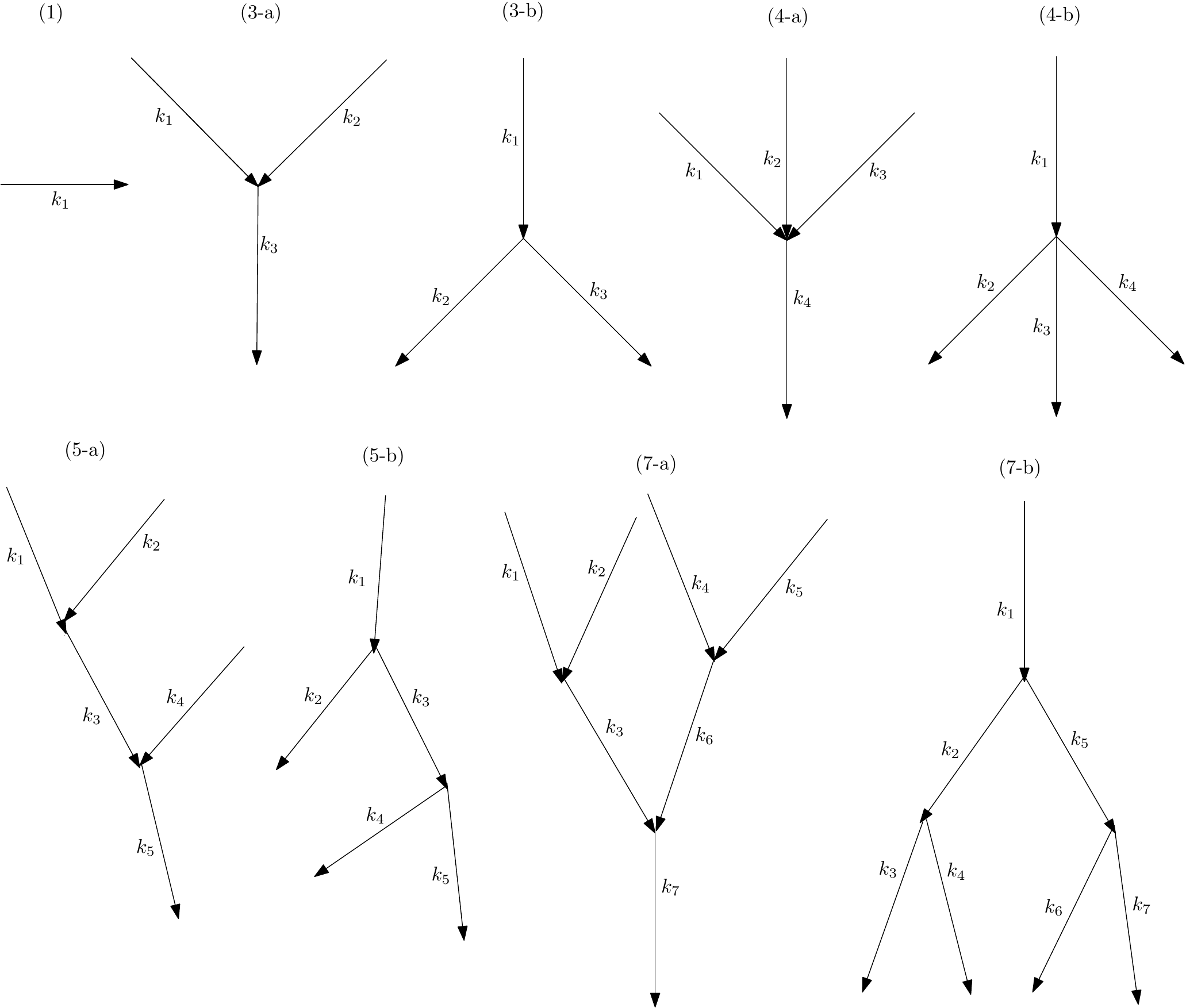}
%
\caption{The river networks with multiple branches that are considered in Section
\ref{simulationsection}. The arrows represent the direction of the water
flow.  The $i$-th branch of the network is represented by $k_i$.
 } \label{rivershapesfig}
\end{figure}

\vspace{0.2cm}

\begin{table}[htp]
\begin{center}
\begin{tabular}{|c|c|c|c|c|c|c|c|c|c|c|c|}
  \hline
  \mbox{Parameter }&$L$ \& $l_j$ &$D_j$ &$v_j$ &$A_j$  & $r_j$&$m_j$ &$Q_j$ &$n_j$ &$B_j$ &$S_{0j}$ &$y_j$\\
  \hline
 \mbox{Unit }& m  &m$^2/$s & m/s &m$^2$ &1/s &1/s &m$^3$/s &s/m$^{1/3}$ &m &m/m&m\\
  \hline

\end{tabular}
\end{center}
\caption{The unit of parameters.}
\label{parameterunit}
\end{table}

\subsection{The influence of the river network structure on population persistence}

Natural rivers are rarely in form of single branches but in various types of networks. The structure (or topology)
of the river network influences hydrodynamics in the network as well as the intrinsic  ecosystem dynamics.

\begin{figure}[t!]
\centering
\includegraphics[height=2.2in,
width=2.5in]{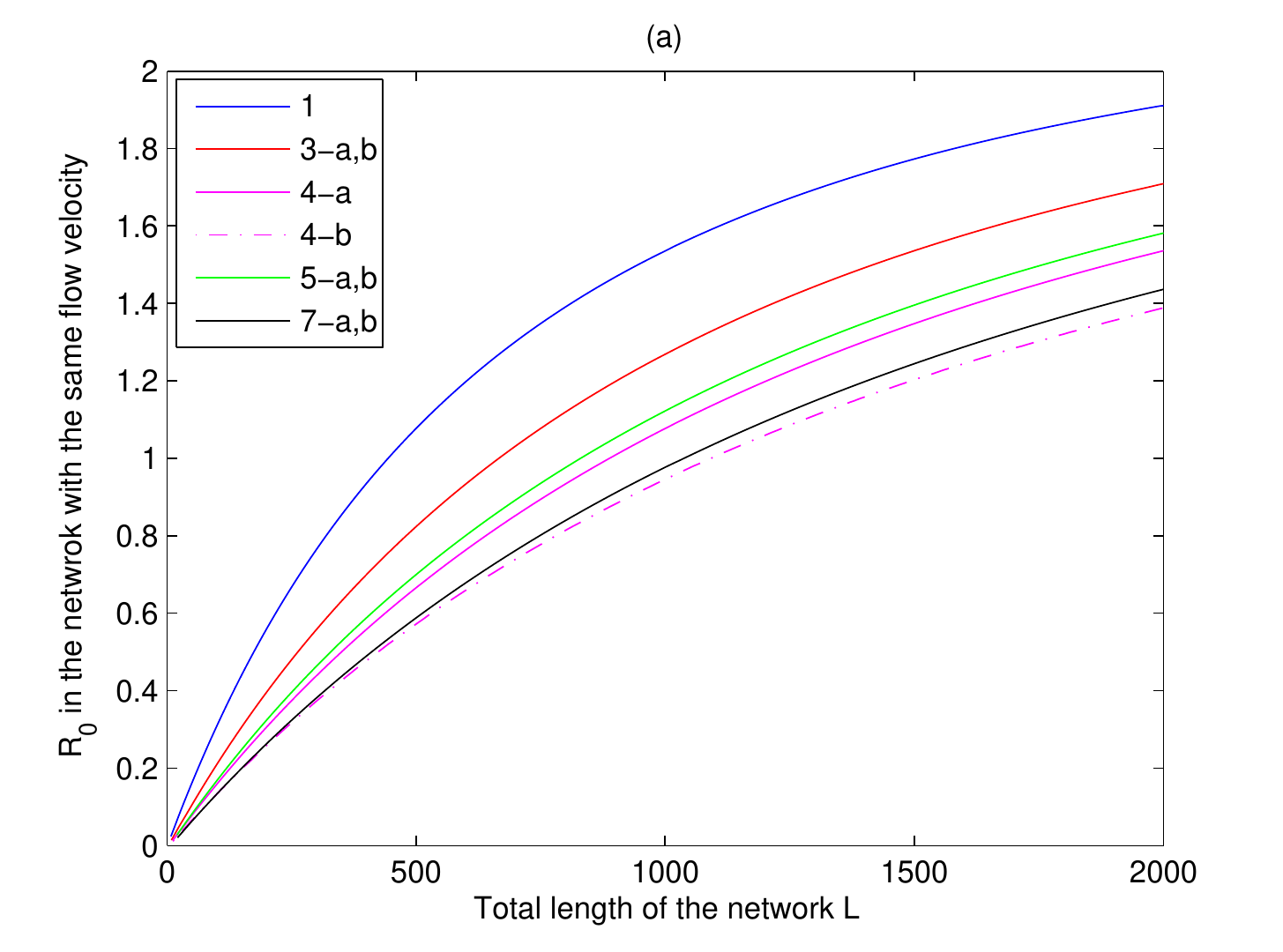}
  \includegraphics[height=2.2in, width=2.5in]{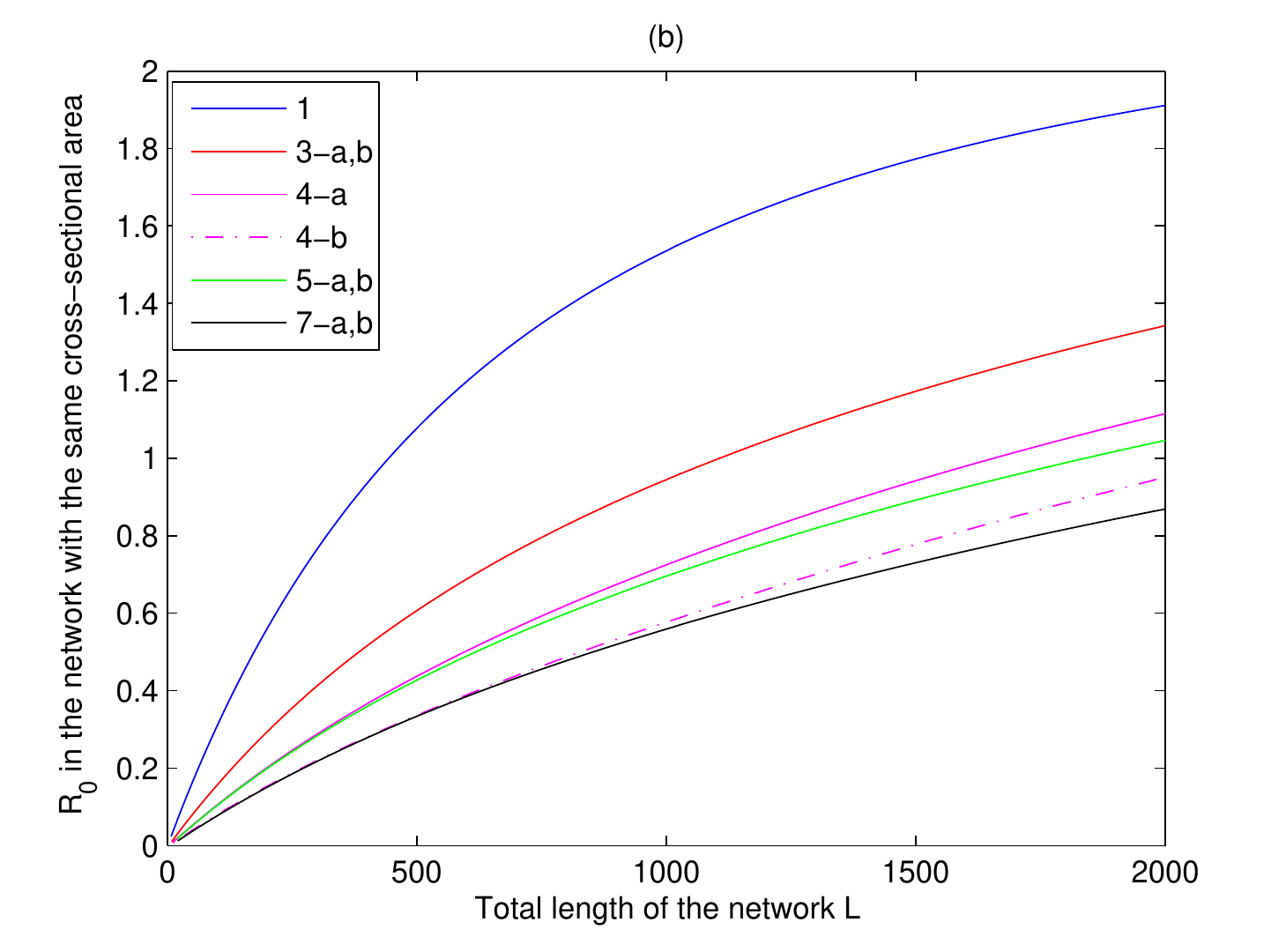}

   \includegraphics[height=2.2in, width=2.5in]{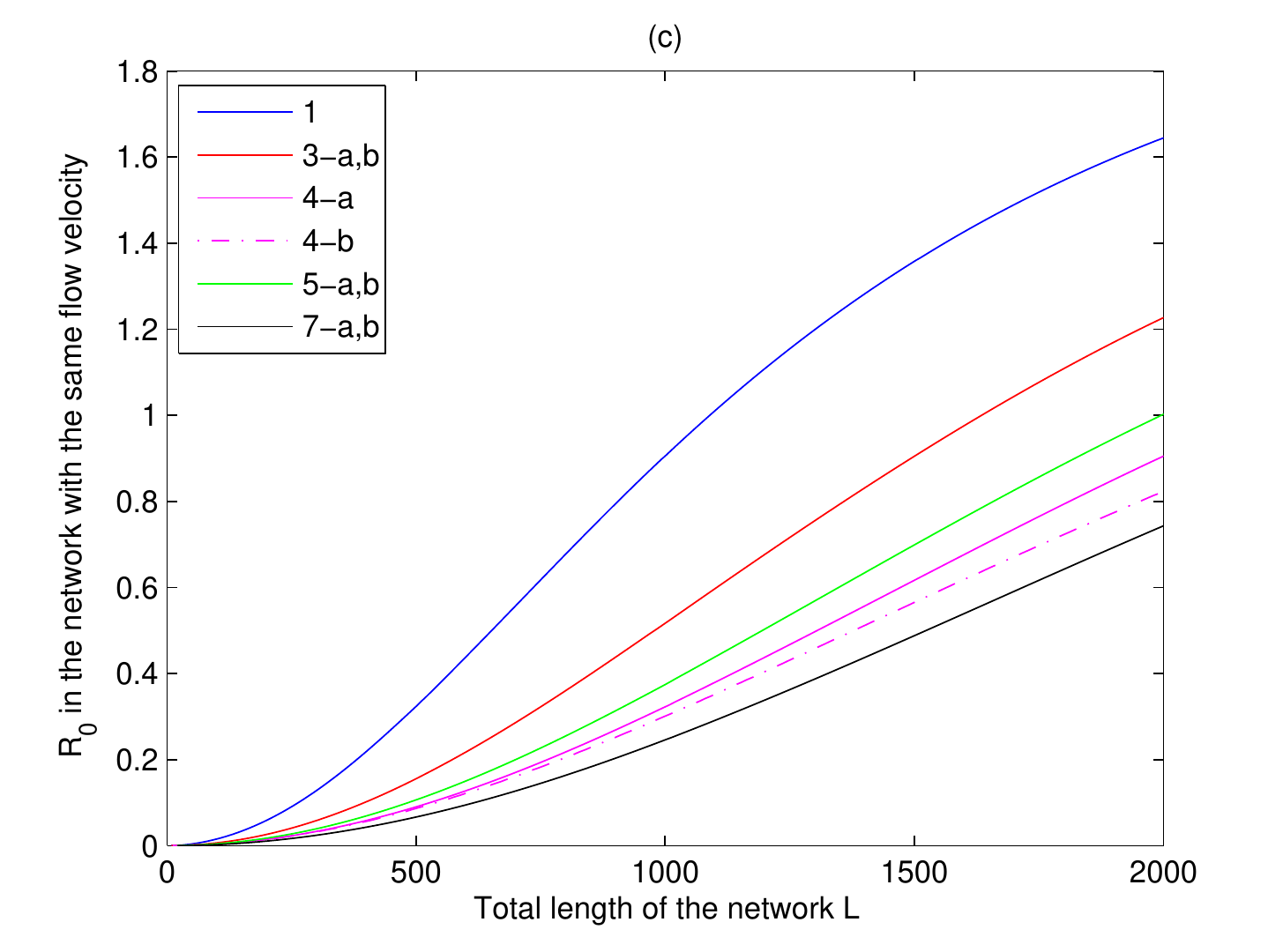}
   \includegraphics[height=2.2in, width=2.5in]{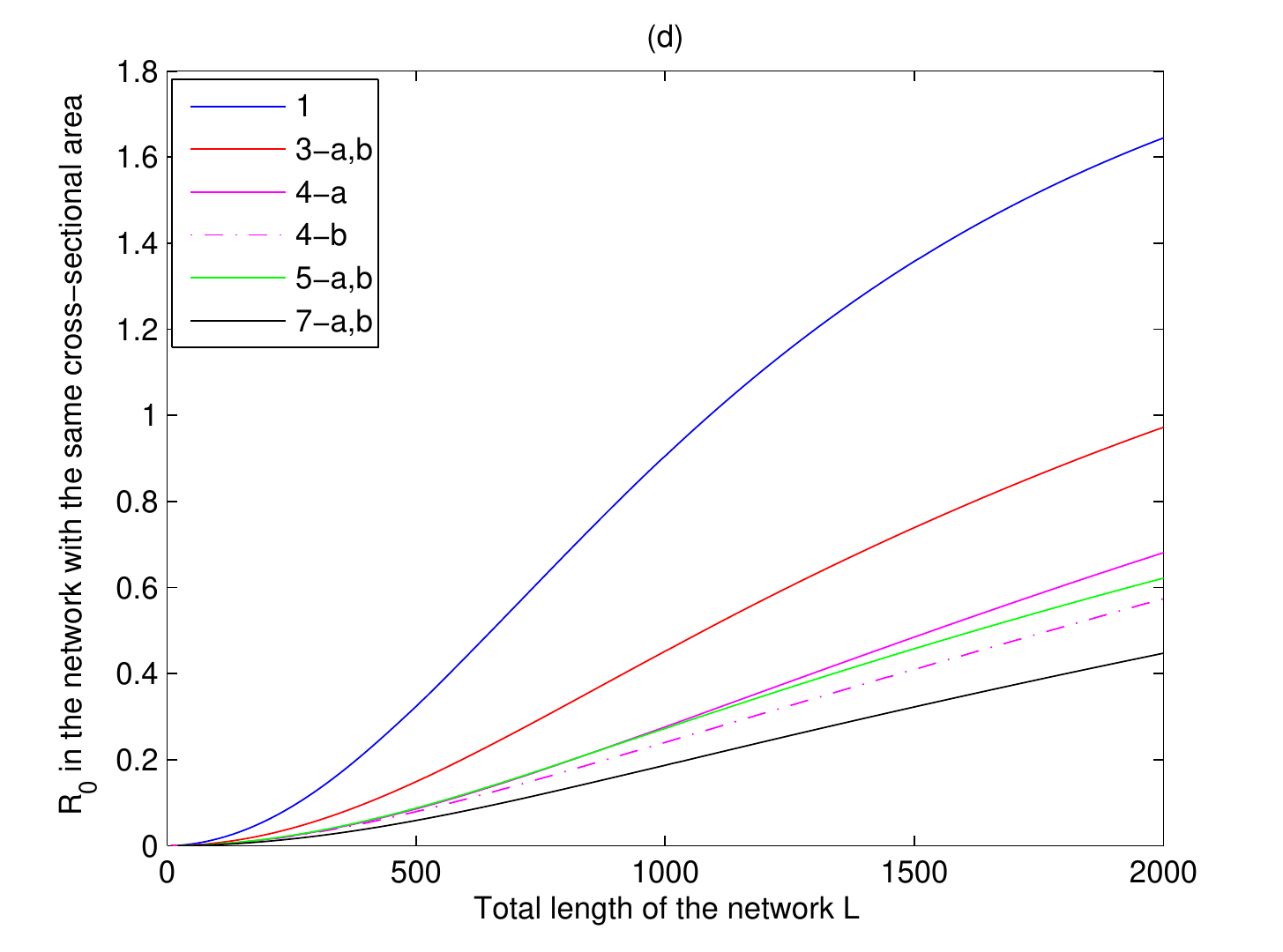}

\includegraphics[height=2.2in,
width=2.5in]{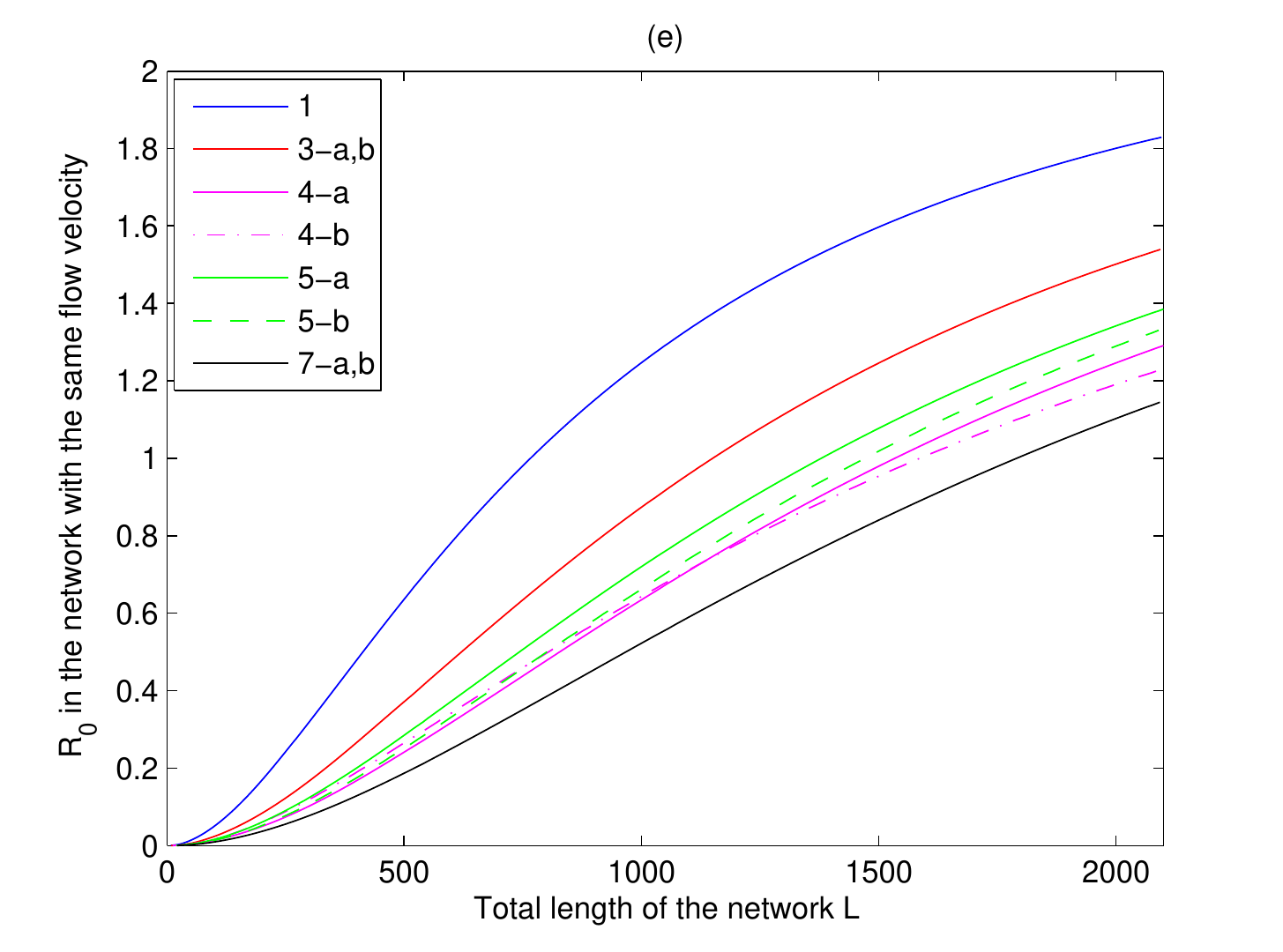}
  \includegraphics[height=2.2in, width=2.5in]{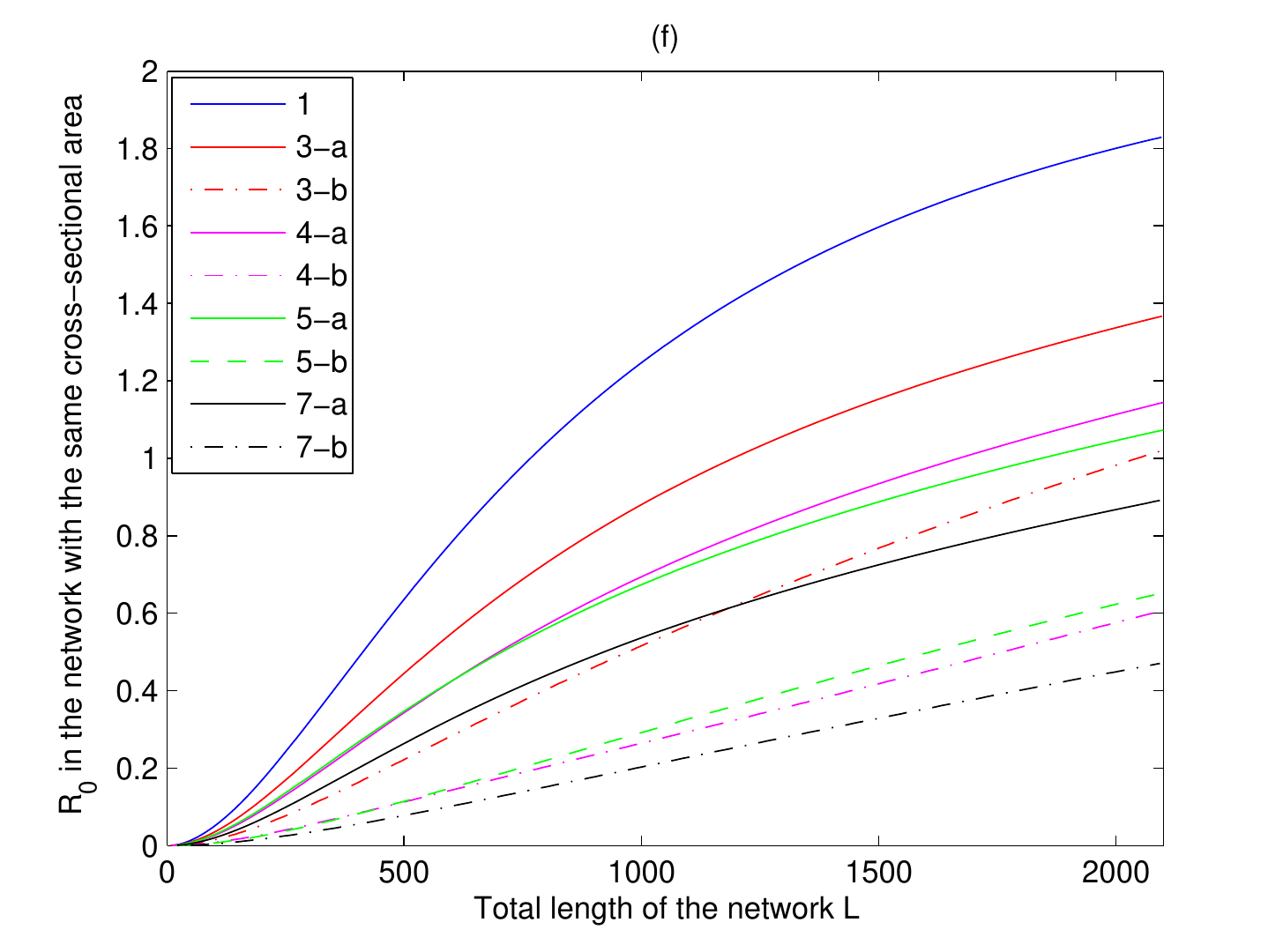}
\caption{The relationship between $\mathcal{R}_0$ and the total length $L$ of
the river network. The curve ``1" represents $\mathcal{R}_0$ in the river of a
single branch. Parameters:  $D_j=0.35$,
   $m_j=0.06/(24\times 3600)$,
  $r_j=0.45/(24\times3600)$. In (a, c, e), the flow advection rate is the same in the
network $v_j=0.0015$; the cross-sectional area is $A_j=1$ in the
upstream branches before merging (e.g., $A_1=A_2=1$ in (3-a)) or in
the downstream branches after splitting  (e.g., $A_2=A_3=1$ in
(3-b)). In (b, d, f), the cross-sectional area is the same in the network
$A_j=1$;  the flow advection rate is $v_j=0.0015$ in the upstream
branches before merging (e.g., $v_1=v_2=0.0015$ in (3-a)) or in the
downstream branches after splitting  (e.g., $v_2=v_3=0.0015$  in
(3-b)). In (a, b), (ZF-FF) boundary conditions are applied. In
(c, d), (H-H) boundary conditions are applied. In (e, f), (ZF-H) boundary conditions are applied.
 } \label{sb1b2k12region}
\end{figure}
 \begin{figure}[t!]
\centering
\subfigure[]{ \label{left}
\hspace{-0.6cm}
\includegraphics[height=2.5in,
width=2.8in]{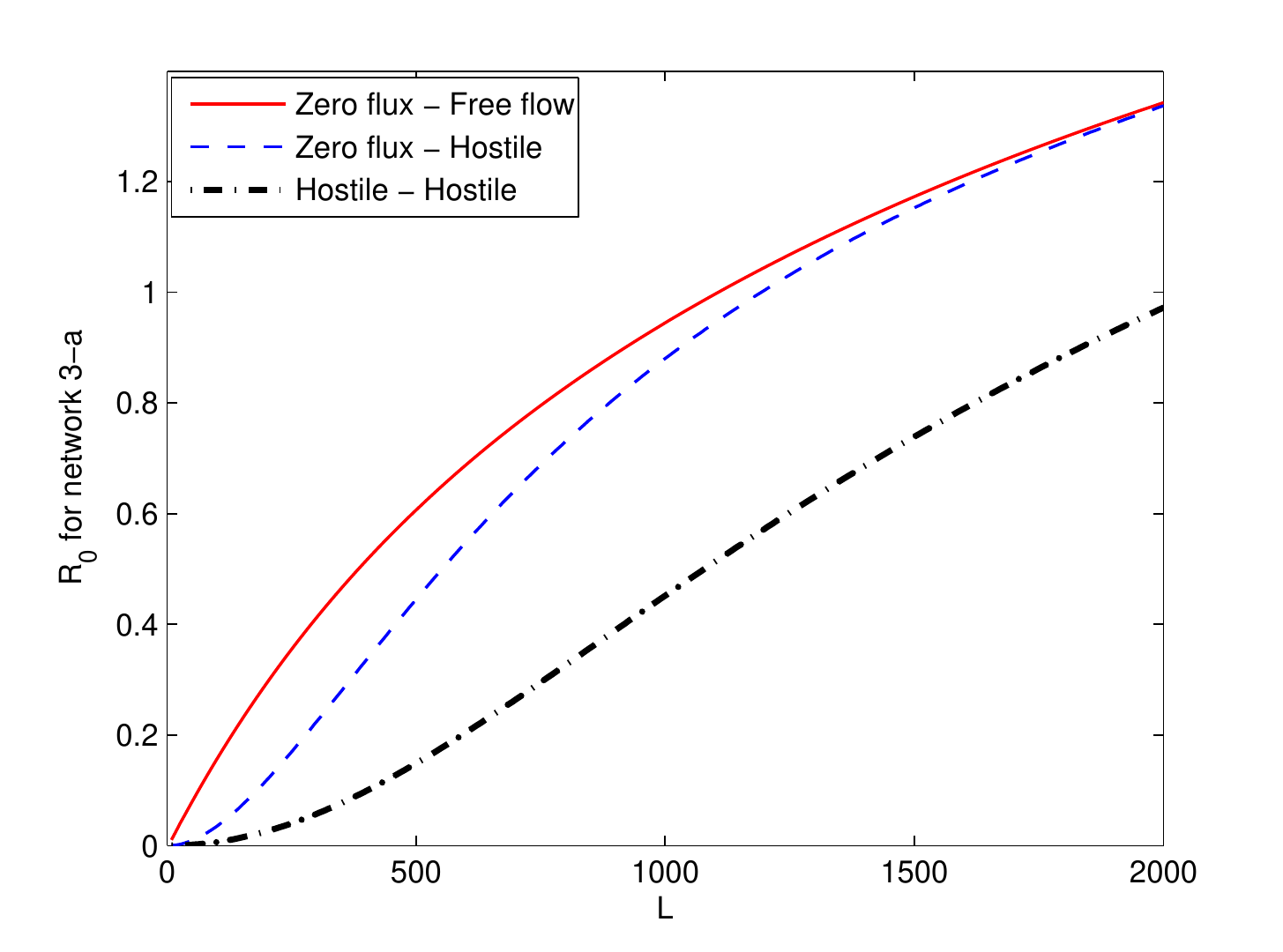}}
\hspace{-0.55cm}
  \subfigure[]{   \label{right}\includegraphics[height=2.5in, width=2.8in]{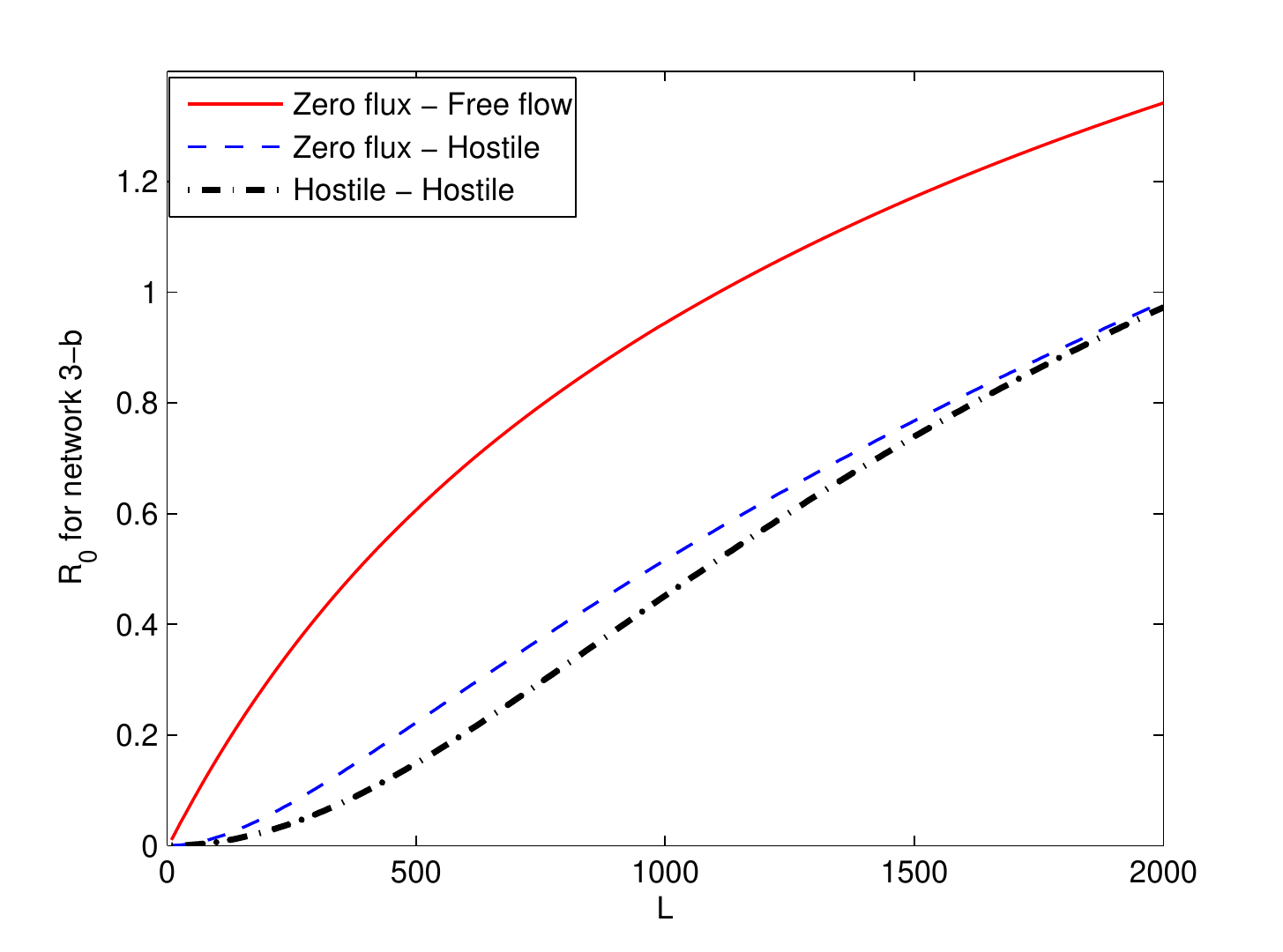}}
\caption{The relationship between $\mathcal{R}_0$ and the total length $L$ of
the river networks (3-a) and (3-b), under different boundary
conditions. Parameters are the same as for Figure
\ref{sb1b2k12region}.
 } \label{sb1b2k12region3bc}
 \end{figure}

To see how the network structure influences population persistence, we compare the values of $\mathcal{R}_0$ in a single branch river and in all river networks in Figure \ref{rivershapesfig}.
Suppose that the growth rate and the diffusion rate do not change throughout each network. We vary the flow conditions by fixing the advection rates but varying the cross-sectional areas or fixing the cross-sectional areas but varying the advection rates, according to the conservation relation (\ref{conservationflow}) of the flow at the interior junctions.

\subsubsection{The influence of the total length of the network}
In a specific network structure, assume that the lengths of all branches are
the same and vary the total length of the network. Figure
\ref{sb1b2k12region}  shows that when the
total length increases, the net reproductive rate $\mathcal{R}_0$ increases for
all  the networks in consideration. This coincides with the well-known result in
one-dimensional river (see e.g., \cite{Jin2011,Lutcherlewis2005}):
given the same habitat conditions, increasing the total habitat size helps population persistence.

Figure \ref{sb1b2k12region} also shows that in radial trees ((3-a), (3-b), (7-a), and (7-b) in Figure \ref{rivershapesfig}), when the total length of the network is fixed, increasing the number of levels reduces the value of  $\mathcal{R}_0$.
It has been shown in \cite{Sarhad2014} that population dynamics on a radial tree is equivalent to that in a one-dimensional river formed from one upstream to one downstream of the tree. Our observation coincides with this result in \cite{Sarhad2014} since a radial tree of $3$ branches with total length $L$ is equivalent to a one-dimensional river of length $2L/3$, while a radial tree of $7$ branches with total length  $L$ is equivalent to a one-dimensional river of length $3L/7$.
However, one cannot conclude a general result from this that in river networks, when the total river length is fixed, the networks with more branches have smaller net reproductive rates. In the non-radial networks (4-a,b) and (5-a,b), when the flow velocity is fixed, the net reproductive rates of networks (5-a) and (5-b) are larger than those of networks (4-a) and (4-b); see Figure \ref{sb1b2k12region} (a,c). Networks with $5$ branches also have larger $\mathcal{R}_0$ than those with $4$ branches when the cross-sectional areas are fixed and the total river lengths are small; see Figure \ref{sb1b2k12region} (b,d).

\subsubsection{The influence of boundary conditions}

Figure \ref{sb1b2k12region} shows the relationship between the total river network length $L$ and the net reproductive rate $\mathcal{R}_0$ under different boundary conditions.
When (ZF-FF) or (H-H) boundary conditions are applied, the same number of branches merging from the upstream or splitting into the downstream lead to the same net reproductive rate in radial trees (3-a,b) and (7-a,b) as well as in (5-a,b). That is, when both boundary conditions are bad (hostile) or not bad (not hostile), more branches in the upstream or in the downstream does not change the net reproductive rate, provided that two branches merge into one or one branch splits into two at each junction point. Nevertheless, when (ZF-H) boundary conditions are applied, if the total branch numbers are the same, the net reproductive rate of the network merging from the upstream is not less than  the one of the network splitting into the downstream, in (3-a,b), (5-a,b) or (7-a,b). That is, when the downstream boundary condition is bad (hostile), more branches in the downstream cannot result in a  larger net reproductive rate, provided that two branches merge into or are split from one branch at each junction point.

Among all the networks, the $4$-branch networks are exceptional. When (ZF-FF) or (H-H) boundary conditions are applied, $4$ branches merging from the upstream result in a larger $\mathcal{R}_0$ than splitting into the downstream. Nevertheless, when (ZF-H) boundary conditions are applied and the total river length is small, the network with $4$ branches merging from the upstream may have a smaller net reproductive rate than the one with $4$ branches splitting into the downstream if the advection rate is fixed.

Overall, in all these types of river networks with equal branch length, having more upstream branches helps the population persistence or at least does  not accelerate the population extinction, provided that the upstream ends are not the only boundaries that are subjected to hostile conditions and that the total network length is sufficiently large.

To see the effect of boundary conditions on the net reproductive rate more closely, we focus on the $R_0$ values in networks with $3$ branches in Figure \ref{sb1b2k12region3bc}. It shows that $\mathcal{R}_0$ under zero-flux condition at the upstream and free flow condition at the downstream is the largest and the one under hostile condition at both ends is the smallest. When the total length of the
network is small,  $\mathcal{R}_0$ under hostile boundary conditions is much lower since individuals are close to the  hostile boundary and are subjected heavy stress of being removed from the system. When the total length of the network is large, the hostile condition  at the downstream does not make much difference on $\mathcal{R}_0$ compared to the free flow condition at the downstream if more branches are at the upstream, and the upstream conditions do not influence $\mathcal{R}_0$ very much if more branches are at the downstream with hostile condition. Therefore, both the boundary conditions and the type of river network affect the net reproductive rate.

\subsection{Population persistence in  uniform flows}

Hydrological and physical factors in a real river are closely related  to each other (see e.g., \cite{Chaudhrybook,JinLewisBMB}).
To better see how the population persistence is influenced by physical, hydrological and biological conditions, we incorporate the explicit relation between hydrological and physical factors into the population model and investigate their effect on the net reproductive rate and the positive steady state (if exists).

Assume that in the network each branch $k_j$ has a constant bottom slope $S_{0j}$, a constant bottom Manning roughness $n_j$, and rectangular cross sections with a constant width $B_j$. We further assume that the water flow is at the steady state with flow discharge $Q_j$ in $k_j$, hence, there is a uniform flow in $k_j$ \cite{Chaudhrybook}. As a result,  the water depth $y_j$ in $k_j$ can be estimated by the normal depth defined in (\ref{normaldepthq}) in Appendix \ref{appflow},  that is,
\begin{equation}\label{velocityeqn}
y_j=\left(\frac{Q_j^2 n_j^2}{B_j^2 S_{0j} k^2}\right)^{\frac{3}{10}}, 
\end{equation}
which yields the  flow velocity in  branch $k_j$ as
\begin{equation}\label{velocityeqn}
v_j=\frac{Q_j}{A_j}=\frac{Q_j}{B_jy_j}.
\end{equation}
We then substitute (\ref{velocityeqn}) into (\ref{model}) to study how parameters influence population persistence in uniform flows in river networks where two or three branches merge into one branch (i.e., (3-a) and (4-a) in Figure \ref{rivershapesfig}).

\begin{figure}[t!]
\centering
\subfigure[]{ \label{left}
\includegraphics[height=2.5in,
width=2.8in]{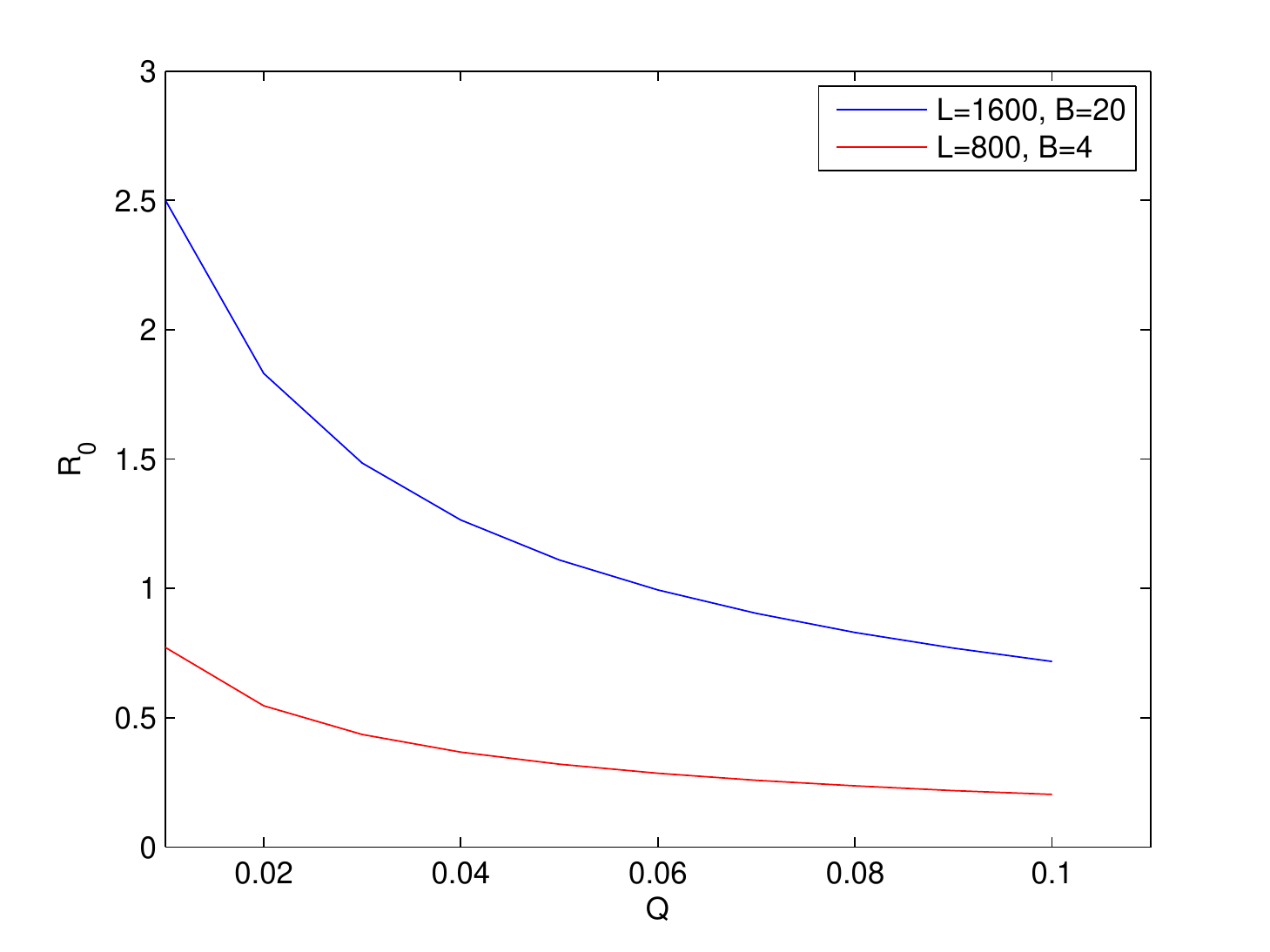}}
  \subfigure[]{   \label{right}\includegraphics[height=2.5in, width=2.8in]{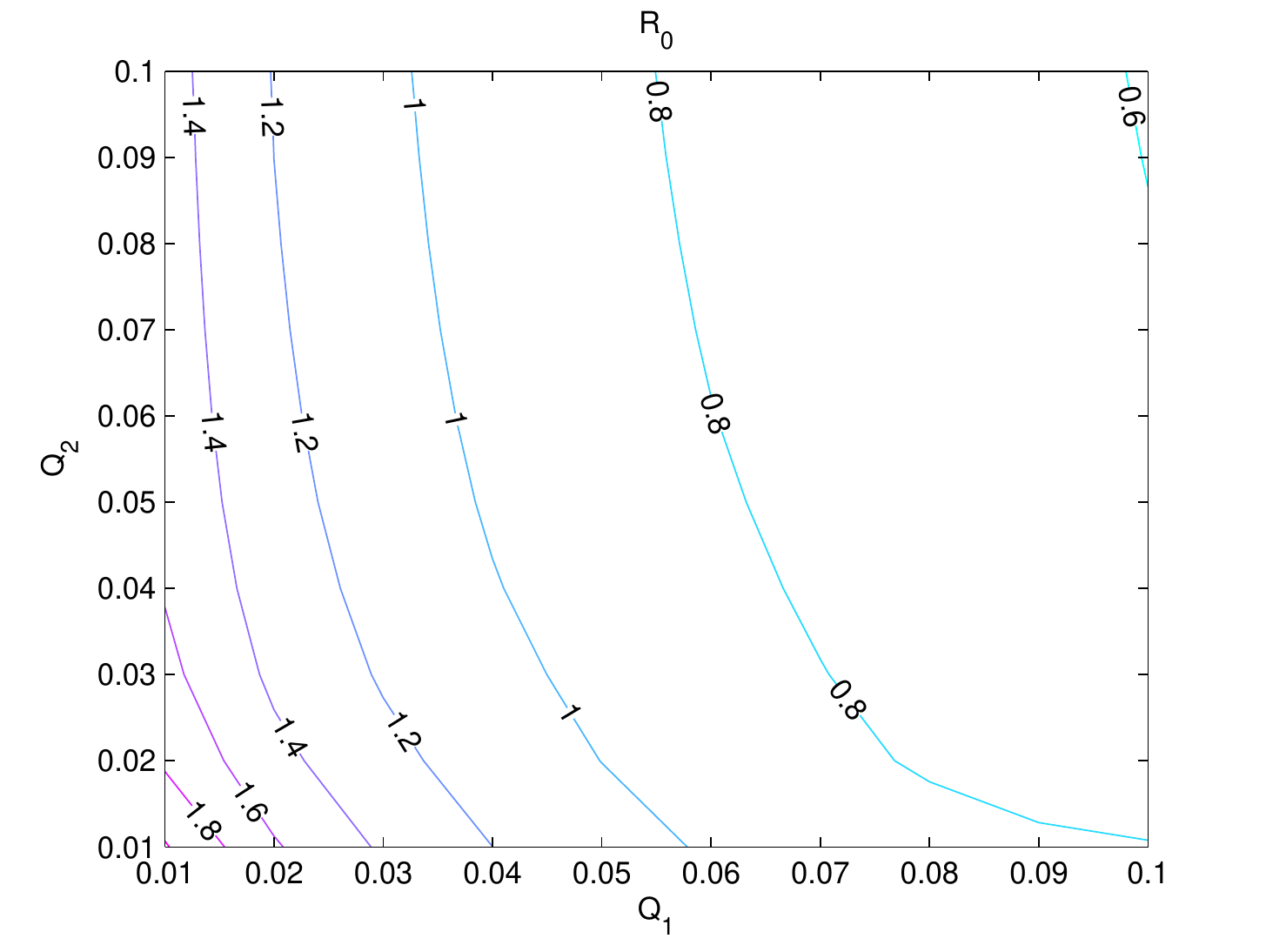}}
\caption{The relationship between $\mathcal{R}_0$ and the flow discharge.
Parameters:  $D_j=0.6$,
   $m_j=0.06/(24\times 3600)$,
  $r_j=0.8/(24\times3600)$,
  $n_j=0.2$, $S_{0j}=0.000001$.
(a): in two single branch rivers. (b): in the network (3-a),
 the curves are contour lines for $\mathcal{R}_0$; other parameters are $l_j=800$,
   $B_1=B_3=20$, $B_2=4$. (ZF-FF) boundary conditions are applied.
 } \label{Q1Q2R0}
\end{figure}

\subsubsection{The influence of the flow discharges on $\mathcal{R}_0$ }

Firstly we consider two isolated river branches, a longer and wider one with length $l_1=1600$ and width $B_1=20$,  a shorter and narrower one with length $l_2=800$ and width $B_2=4$, both with the same bottom slope $S_{0j}=0.000001$ and bottom roughness $n_j=0.2$.
Figure \ref{Q1Q2R0}(a) shows the relationship between the net reproductive rate $\mathcal{R}_0$ and the flow discharge in each river, when the diffusion rate, birth and death rates are the same in both rivers and (ZF-FF) boundary conditions are applied.  $\mathcal{R}_0$ decreases in both rivers as the upstream flow discharge $Q$ increases, but it is larger in the longer and wider river than in the shorter and narrower river.

Now we consider a river network of type (3-a), which is the result of the above small river merging into the large river at the midpoint of the large river. We use subscript $_1$ for parameters in the large river branch and subscript $_2$ for parameters in the small river branch.  Figure \ref{Q1Q2R0}(b)  shows the dependence of $\mathcal{R}_0$ on the flow discharges $Q_1$ and $Q_2$ at the upstream of both river branches. Clearly, $\mathcal{R}_0$ is large when the upstream flow discharges $Q_1$ and $Q_2$ are both small and $\mathcal{R}_0$ is small when both $Q_1$ and $Q_2$ are large. Hence, in a river network, it is still true that low upstream flow discharges help the population persistence and high upstream flow discharges accelerate the population extinction. Moreover, for the parameters we choose, varying $Q_1$ changes $\mathcal{R}_0$ more than varying $Q_2$. That is, given the same habitat and demography conditions, in a merging river network, the upstream flow discharge in the large river influences the global population persistence/extinction more than the upstream flow discharge in the small river. Note that in our case, the population will be extinct in the small river if isolated (see the red curve in Figure \ref{Q1Q2R0}(a)), but merging the rivers into a network helps the population to persist in the whole network, provided that the population can persist in the isolated large river.

\begin{figure}[t!]
\centering
\subfigure[]{ \label{left}
\includegraphics[height=2.5in,
width=2.8in]{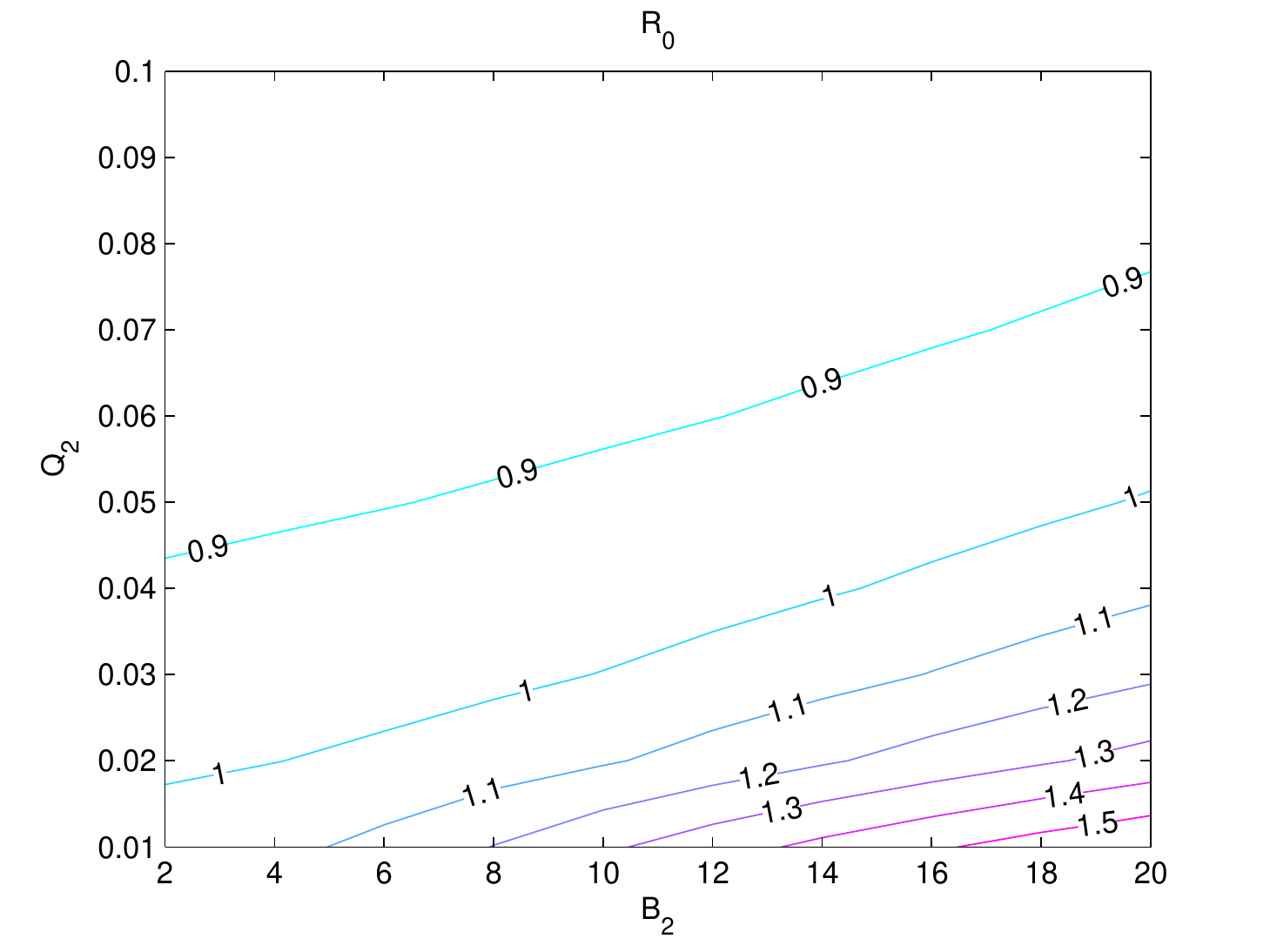}}
  \subfigure[]{   \label{right}\includegraphics[height=2.5in, width=2.8in]{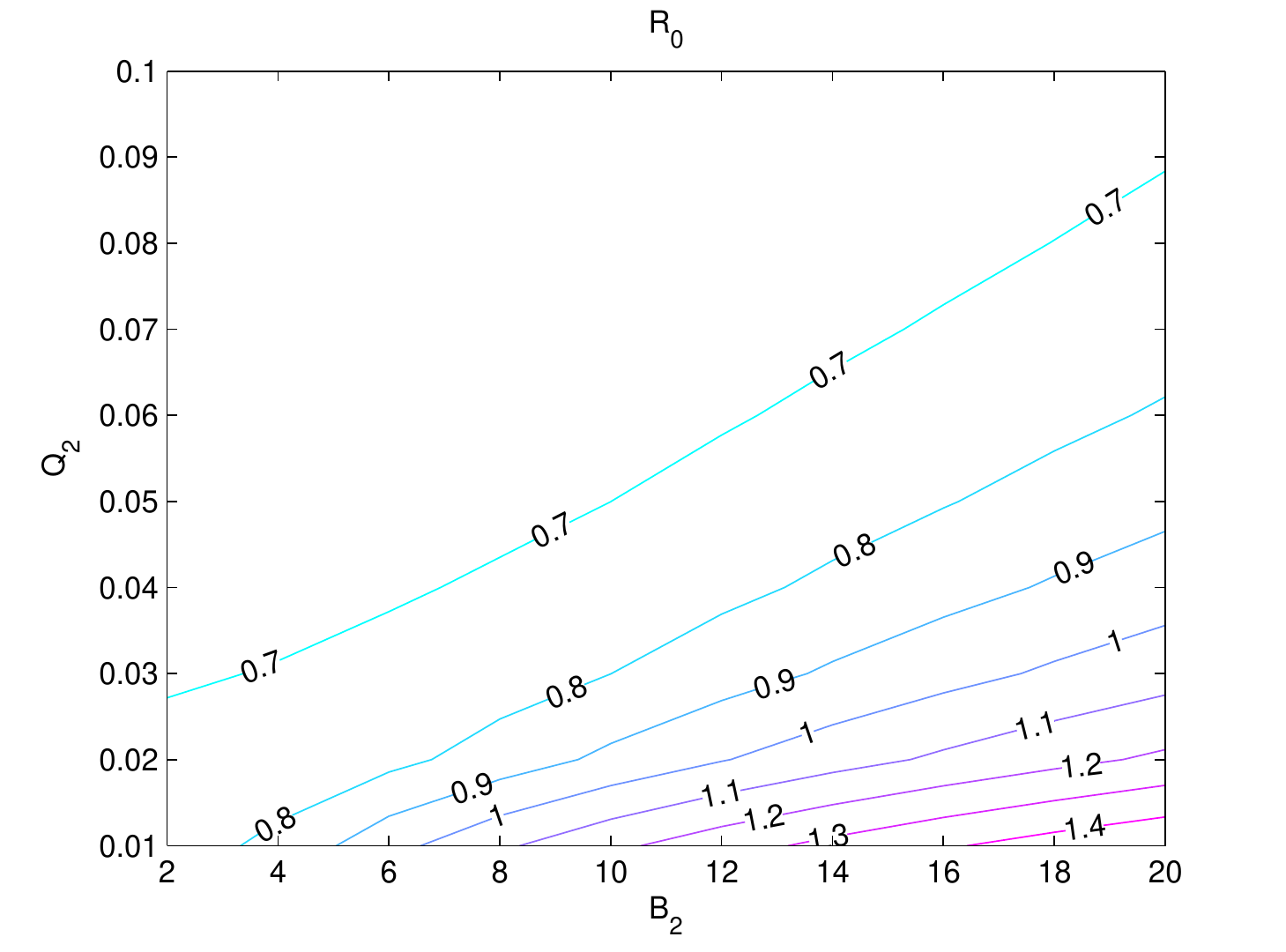}}
\caption{The relationship between $\mathcal{R}_0$ and the flow discharge $Q_2$,
the width $B_2$ of  the branch $k_2$ in network (3-a). The curves are
contour lines for $\mathcal{R}_0$. Parameters:  $D_j=0.6$,
   $m_j=0.06/(24\times 3600)$,
  $r_j=0.8/(24\times3600)$,
  $n_j=0.2$, $S_{0j}=0.000001$,
    $l_j=800$, $B_1=B_3=20$.
    (a):  $Q_1=0.05$;
    (b):  $Q_1=0.09$. Note that in the isolated large river, $\mathcal{R}_0=1.109$ when $Q_1=0.05$ and $\mathcal{R}_0=0.768$ when $Q_1=0.09$. (ZF-FF) boundary conditions are applied.
     } \label{B2Q2R0}
\end{figure}
\begin{figure}[t!]
\centering
\subfigure[]{ \label{left}\includegraphics[height=2.5in, width=2.8in]{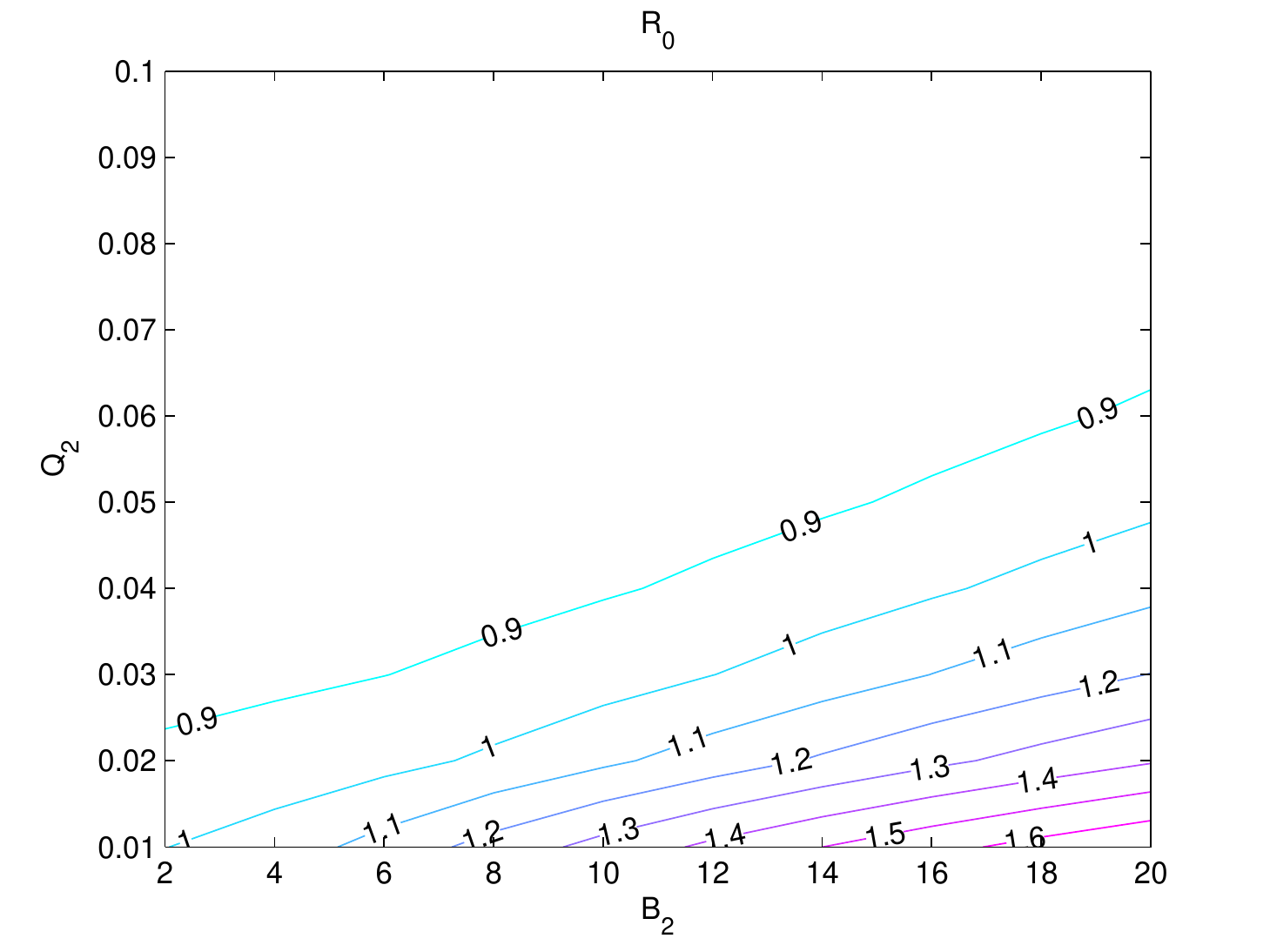}}
\subfigure[]{ \label{right}
\includegraphics[height=2.5in, width=2.8in]{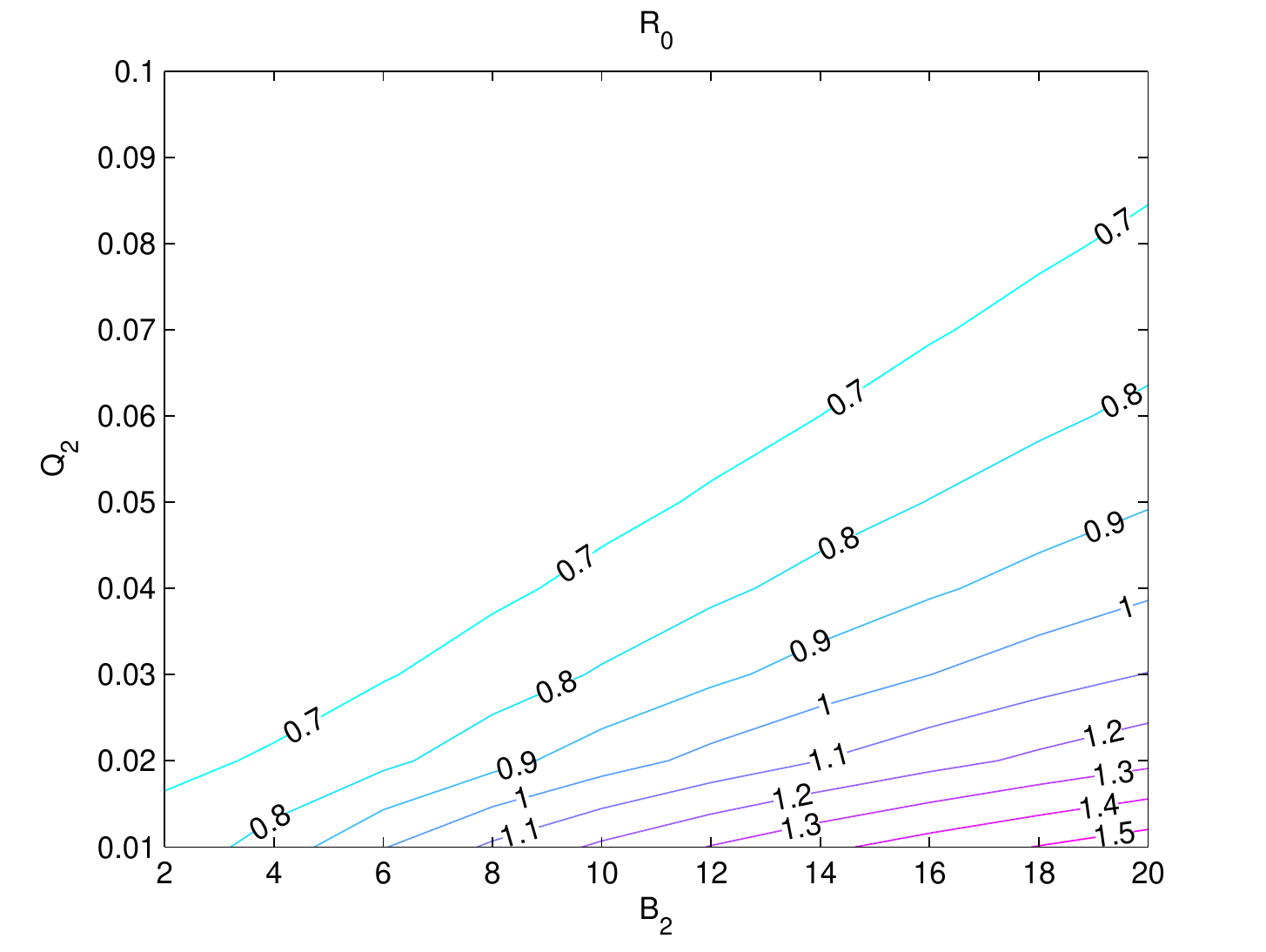}}
\caption{The relationship between $\mathcal{R}_0$ and the flow discharge and the width of the branch $k_2$ and $k_3$ of network (4-a). Parameters:  $D_j=0.6$,
   $m_j=0.06/(24\times 3600)$,
  $r_j=0.8/(24\times3600)$,
  $n_j=0.2$, $S_{0j}=0.000001$,
   $B_1=B_4=20$, $B_2=B_3$, $Q_2=Q_3$. (a): $Q_1=0.05$; (b): $Q_1=0.09$. (ZF-FF) boundary conditions are applied.
     } \label{B2Q2R0-2}
\end{figure}

\subsubsection{The influence of the flow discharges and the widths on $\mathcal{R}_0$}

We continue to consider river networks as the result of merging large and small rivers.
In networks of type (3-a), assume that the large river is given and that the conditions in the small upstream branch vary.
Figure \ref{B2Q2R0} shows the dependence of $\mathcal{R}_0$ on the upstream flow discharge $Q_2$ and the width $B_2$ of the small river branch, in the cases where the population can persist in the isolated large river (Figure \ref{B2Q2R0}(a)) and where the population will be extinct in the isolated large river (Figure \ref{B2Q2R0}(b)). Both figures show the same phenomenon: to increase $\mathcal{R}_0$ or help the population persistence in the whole network, it is necessary to have small upstream flow discharge and large width in the small river, which essentially means that the flow discharge per unit width $Q_2/B_2$ should be low in the small river. If the persistence conditions in the large river become worse (e.g., from Figure \ref{B2Q2R0}(a) to Figure \ref{B2Q2R0}(b)), then lower  $Q_2/B_2$ in the small branch is required to help the population persistence, i.e.,  in order for $\mathcal{R}_0>1$, in the whole network.

In the case where two small rivers with identical conditions merge into a large river and result in a network of type (4-a), the dependence of $\mathcal{R}_0$ and the flow discharge and the width of the small rivers is shown in Figure \ref{B2Q2R0-2}. Similarly, low upstream discharge and large width in the small rivers help the population persistence in the whole network. Comparing Figure \ref{B2Q2R0}(a) and Figure \ref{B2Q2R0-2}, we see that when more small rivers merge into a large river, lower flow discharge per unit width $Q_2/B_2$ is needed to help the population persistence, i.e., for $\mathcal{R}_0$ to go beyond $1$. However, once the conditions are good enough for population to persist, $\mathcal{R}_0$ is larger in the larger network (i.e., (4-a)) than that in the smaller network (i.e., (3-a)).
In general, this indicates that the threshold conditions for population persistence in a large network may be stronger than those in a small network, but population grows better in the large network once it persists.

\begin{figure}[t!]
\centering
\subfigure[]{ \label{left}
\includegraphics[height=2.5in,
width=2.8in]{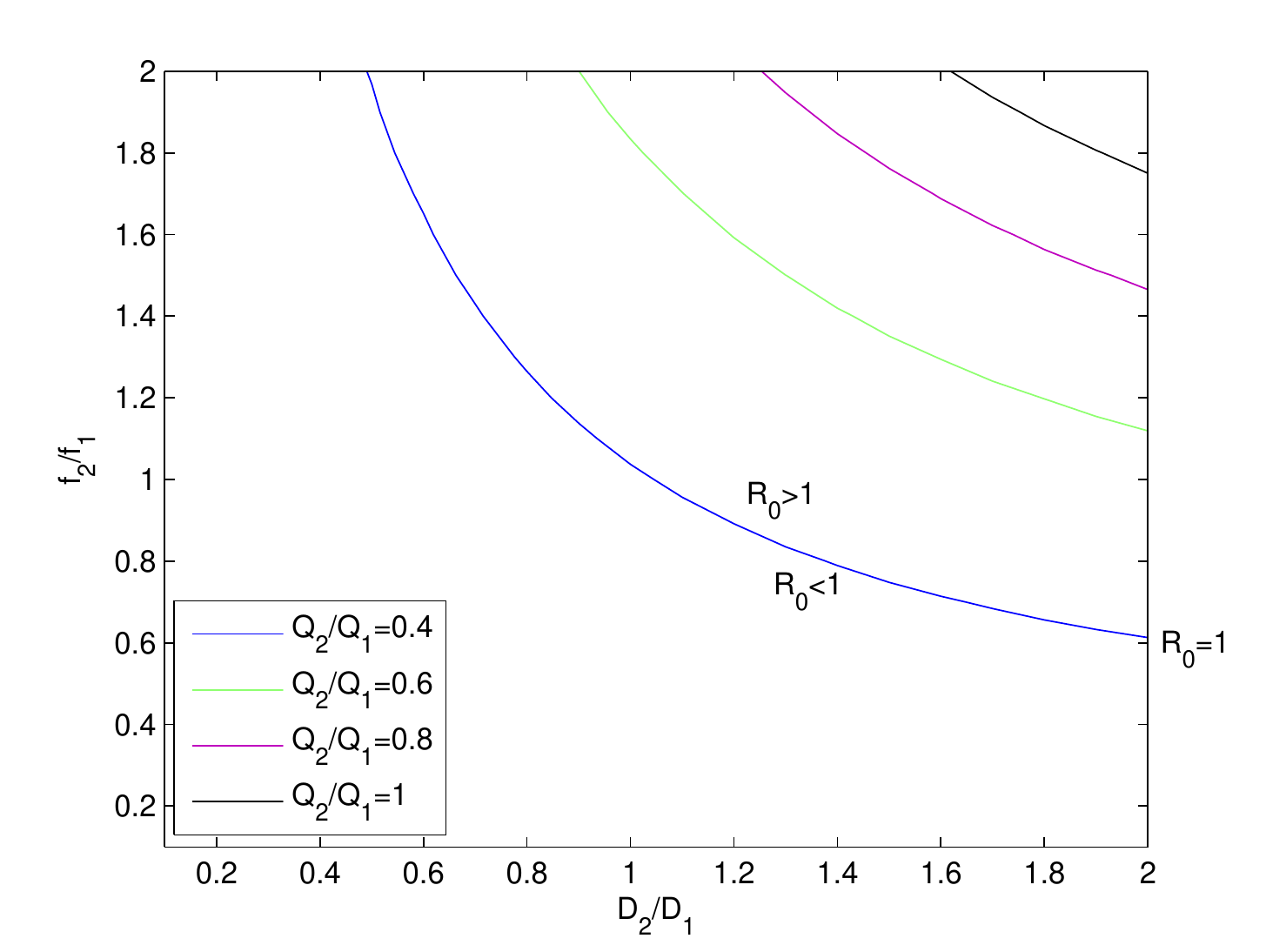}}
  \subfigure[]{   \label{right}\includegraphics[height=2.5in, width=2.8in]{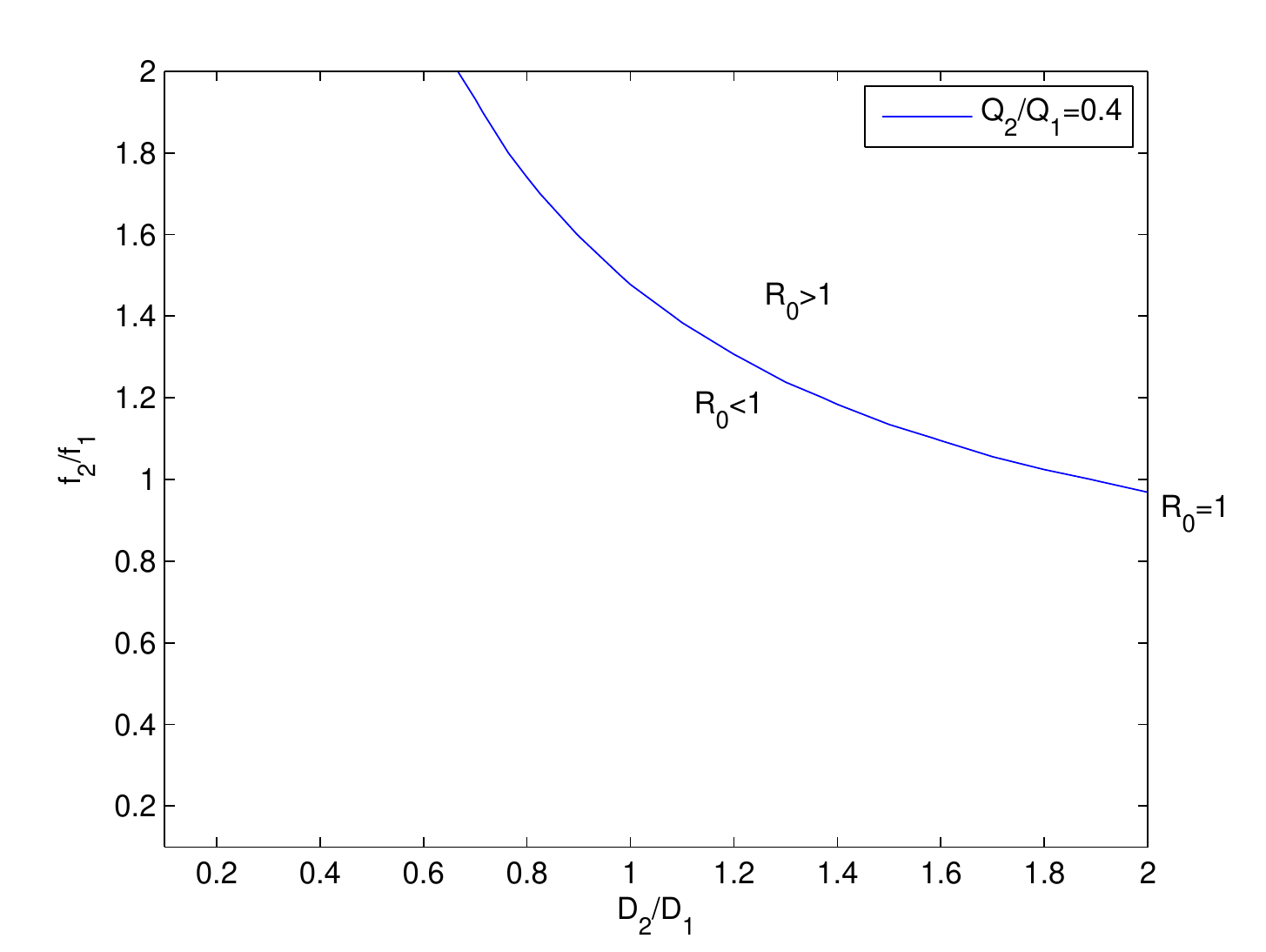}}
\caption{The regions for persistence or extinction in parameter
space in networks (3-a) and (4-a). Parameters:
   $m_j=0.06/(24\times 3600)$, $n_j=0.2$, $S_{0j}=0.000001$, $l_j=800$,
 $Q_1=0.05$. (a): in network (3-a), $D_1=D_3=0.6$, $f_1=f_3=0.8/(24\times3600)$, $B_1=B_3=20$, $B_2=4$;
  (b): $D_1=D_4=0.6$, $f_1=f_4=0.8/(24\times3600)$, $B_1=B_4=20$, $B_2=B_3=4$, $D_2=D_3$, $f_2=f_3$, $Q_2=Q_3$.
  On each curve, $\mathcal{R}_0=1$ under corresponding parameter conditions. Note that population persists in the region where $\mathcal{R}_0>1$ and population will be extinct in the region where $\mathcal{R}_0<1$. (ZF-FF) boundary conditions are applied.
     } \label{dfqR0}
\end{figure}


\begin{figure}[t!]
\centering
\subfigure[]{\label{left}\includegraphics[height=1.55in,width=1.9in]{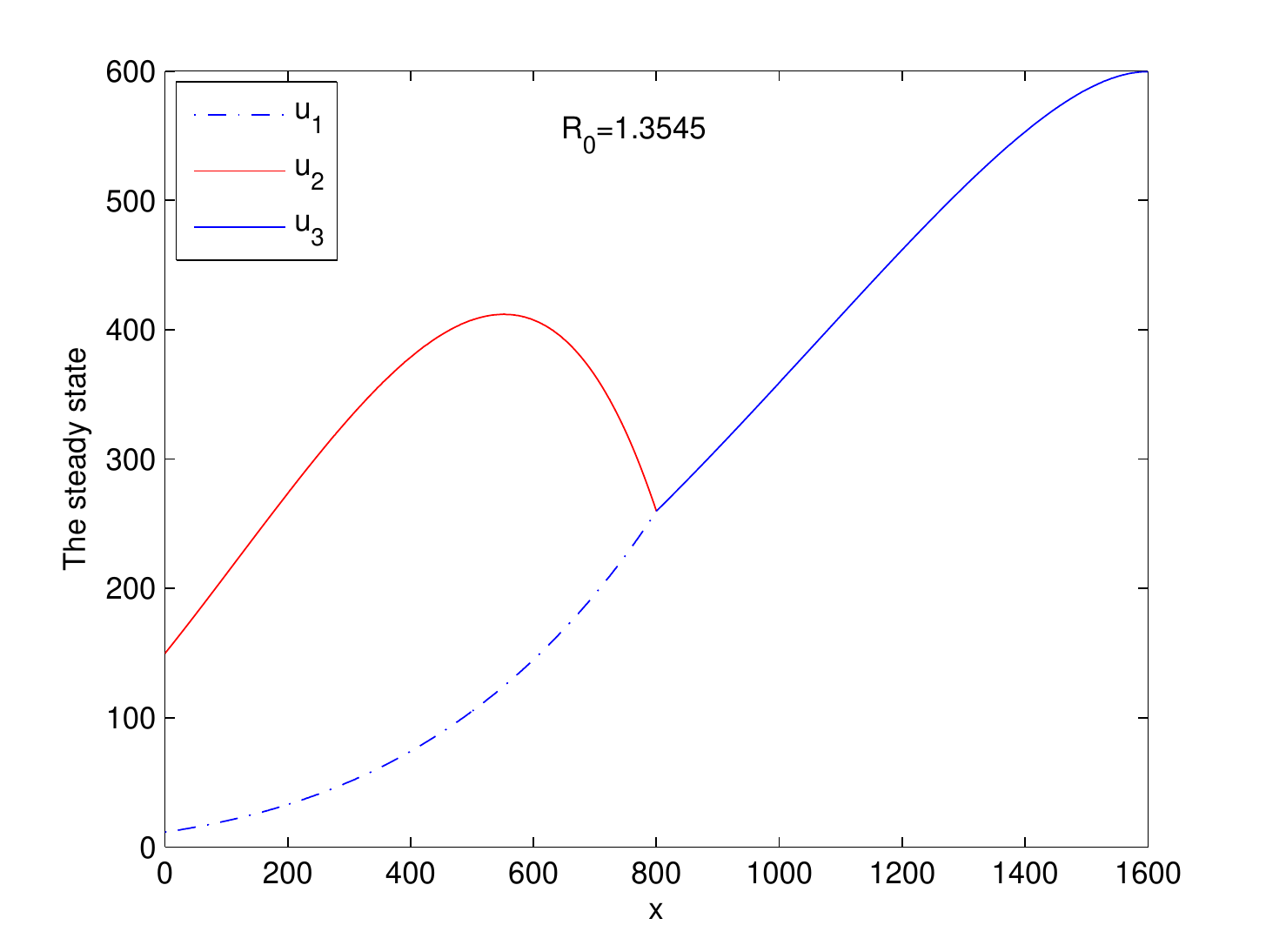}}
\hspace{-0.55cm}
  \subfigure[]{   \label{middle}\includegraphics[height=1.55in, width=1.9in]{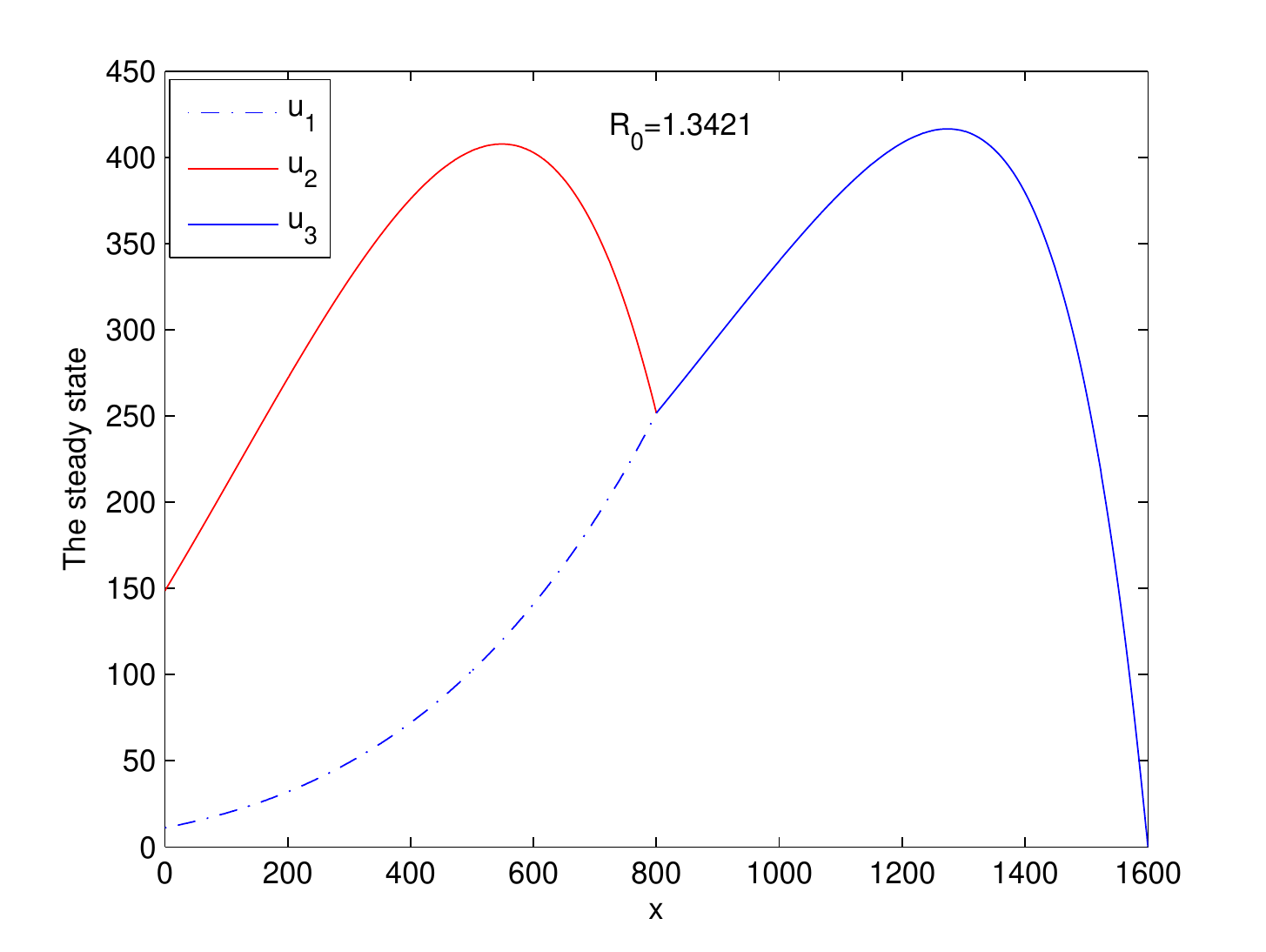}}\hspace{-0.55cm}
  \subfigure[]{   \label{right}\includegraphics[height=1.55in, width=1.9in]{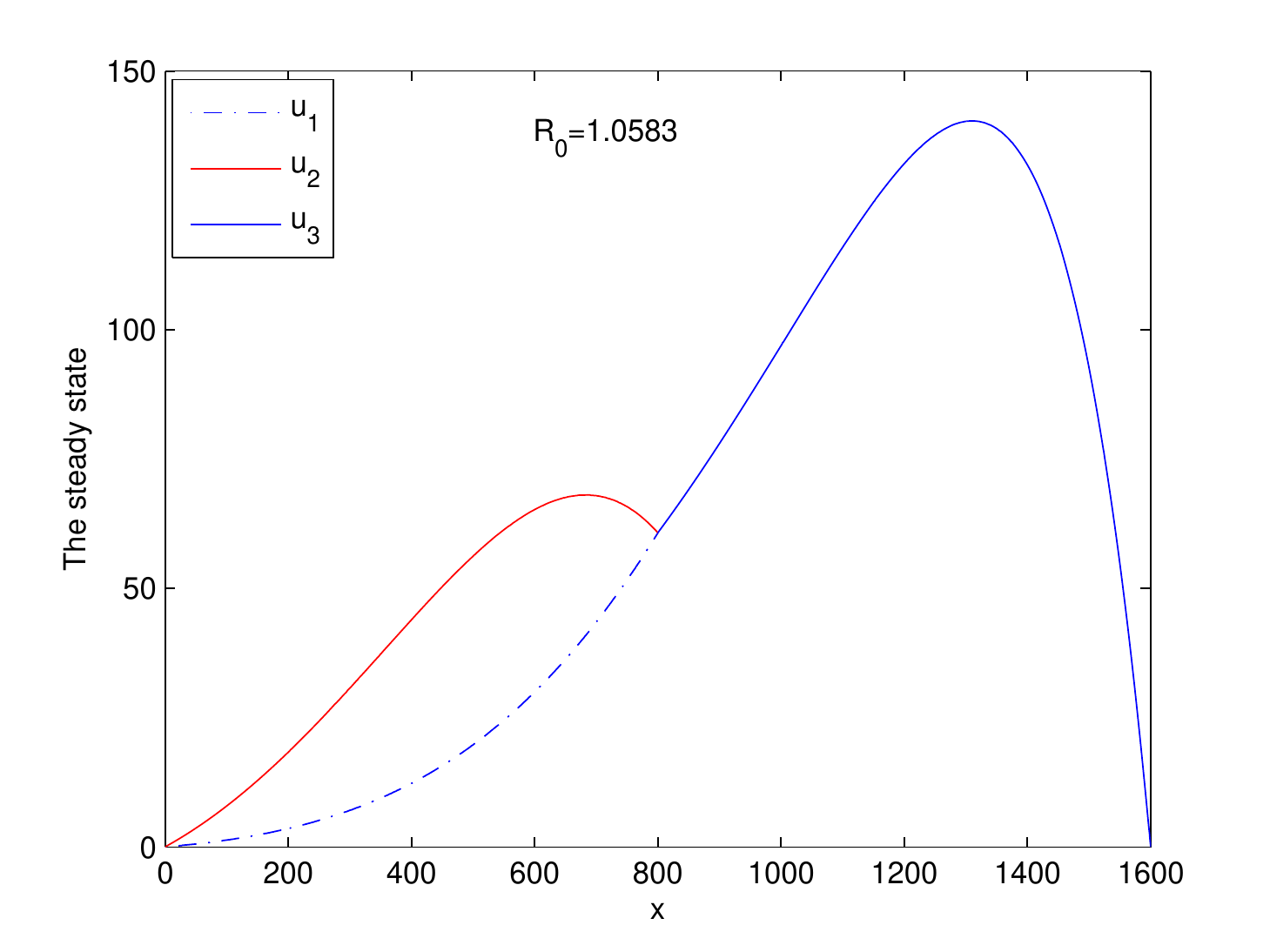}}
  \subfigure[]{ \label{left}
\includegraphics[height=1.55in,
width=1.9in]{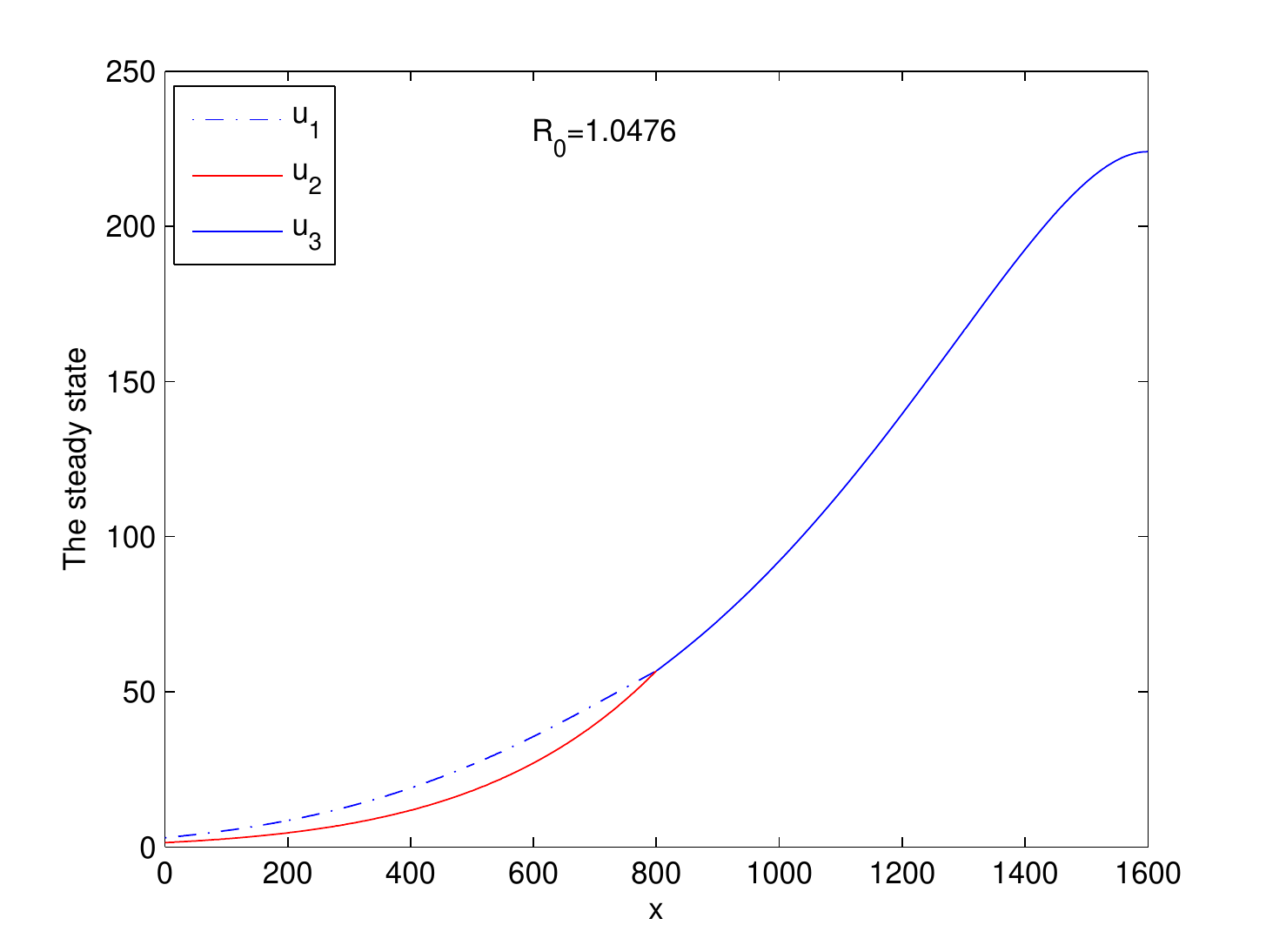}}
\hspace{-0.55cm}
  \subfigure[]{   \label{middle}\includegraphics[height=1.55in, width=1.9in]{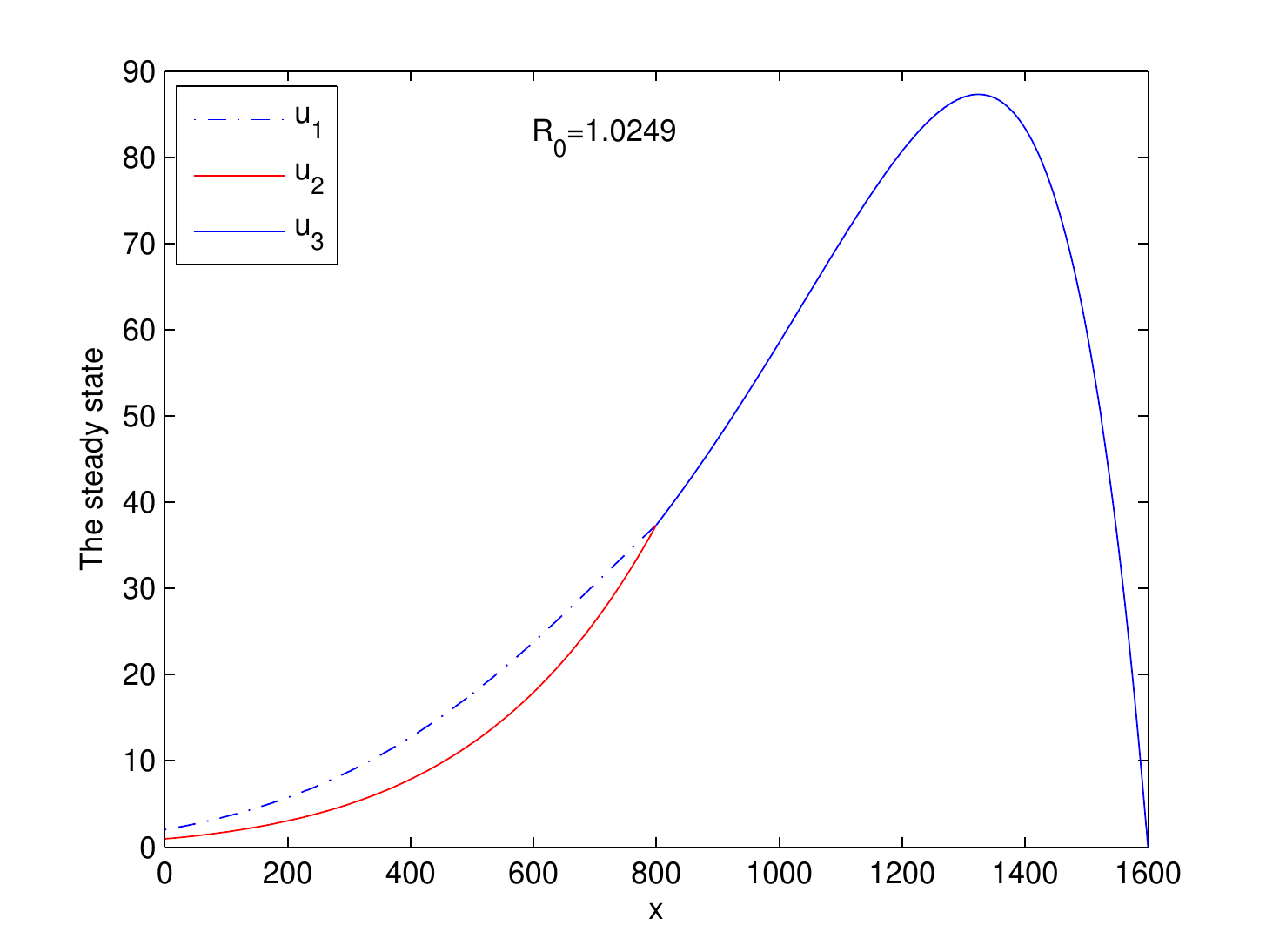}}\hspace{-0.55cm}
\caption{The steady state in network (3-a).
Parameters:
   $m_j=0.06/(24\times 3600)$, $n_j=0.2$, $S_{0j}=0.000001$, $l_j=800$,
 $Q_1=0.05$, $D_1=D_3=0.6$, $f_1=f_3=0.8/(24\times3600)$, $B_1=B_3=20$,
 $B_2=4$; $D_2=1.2D_1$, $f_2=1.2f_1$;
  (a)-(c):  $Q_2=0.1Q_1$; (d)-(f): $Q_2=0.4Q_1$.  Boundary
  conditions: (a) and (d): (ZF-FF) boundary conditions are applied; (b) and (e): (ZF-H) boundary conditions are applied; (c): (H-H) boundary conditions are applied.. Note that
  when $Q_2=0.4Q_1$ and (H-H) boundary conditions are applied,
  $\mathcal{R}_0=0.9301$ and the steady state does not exist.
       } \label{dfqR0ss}
\end{figure}

\subsubsection{The influence of the flow discharge and biological conditions on $\mathcal{R}_0$}

We continue to consider the networks of (3-a) and (4-a), where one or two identical small rivers merge into a large river. We vary the flow discharge $Q_2$, the diffusion rate $D_2$, and the birth rate $f_2$ in the small rivers to see the co-influence of flow conditions and biological conditions on population persistence and extinction. Figure \ref{dfqR0}  shows the regions for population persistence or extinction in the $D_2/D_1$-$f_2/f_1$ plane under different flow conditions. When the conditions in the large river are given, the larger the diffusion rate $D_2$ or the birth rate $f_2$ is, or the smaller discharge $Q_2$ is in the small river branch, the easier it is for the population to persist ($\mathcal{R}_0>1$) in the whole network. Comparing the two panels of Figure \ref{dfqR0}, we see that the parameter region for $\mathcal{R}_0>1$ is smaller in the network (4-a) than that in the network (3-a), under the same flow conditions. This confirms our previous observation that merging more small rivers with the same conditions into a large river leads to stronger threshold conditions for population persistence in the whole river network.

\subsubsection{Good or bad regions in a river network}

To see which part of a river network is good or bad for population persistence, we plot the spatial profile of
the stable positive steady state (when $\mathcal{R}_0>1$) (see Figure \ref{dfqR0ss}) for the network (3-a),
under different hydrological and biological conditions and three sets of boundary conditions. Our simulations show that
when $\mathcal{R}_0$ in the whole network is large,
the positive steady state is large in the upstream branch with better conditions (for persistence) and in the downstream branch. These branches are hence overall good regions for population growth and persistence. When $\mathcal{R}_0$ is less than $1$ or larger than but
close to $1$, 
the positive steady state (if exists) is mainly distributed in the downstream branch, and hence only the downstream branch is the relatively better region in the whole network and the upstream (large and small) branches are considered to be bad regions for population growth or persistence. Figure 
\ref{dfqR0ss} shows an example that the population can persist in the large river
if isolated and the diffusion rate and the birth rate in the small river branch are larger than those in the large
river branches. When the flow discharge in the small branch is much lower than that in the large branch, the resulted
$\mathcal{R}_0$ is large and good regions are distributed in the middle part of the small branch and the downstream branch; when the flow discharge in the small river is not very low, the resulted $\mathcal{R}_0$ is less than $1$ or larger but close to $1$ and the downstream branch is relatively good. If hostile condition is applied at a boundary, then the best region (where
the steady state takes the maximum value) shifts from the region near the boundary to the middle of the branch. In all the cases, it
seems that the downstream branch is at least not the worst region in a merging river network. We can also plot the next generation
distribution in these cases and the distributions are similar to the steady states in Figure \ref{dfqR0ss}.

\section{Discussion}

Real river systems are represented by networks. Previously, integro-differential equations and reaction-diffusion-advection equations were used to describe population dynamics in river networks \cite{Ramirez2012,Sarhad2015,Sarhad2014}, but parameters were assumed to be constants throughout the network and persistence theories in \cite{Sarhad2015,Sarhad2014} were mainly for radially symmetric trees.

We consider reaction-diffusion-advection equations in networks of general trees and establish the fundamental theories for the parabolic and elliptic problems in the networks (such as maximum principle, comparison principle, the existence, uniqueness and estimation of solutions, etc.). We also obtain the existence of the principal eigenvalue of the corresponding eigenvalue problem and establish the theory for population persistence and existence of a stable positive steady state when population persistence.
Furthermore, we define the net reproductive rate $\mathcal{R}_0$, which biologically represents the average number of offsprings that a single individual produces during its lifetime, and prove that the population persists in the whole river network if $\mathcal{R}_0\geq 1$ but will be extinct if $\mathcal{R}_0<1$. The theories were rigorously developed and proved. We convert the calculation of the net reproductive rate $\mathcal{R}_0$ to that of the principal eigenvalue of a generalized eigenvalue problem.  Then we use the
finite difference method to discretize the eigenvalue problem and eventually use the principal eigenvalue of the eigenvalue problem for a matrix to approximate $\mathcal{R}_0$. We then use a vector related to the eigenvector associated with the matrix eigenvalue problem to  approximate the eigenfunction of the next generation operator, which we call the next generation distribution. We also use the finite difference method to discretize the corresponding elliptic equations to obtain the positive steady state when $\mathcal{R}_0>1$.

We conduct numerical simulations in a few simple networks and investigate the influences of hydrological, physical, and biological factors on population persistence in terms of the net reproductive rate
and on the spatial profile of the positive steady state (when exists).
Our results coincide with existing ones in one-dimensional rivers: in a given type of river network, the population persistence becomes easier if the total river length is larger,  the flow discharge/advection is lower, or the diffusion rate or the growth rate is higher. In particular, we compare population dynamics in a smaller network and in a larger network. It turns out that stronger (or better) conditions are required for a population to persist in a larger river network. However, if the population can persist in both networks, then the net reproductive rate is larger in the larger network than in the smaller network. This special case can be considered as a simplification of the real problem  of adding or removing one branch from the network, which may happen when one upstream branch in the dessert suddenly disappears because of drought or when human beings plan to construct a new upstream or downstream branch to a network for some economic or other reasons. Our results answer the following questions: how will such phenomenon or activity influence population dynamics in the network, will the loss of a branch in the dessert cause the extinction of a species, and will a new branch attached to the network help the persistence of a species.

By using the interior connecting condition (\ref{2-interconds}), the differential operator in (\ref{model}) in a tree can be proved to be self-adjoint by rearranging the nodes in the tree and the eigenvalue problem can be written into a Sturm-Liouville problem. Hence the theory for the principal eigenvalue and associated eigenfunction can be established. If the network is not a tree, then we cannot establish the principal eigenvalue with this method and the persistence theory still remains open in this situation. The dynamics of interacting species in river networks is also an interesting future problem.

\section*{Acknowledgments}
\addcontentsline{toc}{section}{Acknowledgments}
Y.J. gratefully acknowledges NSF grant DMS-1411703. R.P. is supported by the NSF of China (No. 11671175, 11571200), the Priority Academic Program Development
of Jiangsu Higher Education Institutions, Top-notch Academic Programs Project of Jiangsu Higher
Education Institutions (No. PPZY2015A013) and Qing Lan Project of Jiangsu Province. J.S. is partially supported by NSF grant DMS-1715651.

\begin{appendices}

\section{Theory of parabolic and elliptic equations on networks}\label{appendixtheories}

\subsection{The strong maximum principle and comparison principle} In this subsection, we establish the strong maximum principle
and comparison principle for parabolic equations on metric graphs, which are fundamental in the investigation of existence,
uniqueness and positivity of solutions to the nonlinear problem {\bf (IBVP)}.

The following strong maximum principle is an analogue of the classical one for open subsets in Euclidian spaces.

 \begin{Lemma}\label{maxiprinciple}
  Assume that $c(x,t)\geq0$ and is bounded from above on $\Omega$. Let $u\in C(\Omega)\cap C^{2,1}(\Omega_p)$ satisfy
 $$
 \frac{\partial u_j}{\partial t}-D_j\frac{\partial^2 u_j}{\partial x_j^2}+v_j\frac{\partial u_j}{\partial x_j}+c_j(x_j,t)u_j\leq 0\ (\geq0),
 \,\, x_j\in (0,l_j),\,\,j\in I_{N-1}, \,\,t\in(0,T],
 $$
and
 $$
 {u}_{i_1}(e_i,t)=\cdots={u}_{i_m}(e_i,t),\ \ \sum_{j=i_1}^{i_m} d_{ij}A_j D_j\frac{\partial u_j}{\partial x_j}(e_i,t)\leq0\ (\geq0),\ \ \forall e_i\in E_r,\ t\in(0,T].
 $$
Suppose that $u\leq M\ (u\geq m)$ on $\Omega$ and $u(x_0,t_0)=M\ (u(x_0,t_0)=m)$ at some point $(x_0,t_0)\in \Omega_p$. If $c(x,t)\not\equiv0$,
suppose that $M\geq0\ (m\leq 0)$. Then
 $$
 u=M\ (u=m)\ \ \mbox{on}\ G\times [0,t_0].
 $$
\end{Lemma}
\begin{proof}

Suppose that $u(x_0,t_0)=M$ at some point $(x_0,t_0)\in \Omega_p$.
We distinguish two cases: (i)\ $x_0\not\in E_r$; (ii)\ $x_0\in E_r$.

In Case (i), clearly $x_0$ is an interior point of some edge $k_j$. The direct application of the strong maximum principle for Euclidean domains
(see, for example, \cite[Theorem 4, Chapter 3]{Protterbook1967}) gives that
 $$
 u_j(x_j,t)=M,\ \ \ \forall (x_j,t)\in[0,l_j]\times[0,t_0].
 $$
Let $k_h$ be an arbitrary edge such that $k_h\cap k_j=\{e_i\}$, where $e_i$ is an end point of $k_j$. If there exists an interior point $y_0$ of $k_h$ such that $u_h(y_0,t_0)=M$, then
 \begin{equation}\label{k_h-1}
 u_h(x_h,t)=M,\ \ \ \forall (x_h,t)\in[0,l_h]\times[0,t_0],
 \end{equation}
due to the strong maximum principle for domains.
If such interior maximum point does not exist, then it is necessary that
 \begin{equation}\label{k_h-2}
 u_h(x_h,t_0)<M,\  \ \ \forall x_h\in(0,l_h).
 \end{equation}
Thus, we can claim that for some $0<\hat t<t_0$, it holds
 \begin{equation}\label{k_h-3}
 u_h(x_h,t)<M,\ \ \ \forall (x_h,t)\in(0,l_h)\times(\hat t,t_0).
 \end{equation}
Suppose that such a claim is false. Then we can find a sequence $\{(\hat x_n, \hat t_n)\}_{n=1}^\infty$
with $\hat x_n\in(0,l_h)$, $\hat t_n<\hat t_{n+1}$ for all $n\geq1$ and $\hat t_n\to t_0$ as $n\to\infty$ such that
$u_h(\hat x_n,\hat t_n)=M$ for all $n\geq1$. By the strong maximum principle for domains again, one has
$u_h(x_h,\hat t_n)=M$ for all $x_h\in[0,l_h]$. Since $\hat t_n\to t_0$ as $n\to\infty$ and $u_h$ is continuous on
$[0,l_h]\times[0,t_0]$, it easily follows that \eqref{k_h-1} holds, which contradicts with \eqref{k_h-2}. Hence, the claim \eqref{k_h-3} is proved.
 In view of \eqref{k_h-3} and the fact that $u_h(x_h,t)$ attains its maximum $M$ at the boundary point $(e_i,t_0)$ of the region $(0,l_h)\times(\hat t,t_0)$,
one then applies the classical Hopf boundary lemma for Euclidean domains (see \cite[Theorem 3, Chapter 3]{Protterbook1967}) to conclude that $d_{hi}\frac{\partial u_h}{\partial x_h}(e_i,t_0)>0$. Recall that $M$ is the maximum value of $u$ on $\Omega$. So for any $j=i_1,\cdots,i_m$ with $j\not=h$  such that $k_j\cap k_h\neq \varnothing$, we have
$d_{ij}\frac{\partial u_j}{\partial x_j}(e_i,t_0)\geq0$. Therefore, it holds
 $$
 \sum_{j=i_1}^{i_m} d_{ij}A_j D_j\frac{\partial u_j}{\partial x_j}(e_i,t_0)>0,
 $$
which is impossible due to our assumption. This contradiction shows that \eqref{k_h-1} must hold. As $G$ is connected, we can assert that $u=M$ in $G\times [0,t_0]$.

We now consider Case (ii). Take $k_j$ to be an arbitrary edge such that $x_0$ is its endpoint and $x_0\in E_r$.
By what was proved in Case (i), we can suppose that $u_j(x_j,t_0)<M,\ \forall x_j\in(0,l_j)$. However, the same argument
as in the last two paragraphs in the proof of Case (i) leads to a contradiction. Thus, $u_j(x_j,t)=M$ for all $(x_j,t)\in[0,l_j]\times[0,t_0]$
and in turn $u=M$ in $G\times [0,t_0]$ by the arbitrariness of $k_j$.
\end{proof}

We remark that \cite[Theorem 1]{vonbelow1991} covers the special case of Lemma \ref{maxiprinciple} that $c(x,t)\equiv0$ and
 $$
 \sum_{j=i_1}^{i_m} d_{ij}A_jD_j\frac{\partial u_j}{\partial x_j}(e_i,t)=0,\, \forall e_i\in E_r,\, t\in(0,T].
 $$
Hence our result is more general and it is useful when dealing with upper-lower solutions.
As a direct application of Lemma \ref{maxiprinciple} as the classical Hopf lemma for Euclidean domains (see, for example, \cite[Theorem 3, Chapter 3]{Protterbook1967}), we have the following Hopf boundary lemma for networks.

\begin{Lemma}\label{signderivative} Assume that $c(x,t)\geq0$ and is bounded from above on $\Omega$. Let $u\in C(\Omega)\cap C^{2,1}(\Omega_p)$ satisfy
 $$
 \frac{\partial u_j}{\partial t}-D_j\frac{\partial^2 u_j}{\partial x_j^2}+v_j\frac{\partial u_j}{\partial x_j}+c(x_j,t)u_j\leq 0\ \, (\geq 0),\ \
 x_j\in(0,l_j),\ j\in I_{N-1},\ t>0.
 $$
Suppose that $u$ is continuously differentiable at some point $(e_i,t)\in E_b\times(0,T]$, $u(e_i,t)=M$, and $u(x,t)<M$ ($>m$) for all $(x,t)\in
\Omega_p$. If $c\not\equiv0$, assume that $M\geq0\ (m\leq0)$. Then $d_{ij}u_{x_j}(e_i,t)>0\ (<0)$.
\end{Lemma}

As a corollary of Lemmas \ref{maxiprinciple} and \ref{signderivative}, we immediately obtain the following comparison principle.

 \begin{Lemma}\label{compprinciple} Assume that $c(x,t)$ is bounded on $\Omega$. Let $u\in C(\Omega)\cap C^{2,1}(\Omega_p)$ satisfy
 \begin{equation}
 \begin{cases}
 \D\frac{\partial u_j}{\partial t}-D_j\frac{\partial^2 {{u}}_j}{\partial x_j^2}+v_j\frac{\partial {{u}}_j}{\partial x_j}+c_j(x_j,t)u_j\geq0,
 \, & x_j\in (0,l_j), \ j\in I_{N-1},\ t\in(0,T),\\
 \displaystyle
 \alpha_{j,1}{u}_j(e_i,t)-\beta_{j,1}\frac{\partial {u}_j}{\partial x_j}(e_i,t)\geq 0, & \forall e_i\in E_u, \ t\in(0,T),\\
 \displaystyle
 \alpha_{j,2}{u}_j(e_i,t)+\beta_{j,2}\frac{\partial {u}_j}{\partial x_j}(e_i,t)\geq 0, & \forall e_i\in E_d,\ t\in(0,T), \\
 \displaystyle
 {u}_{i_1}(e_i,t)=\cdots={u}_{i_m}(e_i,t),\ & \forall e_i\in E_r,\ t\in(0,T),
 \\ \D \sum_{j=i_1}^{i_m} d_{ij}A_jD_j\frac{\partial {u}_j}{\partial x_j}(e_i,t)\geq 0,\ & \forall e_i\in E_r,\ t\in(0,T),\\
 u_j(x_j,0)\geq 0,\ &  x_j\in (0,l_j),\ \ j\in I_{N-1},
 \end{cases}
 \end{equation}
and assume that $\frac{\partial {u}_j}{\partial x_j}(e_i,t)$ exists for $t\in(0,T],\,e_i\in E_b$ if $\beta_{j,s}\not=0$ for some $j\in I_{N-1},\,s\in\{1,2\}$.
Then $u(x,t)\geq 0$ for all $(x,t)\in\Omega$. If $u(x,0)\not\equiv0$, then $u(x,t)>0$ for all $(x,t)\in G\backslash E_0\times(0,T]$.
\end{Lemma}
\begin{proof}
Denote $v(x,t)=e^{-\ell t}u(x,t)$ and take the constant $\ell>0$ to be large so that $\ell+c>0$ on $\Omega$. Elementary computation gives
$v\in C(\Omega)\cap C^{2,1}(\Omega_p)$ satisfy
\begin{equation}
 \begin{cases}
 \D \frac{\partial v_j}{\partial t}-D_j\frac{\partial^2 {{v}}_j}{\partial x_j^2}+v_j\frac{\partial {{v}}_j}{\partial x_j}+[c_j(x_j,t)+\ell]v_j\geq0,
 \, & x_j\in (0,l_j), \ j\in I_{N-1},\ t\in(0,T),\\
 \displaystyle
 \alpha_{j,1}{v}_j(e_i,t)-\beta_{j,1}\frac{\partial {v}_j}{\partial x_j}(e_i,t)\geq 0, & \forall e_i\in E_u, \ t\in(0,T),\\
 \displaystyle
\alpha_{j,2}{v}_j(e_i,t)+\beta_{j,2}\frac{\partial {v}_j}{\partial x_j}(e_i,t)\geq 0, & \forall e_i\in E_d,\ t\in(0,T), \\
 \displaystyle
 v_{i_1}(e_i,t)=\cdots=v_{i_m}(e_i,t),\ & \forall e_i\in E_r,\ t\in(0,T),
 \\ \D \sum_{j=i_1}^{i_m} d_{ij}A_jD_j\frac{\partial {v}_j}{\partial x_j}(e_i,t)\geq 0,\ & \forall e_i\in E_r,\ t\in(0,T),\\
  v_j(x_j,0)\geq 0,\ &  x_j\in (0,l_j),\ \ j\in I_{N-1}.
 \end{cases}
 \end{equation}
Hence, it follows from Lemmas \ref{maxiprinciple} and
\ref{signderivative} that $\min\limits_{\Omega}v(x,t)=m\geq0$ and so
$u(x,t)\geq0$ on $\Omega$. When $u(x,0)\not\equiv0$ (equivalently,
$v(x,0)\not\equiv0$), suppose that $u_j(x_*,t_*)=0$ for some
$(x_*,t_*)\in\Omega_p$. Then $v_j(x_*,t_*)=0=\min_{\Omega}v(x,t)$,
which implies $v\equiv0$ on $\Omega$ by  Lemma \ref{maxiprinciple},
a contradiction. Thus, $u(x,t)>0$ for all $(x,t)\in\Omega_p$.
Additionally, Lemma \ref{signderivative} implies that $u(x,t)>0$
for all $(x,t)\in E_b\backslash E_0\times(0,T]$.
\end{proof}

We remark that  \cite[Theorem 1]{vonbelow1994} states another type of comparison principle for the parabolic problem on graphs.

\subsection{Linear parabolic problem}
In this subsection, we aim to establish the existence, uniqueness, $L^p$ and Schauder estimates of solutions to the following linear parabolic problem:
 \begin{equation}\label{model-a}
 \frac{\partial u_j}{\partial t}=D_j\frac{\partial^2 u_j}{\partial x_j^2}-v_j\frac{\partial u_j}{\partial x_j}+c_j(x_j,t)u_j+g_j(x_j,t),
 \,\, x_j\in (0,l_j), \,\,j\in I_{N-1},\  \ t\in(0,T),
 \end{equation}
associated with the initial condition, boundary and interior connection conditions:
 \begin{equation}\label{model-aa}
 \noindent
 (\ref{modelic}),\, (\ref{bcupstreamzf}),\, (\ref{bcdownstreamzf})\,  \mbox{
 and } (\ref{interconds}),
 \end{equation}
where $T>0$ is a fixed number.

We would like to mention that when the initial data $u^0$ is smooth and satisfies compatibility conditions,
von Below \cite{vonbelow1988} already studied the solvability of (\ref{model-a}).
Here we want to discuss the same issue for the initial data $u^0\in L^p(G)\, (p>1)$
by appealing to the semigroup theory used in \cite{DMugnolo,WArendt,Fijavz}. To the end we need to
make a transformation so that problem (\ref{model-a}) can be written in the form that the framework of \cite{DMugnolo} can apply.

Let
 \begin{equation}\label{pgammadef}
 p_j(x_j)=\eta_j e^{-\frac{v_j}{D_j}x_j},\ \ \quad \zeta_j(x_j)=\frac{p_j(x_j)}{D_j},
 \end{equation}
where $\eta_j$ is a constant to be determined on edge $k_j$. Then (\ref{model-a})  can be written as
 \begin{equation}
 \label{pgammadefnewL}
 \frac{\partial u_j}{\partial t}=\mathcal{A}_j u_j+c_j(x_j,t)u_j+g_j(x_j,t), \,\, x_j\in (0,l_j), \,\,j\in I_{N-1},\  \ t\in(0,T),
 \end{equation}
where
 \begin{equation}\label{pgammadefnewAj}
 \mathcal{A}_j=\frac{1}{\zeta_j(x_j)}\frac{\partial}{\partial x_j}[p_j(x_j)\frac{\partial}{\partial x_j}].
 \end{equation}

Choose one upstream vertex and reorder the vertices and edges such that the chosen vertex is $e_1$,
the edge connecting to $e_1$ is $k_1$, and the endpoint $e_2$ of $k_1$  connects to edges $k_2$, $k_3$, $\cdots$,
and $k_m$. Then at this interior vertex $e_2$, the boundary condition is $\sum\limits_{j=1}^{m}
d_{2j}A_j\left(D_j\frac{\partial u_j}{\partial x_j}\right)(e_2)=0$.
Define $\eta_1=1$ on the edge $k_1$.  Choose suitable $\eta_2$, $\cdots$, $\eta_m$ such that
 $$
 A_1D_1:A_2D_2:\cdots :A_mD_m=p_1(e_2):p_2(e_2):\cdots:p_m(e_2),
 $$
where $p_j(e_2)=p_j(0)$ or $=p_j(l_j)$ depending on whether $e_2$ is
the starting point or ending point of $k_j$. Then the interface
boundary condition at $e_2$ is equivalent to $\sum\limits_{j=1}^{m}
d_{2j}p_j(e_2)\frac{\partial u_j}{\partial x_j} (e_2)=0$. Since $G$
is a tree, we can similarly choose the values for other $\eta_j$'s
and rewrite the interface boundary conditions (\ref{2-interconds})
at interior vertices as
 \begin{equation}\label{pgammadefnewintercon}
 \sum\limits_{j=i_1}^{i_m} d_{ij}p_j(e_i)\frac{\partial u_j}{\partial x_j} (e_i)=0.
 \end{equation}
 Introduce the inner product for functions $\psi,\,\phi\in L^2(G)$ as
 \begin{equation}\label{scalarpdef}
 \langle\psi,\phi\rangle =\sum_{j=1}^{N-1} \int_0^{l_j}
 \zeta_j(x_j) \psi_j{\phi_j} dx_j.
 \end{equation}
Then the differential operator $\mathcal{A}$ with the domain given as in
\cite[Lemma 4.11]{DMugnolo} is self-adjoint with respect to $\langle \cdot,\cdot \rangle$,
and the similar analysis to that of \cite{DMugnolo} shows that $\mathcal{A}$ generates a compact,
contractive and positive strongly continuous semigroup.

Denote
 $$
 g(x,t)=g_j(x_j,t),\,\,\ \ x_j\in (0,l_j), \,\,j\in I_{N-1},\  \ t\in(0,T).
 $$
We now state the following solvability result, and $L^p$ and Schauder estimates for problem (\ref{model-a})-(\ref{model-aa}).
 \begin{Theorem}
 \label{linearexistence}
 Assume that $c\in L^\infty(\Omega)$, $g\in L^p(\Omega)$ for fixed $p>1$.
Then the initial boundary value problem (\ref{model-a})-(\ref{model-aa}) is well-posed on $L^p(G)$, i.e., for
any initial data $u^0\in L^p(G)$, (\ref{model-a})-(\ref{model-aa}) admits a unique strong solution $u\in W^{2,1}_p(\Omega)$
for $t>0$ that continuously depends on the initial data. Moreover, the following estimates hold.
 \begin{enumerate}
 \item[{\rm(i)}] If $u^0\in W^{2}_p(G)$, then the unique solution $u$ satisfies
 $$
 \|u\|_{W^{2,1}_p(\Omega)}\leq C(\|g\|_{L^p(\Omega)}+\|u^0\|_{W^{2}_p(G)})
 $$
for some constant $C>0$ independent of $u,\,u^0,\,g$.

 \item[{\rm(ii)}] If $c\in C^{\alpha,\alpha/2}(\Omega)$, $g\in C^\alpha(\Omega),\, u^0\in C^{2+\alpha}(G)$ for some $\alpha\in(0,1)$ and $u^0$ solves
(\ref{model-a})-(\ref{model-aa}) at $E_b\times\{0\}$ and \eqref{2-interconds} at $E_r\times\{0\}$, and
$[D_j\frac{\partial^2 u^0_j}{\partial x_j^2}-v_j\frac{\partial u^0_j}{\partial x_j}+c_j(\cdot,0)u^0_j+g_j(\cdot,0)](e_i)
=[D_h\frac{\partial^2 u^0_h}{\partial x_h^2}-v_h\frac{\partial u^0_h}{\partial x_h}+c_j(\cdot,0)u^0_h+g_h(\cdot,0)](e_i)$ when $k_j\cap k_h=\{e_i\}$,
then the unique solution $u\in C^{2+\alpha,1+\frac{\alpha}{2}}(\Omega)$ and satisfies
 $$
 \|u\|_{C^{2+\alpha,1+\frac{\alpha}{2}}(\Omega)}\leq C(\|g\|_{C^\alpha(\Omega)}+\|u^0\|_{C^{2+\alpha}(G)})
 $$
for some constant $C>0$ independent of $u,\,u^0,\,g$.
 \end{enumerate}
 \end{Theorem}

\begin{proof}
Note that (\ref{model-a}) with interface condition (\ref{2-interconds}) can be written as (\ref{pgammadefnewL})
with condition (\ref{pgammadefnewintercon}), and the differential operator $\mathcal{A}$ is self-adjoint and generates a compact,
contractive and positive strongly continuous semigroup. Thus, by adjusting the scalar product to (\ref{scalarpdef}) and
defining corresponding functions and operators based on the boundary conditions in (\ref{model-aa}),
the analysis in \cite{DMugnolo} (see also \cite{WArendt,Fijavz})
can be borrowed to show that problem (\ref{model-a})-(\ref{model-aa}) with the initial data
$u^0\in L^p(G)$ has a unique classical solution for $t>0$ that continuously depends on the initial data.

The Schauder estimates in the assertion (ii) have been derived by \cite{vonbelow1988}.
The $L^p$-estimates in the assertion (i) follow similarly as in the proof of the Theorem in \cite{vonbelow1988}.
The proof consists mainly of showing that the initial boundary value problem (\ref{model-a})-(\ref{model-aa}) is equivalent to a
well-stated initial boundary value problem for a parabolic system, where
the $L^p$-estimate results of \cite{LSU,So} for such a parabolic system can be applied. The details are omitted here.
\end{proof}

\subsection{Nonlinear problem {\bf (IBVP)}}
This subsection is devoted to the existence, uniqueness and positivity of solutions to the nonlinear problem {\bf (IBVP)}. Assumptions {\bf [H1]} and {\bf [H2]} are assumed throughout this section.

We begin by introducing the definition of upper and lower solution associated with problem {\bf (IBVP)} as follows.

\begin{Definition}\label{defuppsopara} \begin{enumerate} \item[] A function $\overline{u}\in C^{2,1}(\Omega)$ is an upper solution of problem {\bf (IBVP)}
if $\overline{u}$ satisfies the following conditions
 \begin{equation}\label{modelssupsol}
 \begin{cases}
 \D
\frac{\partial \overline u_j}{\partial t}-D_j\frac{\partial^2
{\overline{u}}_j}{\partial x_j^2}+v_j\frac{\partial
{\overline{u}}_j}{\partial x_j}
 \geq f_j(x_j,{\overline{u}}_j){\overline{u}}_j,  & x_j\in (0,l_j), \ j\in I_{N-1},\ t\in(0,T),\\
 \displaystyle
\alpha_{j,1}\overline{u}_j(e_i,t)-\beta_{j,1}\frac{\partial \overline{u}_j}{\partial x_j}(e_i,t)\geq 0, &  \forall e_i\in E_u, \ t\in(0,T),\\
 \displaystyle
 \alpha_{j,2}\overline{u}_j(e_i,t)+\beta_{j,2}\frac{\partial \overline{u}_j}{\partial x_j}(e_i,t)\geq 0, &\forall e_i\in E_d, \,\ t\in(0,T),
 \\
\overline{u}_{i_1}(e_i,t)=\cdots=\overline{u}_{i_m}(e_i,t),  & \forall e_i\in E_r,\ t\in(0,T),\\
\sum_{j=i_1}^{i_m} d_{ij}A_jD_j\frac{\partial
\overline{u}_j}{\partial x_j}(e_i,t)\geq 0, & \forall e_i\in E_r,\ t\in(0,T),\\
 \overline
u_j(x_j,0)\geq u^0_j(x_j),   &  x_j\in (0,l_j),\ \ j\in I_{N-1}.
 \end{cases}
 \end{equation}

 \item[] A function $\underline{u}(x,t)\in C^{2,1}(\Omega)$ is a lower solution of problem {\bf (IBVP)}
if $\underline{u}$ satisfies the following conditions:
\begin{equation}\label{deflowsopara}
 \begin{cases}
 \D
 \frac{\partial \underline u_j}{\partial t}-D_j\frac{\partial^2 {\underline{u}}_j}{\partial x_j^2}+v_j\frac{\partial {\underline{u}}_j}{\partial x_j}
 \leq f_j(x_j,{\underline{u}}_j){\underline{u}}_j, &x_j\in (0,l_j), \ j\in I_{N-1},\ t\in(0,T), \\
 \displaystyle
\alpha_{j,1}\underline{u}_j(e_i,t)-\beta_{j,1}\frac{\partial \underline{u}_j}{\partial x_j}(e_i,t)\leq 0, &  \forall e_i\in E_u, \ t\in(0,T),\\
 \displaystyle
\alpha_{j,2}\underline{u}_j(e_i,t)+\beta_{j,2}\frac{\partial \underline{u}_j}{\partial x_j}(e_i,t)\leq 0,  & \forall e_i\in E_d,\ t\in(0,T),  \\
 \displaystyle
 \underline{u}_{i_1}(e_i,t)=\cdots=\underline{u}_{i_m}(e_i,t), &\forall e_i\in E_r,\ t\in(0,T),
 \\  \sum_{j=i_1}^{i_m} d_{ij}A_jD_j\frac{\partial \underline{u}_j}{\partial x_j}(e_i,t)\leq 0, & \forall e_i\in E_r,\ t\in(0,T),\\
  \underline u_j(x_j,0)\leq u^0_j(x_j), &x_j\in (0,l_j),\ \ j\in I_{N-1}.
 \end{cases}
 \end{equation}
\end{enumerate}
\end{Definition}

According to the definition of upper and lower solutions, one can easily see from Lemmas \ref{signderivative} and \ref{compprinciple} that

\begin{Lemma}\label{upper-lower-property} Assume that $\overline{u}$ and $\underline{u}$
is a pair of upper and lower solutions of problem {\bf (IBVP)} and $\overline{u}\geq\underline u$ on $G\times\{0\}$.
Then $\overline{u}\geq\underline u$ on $\Omega$. If additionally $\overline{u}\geq,\not\equiv\underline u$ on $G\times\{0\}$,
then $\overline{u}>\underline u$ on $(G\backslash E_0)\times(0,T]$.
\end{Lemma}

With the aid of the above preliminaries, we are now able to state the main result of this subsection.

\begin{Theorem}\label{globalexistence} For any $u^0\in L^p(G)\, (p>1)$ with $u^0\geq0$ a.e. in $G$,
problem {\bf (IBVP)} admits a unique classical solution $u$ for all $t>0$, which satisfies $u\geq0$ in $\Omega$. If additionally, $u^0\not\equiv0$, then
$u(x,t)>0$ for all $t>0$ and $x\in G\setminus E_0$.
\end{Theorem}

\begin{proof}
Note that $0$ and $M^\ast=\max\limits_{j\in I_{N-1}}\{M_j\}$ forms a pair of upper and lower solutions to {\bf (IBVP)}.
Thus, in light of Theorem \ref{linearexistence} and Lemma \ref{upper-lower-property}, the existence of
strong solution follows from the standard iteration of lower and upper solutions; the obtained solution is classical due to Theorem \ref{linearexistence} again.
We omit the details of the proof here and refer interesting readers to \cite{Pao,RDbook}.
The uniqueness and positivity of solutions are obvious consequences of Lemma \ref{upper-lower-property}. The proof is thus complete.
\end{proof}
From now on, given $u^0\in L^p(G)$ for some $p>1$, denote by $u(x,t,u^0)$ the unique solution to problem {\bf (IBVP)}. Clearly, we have

\begin{Lemma}\label{Qtmonotone} For any $\psi_1,\psi_2\in L^p(G)$ with $\psi_1\geq,\not\equiv
\psi_2$ on $G$, $u(x,t,\psi_1)> u(x,t,\psi_2)$ for all $x\in G\setminus E_0$  and $t>0$.
\end{Lemma}

We also have the following observation.

\begin{Lemma}\label{Qtstrsubhomo} If {\bf [H3]} is also satisfied, for any $u^0\in L^p$ with $u^0\geq0$ on $G$ and $\lambda\in (0,1)$,
$u(x,t,\lambda u^0)\geq \lambda u(x,t,u^0)$ on $G$ for all $ t>0$.
 \end{Lemma}

\begin{proof}
For any $\lambda\in (0,1)$, clearly $\lambda u(x,t,u^0)$ satisfies
 $$
\begin{array}{ll}
\frac{\partial \lambda u_j}{\partial t}&= D_j\frac{\partial^2 \lambda u_j}{\partial x_j^2}
-v_j\frac{\partial \lambda u_j}{\partial x_j}+ f_j(x_j,u_j)\lambda u_j, \\
&\leq  D_j\frac{\partial^2 \lambda u_j}{\partial x_j^2}-v_j\frac{\partial \lambda u_j}{\partial x_j}+
f_j(x_j,\lambda u_j)\lambda u_j, \, x_j\in (0,l_j),\,\ j\in I_{N-1},\,t>0,
\end{array}
$$
where we used assumption {\bf [H2]}. It follows from Lemma \ref{upper-lower-property} that $\lambda u(x,t,u^0)\leq u(x,t,\lambda u^0)$ on $G$ for all $ t>0$.
\end{proof}

\subsection{Theory of elliptic equations}

It is clear that Lemmas \ref{maxiprinciple} and \ref{signderivative} imply the following
strong maximum principle for elliptic equations and Hopf type boundary lemma.

\begin{Lemma}\label{maxiprinciple-elliptic} Assume that $c(x)\geq0$ and is bounded from above on $G$. Let $u\in C(G)\cap C^{2}(G\backslash E_b)$ satisfy
 $$
 -D_j\frac{\partial^2 u_j}{\partial x_j^2}+v_j\frac{\partial u_j}{\partial x_j}+c_j(x_j)u_j\leq 0\ (\geq0),
 \,\, x_j\in (0,l_j),\,\,j\in I_{N-1},
 $$
and
 $$
 {u}_{i_1}(e_i)=\cdots={u}_{i_m}(e_i),\ \ \sum_{j=i_1}^{i_m} d_{ij}A_j D_j\frac{\partial u_j}{\partial x_j}(e_i)\leq0\ (\geq0),\ \ \forall e_i\in E_r.
 $$
Suppose that $u\leq M\ (u\geq m)$ on $G$ and $u(x_0)=M\ (u(x_0)=M)$ at some point $x_0\in G\backslash E_b$. If $c(x)\not\equiv0$,
suppose that $M\geq0\ (m\leq 0)$. Then
 $$
 u=M\ (u=m)\ \ \mbox{on}\ G.
 $$
\end{Lemma}

\begin{Lemma}\label{signderivative-elliptic} Assume that $c(x)\geq0$ and is bounded from above on $G$. Let $u\in C(G)\cap C^{2}(G\backslash E_b)$ satisfy
 $$
 -D_j\frac{\partial^2 u_j}{\partial x_j^2}+v_j\frac{\partial u_j}{\partial x_j}+c(x_j)u_j\leq 0\ \, (\geq 0),\ \
 x_j\in(0,l_j),\ j\in I_{N-1}.
 $$
Suppose that $u$ is continuously differentiable at some point $e_i\in E_b$, $u(e_i)=M$, and $u(x)<M$ ($>M$) for all $x\in
G$. If $c\not\equiv0$, assume that $M\geq0\ (M\leq0)$. Then $d_{ij}u_{x_j}(e_i)>0\ (<0)$.
\end{Lemma}

The following comparison principle immediately follows from Lemmas \ref{maxiprinciple-elliptic} and \ref{signderivative-elliptic}.

 \begin{Lemma}\label{positivesolell} Assume that $c(x)\geq0$, is bounded from above on $\Omega$
 and $c_j\beta_{j,i}\not\equiv0$ for some $j\in I_{N-1},\,i\in\{1,2\}$. Let $u\in C(G)\cap C^{2}(G\backslash E_b)$ satisfy
 \begin{equation}
 \begin{cases}
  \displaystyle
 -D_j\frac{\partial^2 {{u}}_j}{\partial x_j^2}+v_j\frac{\partial {{u}}_j}{\partial x_j}+c_j(x_j)u_j\geq0,
 \, & x_j\in (0,l_j), \ j\in I_{N-1},\\
 \displaystyle
 \alpha_{j,1}{u}_j(e_i)-\beta_{j,1}\frac{\partial {u}_j}{\partial x_j}(e_i)\geq 0, &\forall e_i\in E_u,\\
 \displaystyle
 \alpha_{j,2}{u}_j(e_i)+\beta_{j,2}\frac{\partial {u}_j}{\partial x_j}(e_i)\geq 0, &\forall e_i\in E_d, \\
 \displaystyle
 {u}_{i_1}(e_i)=\cdots={u}_{i_m}(e_i),\ \ \sum_{j=i_1}^{i_m} d_{ij}A_jD_j\frac{\partial {u}_j}{\partial x_j}(e_i)\geq 0,\ & \forall e_i\in E_r,
 \end{cases}
 \end{equation}
and assume that $\frac{\partial {u}_j}{\partial x_j}(e_i)$ exists for $e_i\in E_b$ if $\beta_{j,s}\not=0$ for some $j\in I_{N-1},\,s\in\{1,2\}$.
Then $u(x)\geq 0$ for all $x\in G$. If $u(x)\not\equiv0$, then $u(x)>0$ for all $x\in G\backslash E_0$.
\end{Lemma}

In what follows, we will establish the existence, uniqueness, $L^p$ and Schauder estimates of solutions to the following linear elliptic problem:
 \begin{equation}\label{modelsslin}
\begin{cases}
\D -D_j\frac{\partial^2 u_j}{\partial x_j^2}+v_j\frac{\partial u_j}{\partial x_j}+c_j(x_j)u_j=g_j(x_j), \,& x_j\in (0,l_j),\ \ j\in I_{N-1}, \\
 \D\alpha_{j,1}u_j(e_i)-\beta_{j,1}\frac{\partial u_j}{\partial x_j}(e_i)=0, & \forall e_i\in E_u, \\
 \D \alpha_{j,2}u_j(e_i)+\beta_{j,2}\frac{\partial u_j}{\partial x_j}(e_i)=0, & \forall e_i\in E_d, \\
 u_{i_1}(e_i)=\cdots=u_{i_m}(e_i),\ \
 \D\sum_{j=i_1}^{i_m} d_{ij}A_jD_j\frac{\partial u_j}{\partial x_j}(e_i)= 0,\,\, & \forall e_i\in E_r.
 \end{cases}
 \end{equation}
Indeed, by using the similar idea to that of \cite{vonbelow1988}, we can write the boundary value problem on $G$
in (\ref{modelsslin}) into an equivalent boundary value problem for an elliptic system and then obtain the following result about the
existence and a priori estimates of solutions of (\ref{modelsslin}).

\begin{Theorem}\label{modelsslinsolex} The following assertions hold.
 \begin{enumerate}
 \item[{\rm(i)}] Assume that $c$ is bounded on $G$ with $c(x)\geq,\not\equiv 0$ and  $g\in L^p(G)$ ($p>1$),
 then (\ref{modelsslin}) admits a unique strong solution $u\in W_p^2(G)$ and
 $$
 \|u\|_{W^2_p(G)}\leq C\|g\|_{L^p(G)},
 $$
where the constant $C$ does not depend on $u,\,g$.

\item[{\rm(ii)}] Assume that $c\in C^{\alpha}(G)$ with $c(x)\geq,\not\equiv 0$ and $g\in C^{\alpha}(G)$, then (\ref{modelsslin})
admits a unique solution $u\in C^{2+\alpha}(G)$ and
 $$
 \|u\|_{C^{2+\alpha}(G)}\leq C\|g\|_{C^{\alpha}(G)},
 $$
where the constant $C$ does not depend on $u,\,g$.

 \end{enumerate}
\end{Theorem}

Next, we develop the theory of upper and lower solutions to establish the existence and uniqueness of solution
to the following nonlinear elliptic equations
\begin{equation}\label{modelss-a}
 \begin{cases}
 \D
 -D_j\frac{\partial^2 u_j}{\partial x_j^2}+v_j\frac{\partial u_j}{\partial x_j}=g_j(x_j,u_j), &x_j\in (0,l_j),\,j\in I_{N-1},  \\
 \displaystyle
 \alpha_{j,1}u_j(e_i)-\beta_{j,1}\frac{\partial u_j}{\partial x_j}(e_i)= 0,&  \forall e_i\in E_u,\\
 \displaystyle
 \alpha_{j,2}u_j(e_i)+\beta_{j,2}\frac{\partial u_j}{\partial x_j}(e_i)= 0, &\forall e_i\in E_d,
 \\   u_{i_1}(e_i)=\cdots=u_{i_m}(e_i),\,\,\,\D \sum_{j=i_1}^{i_m} d_{ij}A_jD_j\frac{\partial u_j}{\partial x_j}(e_i)= 0,& \forall e_i\in E_r.
 \end{cases}
\end{equation}

\begin{Definition}\label{defupplowsoelli}
\begin{enumerate} \item[] A function $\overline{u}\in C^2(G)$ is an upper solution of (\ref{modelss-a})
if $\overline{u}$ satisfies
 \begin{equation}\label{modelssupsol}
 \begin{cases}
 \D D_j\frac{\partial^2 {\overline{u}}_j}{\partial x_j^2}-v_j\frac{\partial {\overline{u}}_j}{\partial x_j}
 +g_j(x_j,{\overline{u}}_j)\leq 0,
  & x_j\in (0,l_j),\,j\in I_{N-1}, \\
 \displaystyle
\alpha_{j,1}\overline{u}_j(e_i)-\beta_{j,1}\frac{\partial \overline{u}_j}{\partial x_j}(e_i)\geq 0, & \forall e_i\in E_u,\\
 \displaystyle
 \alpha_{j,2}\overline{u}_j(e_i)+\beta_{j,2}\frac{\partial \overline{u}_j}{\partial x_j}(e_i)\geq 0,  &  \forall e_i\in E_d, \\
 \displaystyle
\overline{u}_{i_1}(e_i)=\cdots=\overline{u}_{i_m}(e_i),\,\,\,
\sum_{j=i_1}^{i_m} d_{ij}A_jD_j\frac{\partial
\overline{u}_j}{\partial x_j}(e_i)\geq 0,&\forall e_i\in E_r.
  \end{cases}
 \end{equation}

  \item[] A function $\underline{u}(x,t)\in C^2(G)$ is a lower solution of  (\ref{modelss-a})
if $\underline{u}$ satisfies
 \begin{equation}\label{modelsslowsol}
 \begin{cases}
 \D
  D_j\frac{\partial^2 {\underline{u}}_j}{\partial x_j^2}-v_j\frac{\partial {\underline{u}}_j}{\partial x_j}
 +g_j(x_j,{\underline{u}}_j)\geq  0, &  x_j\in (0,l_j),\,j\in I_{N-1},\\
 \displaystyle
 \alpha_{j,1}\underline{u}_j(e_i)-\beta_{j,1}\frac{\partial \underline{u}_j}{\partial x_j}(e_i)\leq 0, & \forall e_i\in E_u,\\
 \displaystyle
 \alpha_{j,2}\underline{u}_j(e_i)+\beta_{j,2}\frac{\partial \underline{u}_j}{\partial x_j}(e_i)\leq 0,&  \forall e_i\in E_d, \\
 \displaystyle
  \underline{u}_{i_1}(e_i)=\cdots=\underline{u}_{i_m}(e_i),\,\,\,\D \sum_{j=i_1}^{i_m} d_{ij}A_jD_j\frac{\partial \underline{u}_j}{\partial x_j}(e_i)\leq 0, &
 \forall e_i\in E_r.
 \end{cases}
 \end{equation}

\end{enumerate}
\end{Definition}

Based on Lemma \ref{positivesolell} and Theorem \ref{modelsslinsolex}, one can use the standard iteration of lower and upper solutions
(see, for instance \cite{Pao,RDbook}) to conclude that

\begin{Theorem}\label{ellipticeqnsolexis}
Let $\overline{u}$ and $\underline{u}$ be a pair of upper and lower solutions of (\ref{modelss-a}) satisfying
$\overline{u}\geq \underline{u}$ on $G$ and $m=\min_G\underline u<M=\max_G\overline u$, and
 $$
 |g_j(x_j,u_j)-g_j(y_j,v_j)|\leq K(|x_j-y_j|^{\alpha}+|u_j-v_j|),\, \forall (x_j,u_j),\,(y_j,v_j)\in G\times[m,M],\, j\in I_{N-1}
 $$
for some constants $K>0$ and $\alpha\in(0,1)$. Then (\ref{modelss-a}) admits
a solution $u\in C^{2+\alpha}(G)$ which satisfies $\underline{u}\leq u\leq
\overline{u}$. Moreover, (\ref{modelss-a}) admits a minimal $\tilde{w}$ and
a maximal solution $\tilde{u}$ in $[\underline{u},\overline{u}]$ in the sense that
for any solution $w$ of  (\ref{modelss-a}) satisfying $\underline{u}\leq w\leq \overline{u}$, we have $\tilde{w}\leq
w\leq \tilde{u}$.
\end{Theorem}

\section{Proof of Proposition \ref{eigenvalueexth}}\label{proofpropeigenvalueexth}

Choose $\xi>0 $  large enough so that $f_j(\cdot,0)-\xi<0$ for all $j\in I_{N-1}$.
For any $g\in X$, Theorem \ref{modelsslinsolex} guarantees that the problem
 \begin{equation}\label{modelsslintt}
 \begin{cases}
 \D
 -D_j\frac{\partial^2u_j }{\partial x_j^2}+v_j\frac{\partial u_j}{\partial x_j}+[\xi-f_j(\cdot,0)]u_j=g_j(x_j), & x_j\in (0,l_j),\,j\in I_{N-1},\\
 \displaystyle
\alpha_{j,1}u_j(e_i)-\beta_{j,1}\frac{\partial u_j}{\partial x_j}(e_i)= 0,  & \forall e_i\in E_u,\\
 \displaystyle
  \alpha_{j,2}u_j(e_i)+\beta_{j,2}\frac{\partial u_j}{\partial x_j}(e_i)= 0,&  \forall e_i\in E_d,\\
 \displaystyle
 u_{i_1}(e_i)=\cdots=u_{i_m}(e_i),\,\,
 \sum_{j=i_1}^{i_m} d_{ij}A_jD_j\frac{\partial u_j}{\partial x_j}(e_i)= 0, &\forall e_i\in
 E_r
  \end{cases}
 \end{equation}
has a unique solution $u$ satisfying
 $$
 \|u\|_{C^{2+\alpha}(G)}\leq C\|g\|_{C^\alpha(G)}\leq C_1\|g\|_{C^1(G)}
 $$
for some constants $C>0$ independent of $u$ and $g$.

Define the operator
 \begin{equation}\label{Tdefinition}
 T: X\rightarrow X,\ \ \,\,u=Tg.
 \end{equation}
Then $T$ is a linear and continuous operator that maps a bounded set in $X$ into a bounded set in
$C^{2+\alpha}(G)$. Note that a bounded set in $C^{2+\alpha}(G)$ is a sequentially compact set in $X$. This implies that $T$ maps
a bounded set in $X$ into a sequentially compact set in $X$.  Hence, $T$ is a compact operator on
$X$. Moreover, by Lemmas \ref{signderivative} and \ref{compprinciple}, $Tg\geq 0$ if $g\in X_+$, and $u=Tg\in X^o$.
Therefore, $T$ is strongly positive. Let $r(T)$ be the spectral radius of $T$. It follows from the well-known Krein-Rutman Theorem
(see, for example, \cite[Theorem 1.2]{Dubook}) that $r(T)>0$ is a simple eigenvalue of $T$ with an
eigenfunction $g^\ast\in X^0$, i.e., $Tg^\ast=r(T) g^\ast$, and there is no other eigenvalue of $T$ associated
with positive eigenfunctions. Thus, $\psi^\ast=Tg^\ast$ satisfies
$-\mathcal{L}\psi^\ast+\xi \psi^\ast=(1/r(T))\psi^\ast$ in $G$, and hence,
$\lambda^\ast=\xi-1/r(T)$ is a simple eigenvalue of (\ref{modellineigen10})
with positive eigenfunction $\psi^\ast\in X^o$ and no other eigenvalues of
(\ref{modellineigen10}) correspond to positive eigenfunctions.
Similarly as in the proof of \cite[Theorem 1.4]{Dubook}, we can obtain that if $\lambda\neq \lambda^\ast$
is an eigenvalue of (\ref{modellineigen10}), then $Re(\lambda)\leq \lambda^\ast.$

\section{Proof of Theorem \ref{persistenceresult}}\label{proofThpersistenceresult}

We first prove (i). Assume that $\lambda^\ast<0$. Let $\phi\in X_+$. Since $\psi^\ast\in X^o$, clearly
there exists $\sigma>0$ such that $0\leq \phi\leq \sigma \psi^\ast$
on $G$. Let $u(\cdot, t,\phi)$ be the solution of {\bf (IBVP)} with
initial condition $\phi$ and $\overline{u}(\cdot,t,\sigma
\psi^\ast)=e^{\lambda^\ast t} \sigma \psi^\ast $ be the solution of
(\ref{modellin}) with initial condition $\sigma \psi^\ast $. By
{\bf [H1]}, we have
 $$
 \frac{\partial \overline{u}_j}{\partial
 t}=D_j\frac{\partial^2 \overline{u}_j}{\partial x_j^2}-v_j\frac{\partial
 \overline{u}_j}{\partial x_j}+f_j(x_j,0)\overline{u}_j\geq D_j\frac{\partial^2
 \overline{u}_j}{\partial x_j^2}-v_j\frac{\partial \overline{u}_j}{\partial
 x_j}+f_j(x_j,\overline{u}_j)\overline{u}_j
 $$
for any $x_j\in (0,l_j)$ and $t>0$. Then
 $$
 \frac{\partial u_j}{\partial t}-[D_j\frac{\partial^2
 u_j}{\partial x_j^2}-v_j\frac{\partial u_j}{\partial
 x_j}+f_j(x_j,u_j)u_j]=0\leq \frac{\partial \overline{u}_j}{\partial
 t}-[D_j\frac{\partial^2 \overline{u}_j}{\partial x_j^2}-v_j\frac{\partial
 \overline{u}_j}{\partial x_j}+f_j(x_j,\overline{u}_j)\overline{u}_j]
 $$
on $(0,l_j)$ for any $t>0$. It follows from Lemma \ref{compprinciple} and the fact
$0\leq \phi\leq \sigma \psi^\ast$ that
 $$
 0\leq u(\cdot,t,\phi)\leq\overline{u}(\cdot,t,\sigma \psi^\ast)
 $$
for any $t\geq 0$. Therefore,
 $$
 0\leq  \lim\limits_{t\rightarrow\infty}u(\cdot,
 t,\phi)\leq
 \lim\limits_{t\rightarrow\infty}\overline{u}(\cdot,t,\sigma
 \psi^\ast)\to0\ \ \mbox{uniformly on} \ G,
 $$
which implies that $u\equiv 0$ is globally attractive for all initial conditions in $X_+$.

Let $\lambda^\ast=0$.
For any  $\phi\in X^o$, there exists some $\sigma_0>0$
such that   $0\leq \phi\leq \overline{u}:=\sigma_0\psi^\ast$ and that $\overline{u}$ is an  upper solution of (\ref{modelss}).
Let $u^{(1)}$ and $u^{(2)}$ be solutions of {\bf (IBVP)} with initial conditions
$\phi$ and $\overline{u}$, respectively. It follows from Theorem
\ref{compprinciple} that $0\leq u^{(1)}(x,t)\leq
u^{(2)}(x,t)\leq \overline{u}(x)$. Note that  $u^{(2)}$ is bounded and monotonically decreasing in $t$
by following the same proof of \cite[Lemma 3.2.4]{RDbook}). Therefore, the similar proof of \cite[Lemma
3.2.5]{RDbook}) shows that $\lim\limits_{t\rightarrow \infty}u^{(2)}(\cdot,t)=V\geq 0$ and $V$ is  a classical solution of  (\ref{modelss}).
If $V(x)>0$ at some $x\in G$, clearly $V\in X^o$. Then, by Proposition \ref{equivalent-property},
we can easily obtain $\lambda^\ast>0$, which gives rise to a contradiction. Hence, $V(x)\equiv 0$ on $G$.
Therefore, $\lim\limits_{t\rightarrow \infty}u^{(1)}(x,t)=0$. So we conclude that when $\lambda^\ast=0$, $u\equiv 0$ is \
globally attractive for {\bf (IBVP)} with respect to all initial conditions in $X_+$. Thus, (i) is verified.

We next prove (ii).  Assume that $\lambda^\ast>0$. Let $\psi^\ast$ be the eigenfunction associated
with $\lambda^\ast$. For sufficiently small $\epsilon>0$, we have $f_j(x_j,\epsilon )
\geq f_j(x_j,0)-\lambda^\ast$ for all $x_j\in(0,l_j)$, $j\in\{1,\cdots, N-1\}$. This implies that  $w_1=\epsilon \psi^\ast$ is a  lower
solution of  (\ref{modelss}). Note that for any constant $K^\ast>\max\{M_1,\cdots,M_{N-1}\}$, $w_2(x)=K^\ast$ is an
upper solution of (\ref{modelss}). Let $\epsilon>0$ be sufficiently small such that
$w_2(x)\geq w_1(x)=\epsilon \psi^\ast(x)$ and $K^\ast> \min_{x\in G}\{\epsilon \psi^\ast(x)\}$.
It follows from Theorem \ref{ellipticeqnsolexis} that (\ref{modelss}) admits a positive solution $u^\ast\in X^o$.

Assume there are two distinct positive steady states of {\bf (IBVP)}: $u_1^\ast$ and $u_2^\ast$. Since $K^*$ can be arbitrarily large
and $\epsilon$ can be arbitrarily small, without loss of generality, assume that $u_1^\ast$ is the maximal solution of  (\ref{modelss})
and $u_2^\ast$ is the minimal solution of  (\ref{modelss})
in $[\epsilon \psi^\ast, K^\ast]$. Then $u_2^\ast\leq u_1^\ast$ on $G$. It suffices to show $u_2^\ast=u_1^\ast$. Suppose that
$u_2^\ast\leq,\not= u_1^\ast$ on $G$. Recall that $u_1^\ast,\,u_2^\ast\in X^o$. By defining
$\tau^0=\sup\{\tau>0:\ u_2^\ast\geq\tau u_1^\ast\ \mbox{on}\ G\}$, we then have $\tau^0\in(0,1)$ and $u_2^\ast\geq\tau^0 u_1^\ast$ on $G$.  Due to {\bf [H3]}, we further observe that $u_2^\ast\geq,\not\equiv\tau^0 u_1^\ast$.
Thus, by Lemmas \ref{signderivative}, \ref{Qtmonotone} and \ref{Qtstrsubhomo}, we get
 $$
 u_2^\ast=Q_t(u_2^\ast)\gg Q_t(\tau^0 u_1^\ast)\geq \tau^0 Q_t(u_1^\ast)=\tau^0 u_1^\ast, \,\forall t>0,
 $$
where $Q_t$ is the solution map of {\bf (IBVP)} defined as $Q_t(\psi)=u(x,t,\psi)$ for the solution
$u(x,t,\psi)$  of {\bf (IBVP)} with initial condition $\psi$.
This implies that $u_2^\ast(x)-\tau^0u_1^\ast(x)\in X^o$, which in turn implies that
$u_2^\ast(x)-\tau^0u_1^\ast(x)\geq \tau_0u_1^\ast(x)$ on $G$ for some small $\tau_0>0$.
This is a contradiction to the definition of $\tau^0$. Therefore,
there is only a unique positive steady state $u^\ast$ of {\bf (IBVP)}.

For any $u^0\in X_+\setminus \{0\}$, obviously the unique solution $u$ of {\bf (IBVP)} satisfies that
$u(\cdot,t)\in X^o$ for any $t>0$. Thus, we can assume that $u^o\in X^o$. So there exist some $\epsilon_0>0$ and $\sigma_0\geq 1$
such that $\underline{u}=\epsilon_0 \psi^\ast$ and $\overline{u}=\sigma_0 K^\ast$
are lower and upper solutions of (\ref{modelss}), respectively, and
 $$
 \underline{u}=\epsilon_0 \psi^\ast\leq u^o\leq\sigma_0 K^\ast=\overline{u}\ \  \mbox{on } G.
 $$
Let $u_1,\, u_2$ be solutions of {\bf (IBVP)} with initial conditions
$\underline{u}$ and $\overline{u}$, respectively. It follows from Lemma
\ref{upper-lower-property} that $\underline{u}(x)\leq u_1(x,t)\leq
u(x,t)\leq u_2(x,t)\leq \overline{u}(x)$. As before, it can be easily proved that $u_1$
and $u_2$ are monotonically increasing and decreasing in t,
respectively (\cite[Lemma 3.2.4]{RDbook}). Therefore, $u_1 $
and $u_2$ are bounded and monotonic with respect to $t$. Additionally, we can claim
$\lim\limits_{t\rightarrow \infty}u_1(\cdot,t)=U$ and
$\lim\limits_{t\rightarrow \infty}u_2(\cdot,t)=V$. Furthermore, we
can prove $U$ and $V$ are solutions of  (\ref{modelss}) (see Lemma
3.2.5 in \cite{RDbook}). Then $\underline{u}(x)\leq U(x)\leq
V(x)\leq \overline{u}(x)$. Hence, $U(x)=V(x)=u^\ast$, and
$\lim\limits_{t\rightarrow \infty}u(x,t)=u^\ast(x)$. Therefore, we
have proved that $u^\ast$ is globally attractive with respect to any
positive initial values in $X_+\setminus \{0\}$.

\section{The hydrological relation in a gradually varying flow} \label{appflow}

 Recall that the governing equation for the gradually varied flow is given by \begin{equation} \label{eq:dydx0}
\frac{d y}{dx} = \frac{S_0(x)-S_f(y)}{1-F_r^2(y)}
\end{equation}
(see (5-7) in \cite{Chaudhrybook}), where $x$ (unit:~m) represents the longitudinal location along the river,  $y(x)$ (unit:~m) is the water depth at location~$x$, $S_0(x)$ is the slope of the channel bed at location~$x$, $S_f$ is the friction slope, i.e., the slope of the energy grade line, $F_r$ is the Froude number that is defined as the ratio between the flow velocity and the water wave propagation velocity.
In the case where the river has a rectangular cross section with a constant width~$B$ (unit:~m) and a constant bed slope~$S_{0}$, the water depth $y(x)$ is stabilized at the normal depth
\begin{equation}\label{normaldepthq}
y_n=\left(\frac{Q^2 n^2}{B^2 S_{0} k^2}\right)^{\frac{3}{10}},
\end{equation}
where $Q$ (unit: m$^3$/s) is the flow discharge,  $k=1$ is a dimensionless conversion factor, and $n$ (unit: s/m$^{1/3}$) is Manning's roughness coefficient, which represents the resistance to water flows in channels and depends on factors such as the bed roughness and sinuosity. The flow in such a river is called a uniform flow. See more details in \cite{Chaudhrybook} or Appendices C and D in \cite{JinLewisBMB}.

\end{appendices}

\bibliographystyle{plain}
\bibliography{network-20180420submission-rr.bbl}

\end{document}